\documentclass[10pt]{article}
\usepackage{delarray} 
\usepackage{amscd}
\usepackage{amssymb}
\usepackage{amsmath}
\usepackage{amsthm,amsxtra}
\usepackage{latexsym,amsmath,mathrsfs}
\usepackage[bookmarks,colorlinks]{hyperref}
\usepackage[all]{xy}
\usepackage{amsfonts}
\usepackage{color}
\usepackage{fancybox}
\usepackage{colordvi}
\usepackage{multicol}
\usepackage{textcomp}
\usepackage{stmaryrd}
\usepackage[dvips]{graphicx}
\usepackage{enumerate,xspace}
\usepackage{geometry}
\usepackage{tocbibind}

\def\cyc{\mathrm{cyc}}
\def\ord{\mathrm{ord}}
\def\GL{\mathrm{GL}}

\def\BA{\mathbf{A}}
\def\BC{\mathbf{C}}
\def\BQ{\mathbf{Q}}
\def\BR{\mathbf{R}}
\def\BZ{\mathbf{Z}}

\def\CA{\mathcal{A}}
\def\CD{\mathcal{D}}

\def\CW{\mathcal{W}}
\def\ord{\mathrm{ord}}
\def\log{\mathrm{log}}
\def\LOG{\mathrm{LOG}}
\def\Dist{\mathrm{Dist}}
\def\Tei{\mathrm{Tei}}

    \theoremstyle{plain}

    \newtheorem{thm}{Theorem}[section] \newtheorem{cor}[thm]{Corollary}
    \newtheorem{lem}[thm]{Lemma} 
    \newtheorem{prop}[thm]{Proposition}
    
    \newtheorem{Thm}[thm]{Theorem}

    \theoremstyle{definition}

    \newtheorem{defn}[thm]{Definition}

    \theoremstyle{remark}
    \newtheorem {rem}[thm]{Remark}

\numberwithin{equation}{section}

\newcommand{\wvec}[4]{{\scriptsize{\big [ \!\!
\begin{array}{cc} #1 \!\!\! & \!\!\! #2 \\ #3 \!\!\! & \!\!\! #4 \end{array} \!\! \big ] }}}

\newcommand{\binc}[2]{{\scriptsize{\big ( \!\!
\begin{array}{c} #1   \\ #2   \end{array} \!\! \big ) }}}

\begin{document}

\title{Anticyclotomic exceptional zero phenomenon for Hilbert modular forms}
\author{Bingyong Xie \footnote{This paper is supported by
the National Natural Science Foundation of China (grant 11671137),
and supported in part by Science and Technology Commission of
Shanghai Municipality (no. 18dz2271000).}
 \\ \small Department of Mathematics, East China Normal University, Shanghai, China \\ \small byxie@math.ecnu.edu.cn} \date{}
\maketitle


\begin{abstract} In this paper we study the exceptional zero
phenomenon for Hilbert modular forms in the anticyclotomic setting.
We prove a formula expressing the leading term of the $p$-adic
$L$-functions via arithmetic $L$-invariants. 
\end{abstract}

\section{Introduction}

\subsection{Exceptional zero phenomenon for cyclotomic $\BZ_p$-extensions}

Exceptional zero phenomenon was firstly discovered by Mazur, Tate
and Teitelbaum in their remarkable paper \cite{MMT}. They considered
elliptic modular forms .

For a modular form $f=\sum a_n q^n$ of weight $2$, of level $N$ and
character $\varepsilon$, the attached classical $L$-function
$L(s,f)$ admits an analytic continuous to the entire complex plane,
and satisfies a functional equation which relates its values at $s$
and $k-s$. The famous Birch and Swinnerton-Dyer conjecture relates
the behaviour of $L(s,f)$ at $s=1$ to the arithmetic of $A_\phi$,
the abelian variety associated to $\phi$ by Eichler-Shimura
correspondence.

A $p$-adic analogue of the Birch and Swinnerton-Dyer conjecture is
formulated in \cite{MMT}, which replaces $L(s,f)$ by a $p$-adic
$L$-function $L_p(s,f)$ attached to the cyclotomic $\BZ_p$-extension
$\BQ_{\cyc}$ of $\BQ$. Let $\alpha$ be the root of
$X^2-a_pX+\varepsilon(p)p$ that is a unit. Then there exists a
continuous function $L_p(f, \chi)$ on the set of continuous $p$-adic
characters $\chi$ of $\mathrm{Gal}(\BQ^{\cyc}_\infty/\BQ)$, such
that when $\chi$ is of finite order, we have the interpolation
formula \begin{equation} \label{eq:interp} L_p(f,\chi) = e_p(\chi)
\frac{L(1,f,\chi)}{\Omega_f} \end{equation} where $\Omega_f$ is a
real period making $\frac{L(1,f,\chi)}{\Omega_f}$ algebraic, and
$e_p(\chi)$ is the $p$-adic multiplier.
Letting $\chi$ varying in $\{ \langle \ \cdot \ \rangle^s: s\in
\BC_p, |s|\leq 1\}$ where $\langle \ \cdot \ \rangle$ is the
Teichm\"uller character, we obtain the analytic function
$L_p(s,f):=L_p(f, \langle\ \cdot\ \rangle^s).$

When $p||N$ and $a_p=1$, $e_p(1)=0$ and thus by the interpolation
formula $L_p(s,f)$ vanishes at $0$. The exceptional zero conjecture
made in \cite{MMT} relates the first derivative of $L_p(s,f)$ at $0$
to the algebraic part of $L(f,1)$:
\begin{equation} \label{eq:1-order}
L'_p(0,f) = \mathcal{L}(f) \frac{L(1,f)}{\Omega_f} ,
\end{equation} where $\mathcal{L}(f)$, is the so called
$L$-invariant, an isogeny invariant of $A_f$. When $A_f$ is an
elliptic curve (i.e. all of $a_n$ are rational),
$$\mathcal{L}(f)=\frac{\log_p(q(A_f))}{\ord_p(q(A_f))}$$ is defined by
using the $p$-adic period $q(A_f)$ of $A_f$; in this case it is
showed in \cite{BDGG} that $\mathcal{L}(f)\neq 0$.

This conjecture was proved by Greenberg and Stevens \cite{GS} using
Hida's theory of $p$-adic families of ordinary eigenforms and the
$2$-variable $p$-adic $L$-functions attached to them.

In \cite{Hida} Hida formulated a generalization of this exceptional
zero conjecture. To state it we let $F$ be a totally real field, and
let $\mathbf{f}$ be a Hilbert modular form of parallel weight $2$
over $F$. Corresponding to $\mathbf{f}$ we also have the classical
$L$-function $L(s,\mathbf{f})$ and the $p$-adic $L$-function
$L_p(s,\mathbf{f})$ attached to the cyclotomic $\BZ_p$-extension
$F\BQ_\infty^{\mathrm{cyc}}$ of $F$; $L_p(s,\mathbf{f})$ satisfies
an interpolation formula similar to (\ref{eq:interp}) with the
$p$-adic multiplier replaced by a product
$\prod\limits_{\mathfrak{p}|p}e_\mathfrak{p}(\cdot)$.

Let $J_p$ be the set of primes of $F$ above $p$, $J_1$ the subset of
prime $\mathfrak{p}|p$ such that $\alpha_\mathfrak{p}=1$, i.e.
$\mathfrak{p}$ exactly divided the level of $\mathbf{f}$, and the
Fourier coefficient $a_\mathfrak{p}$ at $\mathfrak{p}$ is $1$. Put
$J_2=J_p\backslash J_1$. Then $e_\mathfrak{p}(1)=0$ if and only if
$\mathfrak{p}\in J_1$.

Hida \cite{Hida} conjectured that
\begin{equation}\label{eq:order}
\ord_{s=0}L_p(s,\mathbf{f}) \geq r:= \sharp (J_1)
\end{equation} and
\begin{equation}\label{eq:coef} \frac{d^rL_p(s,\mathbf{f})}{ds^r}\big|_{s=0} = r! \prod_{\mathfrak{p}\in
J_1} \mathcal{L}_\mathfrak{p}(\mathbf{f}) \cdot
\prod_{\mathfrak{p}\in J_2}e_\mathfrak{p}(1)\cdot
\frac{L(1,\mathbf{f})}{\Omega_\mathbf{f}}
\end{equation} where $\mathcal{L}_\mathfrak{p}(\mathbf{f})$ is the
$L$-invariant of $\mathbf{f}$ at $\mathfrak{p}$. When $\mathbf{f}$
is attached to an elliptic curve $A$ over $F$ that is split
multiplicative at $\mathfrak{p}$ with Tate period
$q({A/F_\mathfrak{p}})$, one has
$$\mathcal{L}_\mathfrak{p}(\mathbf{f}) =
\frac{\log_p(N_{F_\mathfrak{p}/\BQ_p}q(A/F_\mathfrak{p}))}{\ord_p(N_{F_\mathfrak{p}/\BQ_p}q(A/F_\mathfrak{p}))}.
$$

When $r=1$, Mok \cite{Mok} proved (\ref{eq:order}) and
(\ref{eq:coef}) under some local assumption by extending the method
in \cite{GS}. In general, Spiess \cite{Spiess} proved
(\ref{eq:order})  unconditionally and  (\ref{eq:coef}) under some
local assumption.

\subsection{Exceptional zero phenomenon for anticyclotomic
$\BZ_p$-extensions}

To explain what an anticyclotomic $\BZ_p$-extension is, let
$\mathcal{K}$ be an imaginary quadratic field. There are many
different $\BZ_p$-extensions of $\mathcal{K}$. Among them exactly
two have the property that $\mathcal{K}_\infty/\BQ$ is Galois. One
is the cyclotomic $\BZ_p$-extension $\mathcal{K}
\BQ^{\mathrm{cyc}}_\infty$. The other is called the {\it
anticyclotomic $\BZ_p$-extension} of $\mathcal{K}$ and denoted by
$\mathcal{K}^{-}_\infty$; $\mathcal{K}^{-}_\infty$ is characterized
by $\mathcal{K}^{-}_\infty=\bigcup\limits_n \mathcal{K}^{-}_n$,
where $\mathcal{K}^-_n$ is Galois over $\BQ$ and
$\mathrm{Gal}(\mathcal{K}^-_n/\BQ)$ is isomorphic to the dihedral
group of order $2p^n$.

Bertolini and Darmon \cite{BD99} proved that, when $p$ is split in
$\mathcal{K}$, exceptional zero phenomenon again holds for elliptic
modular forms. Now, one considers the complex $L$ function
$L(s,f/\mathcal{K})$ associated to the base change of $f$ to
$\mathcal{K}$; the $p$-adic $L$-adic function is replaced by the
anticyclotomic one, denoted by
$L_p(s,f/\mathcal{K})=L_p(f/\mathcal{K}, \langle\ \cdot\
\rangle_{\mathrm{anti}}^s)$, that interpolates $L(1,f/\mathcal{K},
\chi)$ with $\chi$ being anticyclotomic characters, i.e.\ characters
of $\mathrm{Gal}(\mathcal{K}^{-}_\infty/\mathcal{K})$, instead of
cyclotomic characters, i.e. characters of
$\mathrm{Gal}(\mathcal{K}\BQ^{\mathrm{cyc}}_\infty/\mathcal{K})$.
Here $\langle\ \cdot\ \rangle_{\mathrm{anti}}$ is the special
anticyclotomic character that plays the same role in the
anticyclotomic setting as  the Teichm\"uller character $\langle \
\cdot \ \rangle$ in the cyclotomic setting. See Section
\ref{ss:anti-ext-char} for its definition. The function
$L_p(s,f/\mathcal{K})$ was constructed by Bertolini and Darmon in
\cite{BD96}, and was showed to satisfy the interpolation formula by
Chida and Hsieh in \cite{CH}.

Their formula says
\begin{equation}\label{eq:2-order}
\frac{d^2 L_p(s,f/\mathcal{K})}{d s^2} \big|_{s=0} =
\mathcal{L}^\Tei(f)^2 \cdot \frac{L(1,f/\mathcal{K})}{\Omega^-_f} ,
\end{equation} where $\Omega^-_f$ is certain complex period making
$\frac{L(1,f/\mathcal{K}, \chi)}{\Omega^-_f}$ algebraic, and
$\mathcal{L}^\Tei(f)$ is the $L$-invariant of Teitelbaum type. To
define $\mathcal{L}^\Tei(f)$ Bertolini and Darmon \cite{BD99} fixed
a branch of the $p$-adic logarithmic.

The reader may notice that it is the first derivative appeared in
(\ref{eq:1-order}), while it is the second derivative appeared in
(\ref{eq:2-order}). It is rather mysterious to the author this
difference between the cyclotomic and the anticyclotomic settings.

The purpose of our paper is to study the exceptional zero phenomenon
for Hilbert modular forms in the anticyclotomic setting.

We keep Bertolini and Darmon's assumption that $p$ is split in
$\mathcal{K}$. Let $F$ be a totally real field, $K=\mathcal{K}F$.
Let $\mathbf{f}$ be a Hilbert modular form of parallel weight $2$
over $F$. Let $\pi$ denote the automorphic representation of
$\mathrm{GL}_2(\BA_F)$ associated to $\mathbf{f}$, $\pi_K$ the base
change of $\pi$ to $K$.

Our first task is to construct the anticyclotomic $p$-adic
$L$-function $L_p(\mathbf{f}/K,\chi)$ that interpolates the special
$L$-values of $L(1,\mathbf{f}/K,\cdot)=L(\frac{1}{2},\pi_K,\cdot)$.
This function is defined on the set
\begin{equation} \label{eq:cont-char}
\left\{\ \chi: p\text{-adic continuous characters of }
\mathrm{Gal}(\mathcal{K}^-_\infty F/ F) \text{ unramified outside
}p\ \right\}.
\end{equation}

We need some level conditions on $\mathbf{f}$. We decompose the
level $\mathfrak{n}$ of $\mathbf{f}$ as
$$
\mathfrak{n}
=\mathfrak{n}^+\mathfrak{n}^-\prod_{\mathfrak{p}|p}\mathfrak{p}^{r_\mathfrak{p}},$$
where $\mathfrak{n}^+$ (resp. $\mathfrak{n}^-$) is only divided by
primes that are split (resp. inert or ramified) in $K$, and
$(\mathfrak{n}^+\mathfrak{n}^-,p)=1$. We assume that

$\bullet$ $r_\mathfrak{p}\leq 1$ for each $\mathfrak{p}|p$;

$\bullet$ $\mathfrak{n}^-$ is square-free;

$\bullet$ the cardinal number of prime factors of $\mathfrak{n}^{-}$
that are inert in $K$ has the same parity as $[F:\BQ]$.

\begin{Thm}\label{thm:A} Suppose that $\mathbf{f}$ satisfies the above level conditions and that $\mathbf{f}$ is ordinary at $p$ in the sense that
$\alpha_\mathfrak{p}$ is a $p$-adic unit for each prime
$\mathfrak{p}$ above $p$.
 \begin{enumerate}\item There exists a nonzero complex number $\Omega^-_{\mathbf{f}}$ and
a $p$-adic continuous function $L_p(\mathbf{f}/K, \chi)$ on the set
$(\ref{eq:cont-char})$ such that when $\chi$ is of finite order,
$\frac{L(1,\mathbf{f}/K, \chi )}{\Omega^-_\mathbf{f}}$ is algebraic
and
$$ L_p(\mathbf{f}/K,\chi) = \prod_{\mathfrak{p}|p} e_\mathfrak{p}(\chi)  \cdot  \frac{L(1,\mathbf{f}/K, \chi)}{\Omega^-_\mathbf{f}}.
$$ Here $e_\mathfrak{p}(\chi)$ is the multiplier
$$ e_\mathfrak{p} (\chi) =\left\{\begin{array}{ll}
\frac{\alpha_\mathfrak{p}^2(1-\alpha_\mathfrak{p}\chi(\mathfrak{P}))(1-\alpha_\mathfrak{p}\chi(\bar{\mathfrak{P}}))}{p^{2[F_\mathfrak{p}:\BQ_p]}}
& \text{ if } \ord_\mathfrak{P}\chi=0 \text{ i.e. } \chi \text{ is
unramified at } \mathfrak{P},
\\  \frac{p^{n[F_\mathfrak{p}:\BQ_p]}}{\alpha_\mathfrak{p}^{2n}} & \text{ if } \ord_\mathfrak{P}\chi=n >0 , \end{array}\right.
$$ where $\mathfrak{p}=\mathfrak{P}\bar{\mathfrak{P}}$.
\item When $\chi$ varies in the analytic family $\{\langle \ \cdot \
\rangle_{\mathrm{anti}}^s: s\in \BC_p , |s|_p <\frac{1}{p^2}\}$, we
obtain an analytic function
$$ L_p(s,\mathbf{f}/K): = L_p(\mathbf{f}/K, \langle \ \cdot \
\rangle_{\mathrm{anti}}^s). $$ Here, by abuse of notation the
restriction of $\langle \ \cdot \ \rangle_{\mathrm{anti}}$ to
$\mathrm{Gal}(\mathcal{K}^-_\infty F/ F)$ is again denoted by
$\langle \ \cdot \ \rangle_{\mathrm{anti}}$.
\end{enumerate}
\end{Thm}

To construct the $p$-adic $L$-functions demanded in Theorem
\ref{thm:A}, we embed the anticyclotomic extension
$\mathcal{K}^-_\infty F/F$ into a much larger one. For each set of
places $J$ of $F$ above $p$, the union of the ring class fields of
conductor $\prod_{\mathfrak{p}\in J}\mathfrak{p}^{n_\mathfrak{p}}$
($n_\mathfrak{p}\geq 0$ for each $\mathfrak{p}\in J$) contains a
maximal subfield whose Galois group is a free module $\BZ_p$ of rank
$\sharp(\Sigma_J)$, which we call the $\Gamma^-_J$-extension. See
our context for the meaning of the notation $\Sigma_J$. When
$J=J_p$, the full set of places above $p$, the
$\Gamma^-_{J_p}$-extension contains $\mathcal{K}^-_\infty F$.

When $J=\{\mathfrak{p}\}$ is single, Hung \cite{Hung14} extended
Chida and Hsieh's method \cite{CH} to construct a $p$-adic
$L$-function for the $\Gamma^-_\mathfrak{p}$-extension. In this
paper we extend their method to the case of arbitrary $J$. For each
$J$ we obtain a multi-variable $p$-adic $L$-function denoted by
$L_J$ that interpolating the special values
$L_{\mathrm{alg}}(\frac{1}{2}, \pi_K, \chi)$ with $\chi$ factors
through $\Gamma^-_J$. When $J=J_p$, we refer $L_{J_p}$ as the
total-variable $p$-adic $L$-function. The $p$-adic $L$-function in
Theorem \ref{thm:A} is obtained by restricting $L_{J_p}$.

Let $J_1$ be the set of primes $\mathfrak{p}$ above $p$ such that
$\alpha_\mathfrak{p}=1$, and put $J_2=J_p\backslash J_1$. By
definition $e_\mathfrak{p}(1)=0$ if and only if $\mathfrak{p}\in
J_1$. Hence, if $J_1$ is non-empty, then $L_p(0,\mathbf{f}/K)=0$.

Comparing Spiess' theorem and Bertolini-Darmon's, one expects the
order of $L_p(s,\mathbf{f}/K)$ at $0$ is at least $2\:\sharp (J_1)
$. Our second task is to show that this is indeed the case.

\begin{Thm}\label{thm:B} Suppose that $\mathbf{f}$ satisfies the same conditions as in Theorem
$\ref{thm:A}$.
\begin{enumerate}
\item Put $r=\sharp (J_1) $. Then
$$\ord_{s=0}L_p(s,\mathbf{f}/K) \geq 2r.$$
\item
We have
$$ \frac{d^{2r}L_p(s,\mathbf{f}/K)}{ds^{2r}}\big|_{s=0} = (r!)^2 \cdot \prod_{\mathfrak{p}\in
J_1} \mathcal{L}^\Tei_\mathfrak{p}(\mathbf{f})^2 \cdot
\prod_{\mathfrak{p}\in J_2}e_\mathfrak{p}(1)\cdot
\frac{L(1,\mathbf{f}/K)}{\Omega^-_\mathbf{f}}.
$$
\end{enumerate}
\end{Thm}

In the above formula $\mathcal{L}^\Tei_\mathfrak{p}(\mathbf{f})$ is
the $L$-invariant of Teitelbaum type of $\mathbf{f}$ at
$\mathfrak{p}$, which is defined in \cite{CMP}. The branch of the
$p$-adic logarithmic $\log$ we choose to define
$\mathcal{L}^\Tei_\mathfrak{p}(\mathbf{f})$ is the same as Bertolini
and Darmon's, and thus is independent of $\mathfrak{p}$ and
$\mathbf{f}$. When $f$ is attached to an elliptic curve $A$, and
$q(A/F_\mathfrak{p})$ is the Tate period, then \begin{equation}
\mathcal{L}^\Tei_\mathfrak{p}(\mathbf{f}) = \frac{\log
(N_{F_\mathfrak{p}/\BQ_p}q(A/F_\mathfrak{p}))}{\ord_p
(N_{F_\mathfrak{p}/\BQ_p}q(A/F_\mathfrak{p}))}.
\end{equation}

We obtain a result beyond Theorem \ref{thm:B}. Indeed, we can
determine all partial derivatives of order $2r$ of the
total-variable $p$-adic $L$-function $L_{J_p}$.

\begin{thm}\label{thm:C} We have the following Taylor expansion of $L_{J_p}$ at
$(0,\cdots, 0)$:
\begin{eqnarray*} L_{J_p}((s_\sigma)_{\sigma\in \Sigma_{J_p}},\pi_K) &=& \frac{L(1,\mathbf{f}/K)}{\Omega^-_\mathbf{f}} \cdot \prod_{\mathfrak{p}\in J_2}e_\mathfrak{p}(1)
\cdot \prod\limits_{\mathfrak{p}\in J_1 }\left(
\sum\limits_{\sigma\in
 \Sigma_{J_p} } s_\sigma
 \mathcal{L}_{\mathfrak{p},\sigma} \right )^2
  +
 \text{ higher order terms}.  \end{eqnarray*} For
the definition of $\mathcal{L}_{\mathfrak{p},\sigma}$ see Definition
\ref{defn:L-inv} in our context. When
$\sigma\in\Sigma_\mathfrak{p}$, $\mathcal{L}_{\mathfrak{p},\sigma}$
is L-invariant of Teitelbaum type; when
$\sigma\notin\Sigma_\mathfrak{p}$,
$\mathcal{L}_{\mathfrak{p},\sigma}$ is a constant independent of
$\pi_\mathbf{f}$.
\end{thm}

We also obtain exceptional zero phenomenon between those $L_{J}$.

\begin{thm}\label{thm:D}
If $\widehat{\chi}$ is a character of $\Gamma^-_J$ such that
$\widehat{\chi}_\mathfrak{p}=1$ for each $\mathfrak{p}\in J_1\cap
J$, then
\begin{equation}\label{eq:thm-D}
L_J(({s}_\sigma)_{\sigma\in\Sigma_J},\pi_K,\widehat{\chi}) =
L_{J\backslash (J_1\cap J)} (\pi_K, \widehat{\chi})\cdot
\prod\limits_{\mathfrak{p}\in J_1\cap J }\left(
\sum\limits_{\sigma\in
 \Sigma_J } s_\sigma
 \mathcal{L}_{\mathfrak{p},\sigma} \right )^2
  +
 \text{ higher order terms}.  \end{equation}
\end{thm} Note that the condition ``$\widehat{\chi}_\mathfrak{p}=1$ for each $\mathfrak{p}\in J_1\cap
J$'' implies that $\widehat{\chi}$ comes from a character of
$\Gamma^-_{J_1\cap J}$ and so $L_{J\backslash (J_1\cap J)} (\pi_K,
\widehat{\chi})$ makes sense.

In the special case that $J=\{\mathfrak{p}\}\subset J_1$ and
$J\backslash (J_1\cap J)=\emptyset$, Theorem \ref{thm:D} was
obtained by Hung \cite[Theorem A and Theorem B]{Hung18} following
Bertolini-Darmon's method. But Hung needed the condition on class
number that $p\nmid \sharp(\widehat{F}^\times K^\times \backslash
\widehat{K}^\times /\widehat{\mathcal{O}}_K^\times)$. This condition
is removed in Theorem \ref{thm:D}.

A result related to Theorem \ref{thm:D} was obtained by Bergunde and
Gehrmann \cite[Theorem 6.5]{BG18}. They proved a formula for leading
terms of anticyclotomic stickelberger elements. Using their result,
one may deduce a formula similar to (\ref{eq:thm-D}) in the case of
$\widehat{\chi}=1$ but less precise. Indeed, the automorphic periods
$q_{S_\mathrm{m},\mathfrak{p}}$ in \cite{BG18} are not showed to
coincide with Tate's period; these periods
$q_{S_\mathrm{m},\mathfrak{p}}$ are even not showed to be
independent of $S_\mathfrak{m}$. While in our formula
(\ref{eq:thm-D}) the leading term is precise.

\subsection{Strategy}

We give a sketch of the proof of Theorem \ref{thm:B} and Theorem
\ref{thm:C}. The argument for Theorem \ref{thm:D} is similar.

The function $L_{J_p}$ is equal to $\mathscr{L}_{J_p}^2$, the square
of another analytic function $\mathscr{L}_{J_p}$. We restrict
$\mathscr{L}_{J_p}$ to one direction.  Let
$(s_\sigma)_{\sigma\in\Sigma_{J_p}}$ be fixed, and put
$$ \mathscr{L}_{J_p}(t; (s_\sigma)):=\mathscr{L}_{J_p}(ts_\sigma)_{\sigma\in \Sigma_{J_p}}.
$$  We will show that
\begin{equation} \label{eq:D}
\frac{d^{n}\mathscr{L}_{J_p}(t; (s_\sigma))}{dt^{n}}\Big|_{t=0} =
\left\{
\begin{array}{cl}
r!\cdot \prod\limits_{\mathfrak{p}\in J_1 }\left(
\sum\limits_{\sigma\in
 \Sigma_{J_p} } s_\sigma
 \mathcal{L}_{\mathfrak{p},\sigma} \right ) \cdot \mathscr{L}_{J_p\backslash J_1}(0,\cdots,0) & \text{ if } n=r, \\
 0 & \text{ if } n<r. \end{array} \right. \end{equation}
Taking $(s_\sigma)=(1,\cdots, 1)$ we obtain Theorem \ref{thm:B}.
Letting $(s_\sigma)$ vary we obtain Theorem \ref{thm:C}.

Our approach to (\ref{eq:D}) is somewhat a mixture of Spiess' and
Bertolini-Darmon's.

In \cite{Spiess} Spiess related the computation of
$\frac{d^r}{ds^r}L_p(s,f)\big{|}_{s=0}$ to group cohomology. Let
$F^\times_+$ denote the group of totally positive elements of $F$.
Each $\mathfrak{p}\in J_1=\{\mathfrak{p}_1,\cdots, \mathfrak{p}_r\}$
is associated with an element $c_{\ell_\mathfrak{p}}\in
H^1(F^\times_+, C_c(F_\mathfrak{p},\BC_p))$. Then
$\frac{d^r}{ds^r}L_p(s,f)\big{|}_{s=0}$ was showed to equal the
cap-product \begin{equation} \label{eq:spiess}
\frac{d^r}{ds^r}L_p(s,f)\Big{|}_{s=0} =
(-1)^{\binc{r}{2}}r!(\kappa\cup c_{\ell_{\mathfrak{p}_1}}\cup \cdots
\cup c_{\ell_{\mathfrak{p}_r}})\cap \vartheta. \end{equation} Here,
$\kappa\in H^{[F:\BQ]-1}(F^\times_+,\mathscr{D})$ is in the
$([F:\BQ]-1)$-th cohomology of $F^\times_+$ with values in a
distribution space, and $\vartheta$ is the fundamental class of the
quotient $M/F^\times_+$ of some $(r-1+[F:\BQ])$-dimensional real
manifold $M$ with a free action of $F^\times_+$.

Instead of $F^\times_+$, the group $\Delta$ we use is a subgroup of
$F^\times\backslash K^\times$, which is a free abelian group of rank
$r$.

We reduce the computation of $\frac{d^{n}\mathscr{L}_{J_p}(t;
(s_\sigma))}{dt^{2r}}\big|_{t=0}$ to
$$ \int l^n \widetilde{\mu}_{J_p} $$
the integration of $l^n$ on some domain of $\BA^\infty_F\backslash
\BA^\infty_K$ for certain $F^\times\backslash K^\times$-invariant
distribution $\widetilde{\mu}_{J_p}$, where $l$ is a logarithm
function depending on the direction $(s_\sigma)_{\Sigma_{J_p}}$. In
the case of $n<r$, we show directly that the integration is zero. In
the case of $n=r$, we relate the integration of $l^{r}$ to group
cohomology. Being invariant by $\Delta$, $\widetilde{\mu}_{J_p}$
provides an element in the $0$-th cohomology group of $\Delta$ with
values in some space of distributions.

To each $\mathfrak{p}\in J_1=\{\mathfrak{p}_1,\cdots,
\mathfrak{p}_r\}$ and the logarithmic function $l$ we attach an
element $[c_{\mathfrak{p},l}]$ in the first cohomology of $\Delta$
with values in certain function space
$\mathcal{C}_\mathfrak{p}^{\leq 1}$. Taking cup-product we obtain an
element $[c_{\mathfrak{p}_1, l}]\cup \cdots \cup
[c_{\mathfrak{p}_r,l}]$ in the $r$-th cohomology of $\Delta$ with
values in some function space $\mathcal{C}_{J_1}^{\leq 1}$. Then we
have an element $$[\widetilde{\mu}_{J_p}]\cup([c_{\mathfrak{p}_1,
l}]\cup \cdots \cup [c_{\mathfrak{p}_r,l}]) \in H^r(\Delta, \BC_p)$$
whose evaluation at the $r$ generators of $\Delta$ is closely
related to $\int l^n \widetilde{\mu}_{J_p}$.

Similarly, to $\mathfrak{p}\in J_1$ and the order function $l$ we
attach an element $[c_{\mathfrak{p},\ord}]\in H^1(\Delta,
\mathcal{C}_\mathfrak{p}^{\leq 1})$. Taking product we obtain
$$[\widetilde{\mu}_{J_p}]\cup([c_{\mathfrak{p}_1, \ord}]\cup \cdots
\cup [c_{\mathfrak{p}_r,\ord}]) \in H^r(\Delta, \BC_p) $$ whose
evaluation at the the $r$ generators of $\Delta$ is closely related
to $\mathscr{L}_{J_p\backslash J_1}(0,\cdots,0)$.

We adjust Bertolini-Darmon's method \cite{BD99} to show that
\begin{equation}\label{eq:L-prod}
[\widetilde{\mu}_{J_p}]\cup([c_{\mathfrak{p}_1, l}]\cup \cdots \cup
[c_{\mathfrak{p}_r,l}]) = \prod\limits_{\mathfrak{p}\in J_1 }\left(
\sum\limits_{\sigma\in
 \Sigma_{J_p} } s_\sigma
 \mathcal{L}_{\mathfrak{p},\sigma} \right ) \cdot
[\widetilde{\mu}_{J_p}]\cup([c_{\mathfrak{p}_1, \ord}]\cup \cdots
\cup [c_{\mathfrak{p}_r,\ord}]) \end{equation} in
$H^r(\Delta,\BC_p)$.

We compare our method and Spiess' \cite{Spiess}. We use the $r$-th
cohomology group, while Spiess used the $(r-1+[F:\BQ])$-th
cohomology group. In \cite{Spiess} Spiess needed the element
$\vartheta$ in the $([F:\BQ]-1)$-th cohomology group that comes from
the real world. In our approach all things are $p$-adic, so that we
could avoid the comparison between the $p$-adic world and the real
world. However, our construction of the function spaces
$\mathcal{C}^{\leq 1}_\mathfrak{p}$ is more technical than that of
$C_c(F_\mathfrak{p},\BC_p)$ used by Spiess.

The $p$-adic $L$-functions $L_J$ we studied are of multi-variable,
while the $p$-adic $L$-function studied in \cite{Spiess} is of
one-variable. Such a discrepancy occurs, because the maximal abelian
$p$-extension of $K$ and the associated Galois group are rather
large, while according to Leopoldt conjecture (which was proved to
be true if $F/\BQ$ is abelian) the maximal abelian $p$-extension of
$F$ is the cyclotomic extension $F\BQ_\infty^{\mathrm{cyc}}$
 and the Galois group is just $\BZ_p$.

However, considering the Hida families one also has multi-variable
$p$-adic $L$-functions in the cyclotomic setting. The author
believes that it is possible to obtain results parallel to Theorem
\ref{thm:C} and Theorem \ref{thm:D} for these multi-variable
$p$-adic $L$-functions by extending Spiess' method.

Our paper is organized as follows. In Section \ref{sec:char-measure}
we recall some basic facts on Haar measures and construct families
of $p$-adic anticyclotomic characters. In Section \ref{sec:auto-rep}
we recall needed facts on automorphic forms. We extend Chida and
Hsieh's method to construct $p$-adic $L$-functions $L_J$ in Section
\ref{sec:special-value} and Section \ref{sec:p-L-function}. In
Section \ref{sec:har-coh} we define Harmonic cocycle valued and
cohomological valued modular forms, and use them as tools to define
$L$-invariant of Teitelbaum type. In Section \ref{sec:meas-dist} we
compute some integrations closely related to the partial derivatives
of $L_J$. In Section \ref{sec:L-inv-gh} we define the elements
$[c_{\mathfrak{p},l}]$ and $[c_{\mathfrak{p},\ord}]$ in group
cohomology of $\Delta$, and then we prove (\ref{eq:L-prod}). Finally
we prove Theorems \ref{thm:B}, \ref{thm:C} and \ref{thm:D} in
Section \ref{sec:proof-main}.

\subsection{Notations}


Let $\BA$ be the ring of adeles of $\BQ$.

For any prime number $\ell$ we use $|\cdot|$ to denote the absolute value on $\BC_\ell$ such that $|\ell|=1/\ell$.
For a finite extension $L$  of $\BQ_\ell$ in $\BC_\ell$, and a uniformizing element $\omega_L$ of $L$,
we use $|\cdot|_{\omega_L}$ to denote the absolute value on $\BC_\ell$ such that $|a|_{\omega_L}=|N_{L/\BQ_\ell}(a)|$ for $a\in L$.
So $|\cdot|_{\omega_L}=|\cdot|^{[L:\BQ_\ell]}$. If $q_L$ is the cardinal number of the residue field of $L$, then $|\omega_L|_{\omega_L}=q_L^{-1}$.

Let $p$ be an odd prime number.

\section{Characters and measures} \label{sec:char-measure}

\subsection{$p$-adic characters and Hecke characters}

We fix a totally real number field $F$. Let $\Sigma_F$ be the set of
all real embeddings of $F$, $J_p$ the set of places of $F$ above
$p$.

Let $K$ be a totally imaginary extension of $F$. We fix a CM type
$\Sigma_K$ of $K$, i.e. for $\sigma\in\Sigma$ there exists exactly
one element of $\Sigma_K$ (an embedding of $K$ into $\BC$) that
restricts to $\sigma$.  By abuse of notation we again use $\sigma$
to denote this element, and use $\bar{\sigma}$ to denote the other
embedding $K\hookrightarrow\BC$ that restricts to
$\sigma:F\hookrightarrow\BR$. We fix an isomorphism of fields
$\jmath: \BC\xrightarrow{\sim} \BC_p$.

Let $\Omega_p$ be the maximal ideal of $\mathcal{O}_{\BC_p}$. For
each $\sigma\in \Sigma_F$ or $\Sigma_K$ we put
$$\mathfrak{p}_\sigma=\sigma(F)\cap \jmath^{-1}\Omega_p,\ \
\mathfrak{P}_\sigma=\sigma(K)\cap \jmath^{-1}\Omega_p \ \ \text{and}
\ \ \bar{\mathfrak{P}}_\sigma=\bar{\sigma}(K)\cap
\jmath^{-1}\Omega_p.$$ Note that, for a prime $\mathfrak{p}$ of $F$
above $p$, the set of $\sigma$ such that
$\mathfrak{p}_\sigma=\mathfrak{p}$ has cardinal number
$[F_\mathfrak{p}:\BQ_p]$.

For our convenience, we demand $\Sigma_K$ satisfies the following
condition:
\begin{quote}
If $\mathfrak{p}_{\sigma_1}=\mathfrak{p}_{\sigma_2}$, then
$\mathfrak{P}_{\sigma_1}=\mathfrak{P}_{\sigma_2}$. \end{quote} Thus,
when $\mathfrak{p}$ is split in $K$, we may write
$\mathfrak{p}\mathcal{O}_K=\mathfrak{P}\overline{\mathfrak{P}}$,
where $\mathfrak{P}=\mathfrak{P}_\sigma$ for each $\sigma$
satisfying $\mathfrak{p}_\sigma=\mathfrak{p}$. We will use this
convention throughout this paper without mention.

The inclusion $\jmath \sigma$ extends to $K_{\mathfrak{P}_\sigma}$
which will be denoted by $\sigma':K_{\mathfrak{P}_\sigma}
\rightarrow \BC_p$. For $\bar{\sigma}$ instead of $\sigma$, we have
$\bar{\sigma}': K_{\bar{\mathfrak{P}}_\sigma} \rightarrow \BC_p$. We
write $F_{\sigma'}$, $K_{\sigma'}$ and $K_{\bar{\sigma}'}$ for
$\sigma'(F_{\mathfrak{p}_\sigma})$,
$\sigma'(K_{\mathfrak{P}_\sigma})$ and
$\bar{\sigma}'(K_{\bar{\mathfrak{P}}_\sigma})$ respectively.

Let $\widehat{\nu}: \mathrm{Gal}(\overline{K}/K)\rightarrow
\mathcal{O}_{\BC_p}^\times$ be a $p$-adic continuous character. Via
the geometrically normalized reciprocity map
$$\mathrm{rec}_K: K^\times K_\infty^\times\backslash \BA_K^\times
\xrightarrow{\sim} \mathrm{Gal}(\overline{K}/K)$$  we may view
$\widehat{\nu}$ as a character of $\BA_K^\times$ that is trivial on
$K^\times K_\infty^\times$.

Fix $\mathbf{m}=\sum_{\sigma\in \Sigma_F} m_\sigma \sigma$ with
$m_\sigma$ nonnegative. We say that $\widehat{\nu}$ is of type
$(\mathbf{m}, -\mathbf{m})$, if the following holds.

In the case when $\mathfrak{p}$ is inert or ramified in $K$, i.e.
$\mathfrak{p}\mathcal{O}_K=\mathfrak{P}$ or $\mathfrak{P}^2$, we
have
$$\widehat{\nu}_{\mathfrak{P}}(a)=\prod_{\sigma\in\Sigma_K:\mathfrak{p}_\sigma=\mathfrak{p}}
\sigma'\Big(\frac{a}{\bar{a}}\Big)^{m_\sigma}$$ on a compact open
subgroup of $\mathcal{O}_{K_\mathfrak{P}}^\times$.

In the case when $\mathfrak{p}$ is split in $K$, i.e.
$\mathfrak{p}\mathcal{O}_K=\mathfrak{P}\bar{\mathfrak{P}}$ and
$F_\mathfrak{p}\cong K_\mathfrak{P}\cong K_{\bar{\mathfrak{P}}}$, we
have
$$\widehat{\nu}_{\mathfrak{P}}(a)=\prod_{\sigma\in \Sigma_K:\mathfrak{P}_\sigma=\mathfrak{P}}
\sigma(a)^{m_\sigma} $$ and
$$\widehat{\nu}_{\bar{\mathfrak{P}}}(a)=\prod_{\sigma\in \Sigma_K:\mathfrak{P}_\sigma=\mathfrak{P}}
\sigma(a)^{-m_\sigma} $$ on a compact open subgroup of
$\mathcal{O}_{F_\mathfrak{p}}^\times$.

If $\widehat{\nu}$ is of type $(\mathbf{m},-\mathbf{m})$, we attach
to $\widehat{\nu}$ a complex Hecke character $\nu$ of
$K^\times\backslash \BA _K^\times$ by
$$\nu  (a)= \jmath^{-1} \Big(\widehat{\nu}(a) \cdot \prod_{\sigma\in \Sigma_K}
\Big(\frac{
\bar{\sigma}'a_{\bar{\mathfrak{P}}_\sigma}}{\sigma'a_{\mathfrak{P}_\sigma}}\Big)^{m_\sigma}
 \Big) \cdot
\prod_{\sigma \in \Sigma_K}
\Big(\frac{a_\sigma}{\bar{a}_\sigma}\Big)^{m_\sigma}  .
$$ Then $\nu$ is of archimedean type $(\mathbf{m},-\mathbf{m})$, i.e. $\nu_\infty=\prod\limits_{\sigma \in \Sigma_K}
\Big(\frac{a_\sigma}{\bar{a}_\sigma}\Big)^{m_\sigma}$.

Any complex Hecke character of $K^\times\backslash \BA _K^\times$ of
archimedean type $(\mathbf{m},-\mathbf{m})$ comes from a $p$-adic
character of type $(\mathbf{m},-\mathbf{m})$ in the above way.

\subsection{Anticyclotomic extensions and anticyclotomic
characters} \label{ss:anti-ext-char}

Fix a subset $J$ of $J_p$. Put $\Sigma_J:=\{\sigma\in \Sigma_F:
\mathfrak{p}_\sigma\in J\}.$ For $J=\{\mathfrak{p}\}$ we write
$\Sigma_\mathfrak{p}$ for $\Sigma_{\{\mathfrak{p}\}}$.

For each $\mathfrak{p}\in J$ we write
$K_\mathfrak{p}=K\otimes_FF_\mathfrak{p}$. If $\mathfrak{p}$ is
inert or ramified in $K$, then $K_\mathfrak{p}$ is a quadratic
extension of $F_\mathfrak{p}$. If $\mathfrak{p}$ is split in $K$,
then $K_\mathfrak{p}\cong F_\mathfrak{p}^{\oplus 2}$ and
$F_\mathfrak{p}$ diagonally embeds into $K_\mathfrak{p}$.

Let $\mathfrak{c}$ be an ideal of $\mathcal{O}_F$ that is coprime to
$p$. For each $J$-tuple of nonnegative integers
$\vec{n}=(n_{\mathfrak{p}})_{\mathfrak{p}\in J}$, let
$$\mathcal{O}_{\vec{n},\mathfrak{c}}:=\mathcal{O}_F +
\mathfrak{c}\cdot\prod_{\mathfrak{p}\in
J}\mathfrak{p}^{n_\mathfrak{p}} \mathcal{O}_K$$ be the order of $K$
of conductor $\mathfrak{c}\cdot\prod_{\mathfrak{p}\in
J}\mathfrak{p}^{n_\mathfrak{p}} $. Let $K_{\vec{n},\mathfrak{c}}$ be
the ring class field of $K$ of conductor $
\mathfrak{c}\cdot\prod_{\mathfrak{p}\in
J}\mathfrak{p}^{n_\mathfrak{p}}$, and
$\mathcal{G}_{\vec{n},\mathfrak{c}}=\mathrm{Gal}(K_{\vec{n},\mathfrak{c}}/K)$
be its Galois group. Then $\mathcal{G}_{\vec{n},\mathfrak{c}}$ is
isomorphic to $K^\times \widehat{F}^\times \backslash
\widehat{K}^\times/\widehat{\mathcal{O}}_{\vec{n},\mathfrak{c}}^\times$.
Put $\mathcal{G}_{J;\mathfrak{c}}
:=\lim\limits_{\overleftarrow{\;\;\;\vec{n}\;\;}}
\mathcal{G}_{\vec{n},\mathfrak{c}}$.
Let $\Gamma^-_J$ be the maximal $\BZ_p$-free quotient of
$\mathcal{G}_{J;\mathfrak{c}}$.

Note that $$c_0:=[K^\times \widehat{F}^\times \backslash
\widehat{K}^\times:\mathcal{O}_K^\times
\widehat{\mathcal{O}}_F^\times \backslash
\widehat{\mathcal{O}}_K^\times]\leq \mathrm{cl}(K)$$ where
$\mathrm{cl}(K)$ denote the class number of $K$. Via
$\mathrm{rec}_K$
$$\prod_{\mathfrak{p}\in
J}(\mathcal{O}_{K_\mathfrak{p}}^1/\mathcal{O}_{F_\mathfrak{p}}^1)/\mathrm{torsion}$$
is a subgroup of $\Gamma^-_J$ of index dividing $c_0$, and
$$\Gamma^-_J\cong \prod_{\mathfrak{p}\in
J}(\mathcal{O}_{K_\mathfrak{p}}^1/\mathcal{O}_{F_\mathfrak{p}}^1)/\mathrm{torsion}$$
if $c_0$ is corpime to $p$. Here, for a discrete valuation field $L$
let $\mathfrak{m}_L$ be the maximal ideal of $\mathcal{O}_L$ and we
write $\mathcal{O}_L^1= 1+\mathfrak{m}_L$. When
$\mathfrak{p}=\mathfrak{P}\bar{\mathfrak{P}}$ is split in $K$, we
write $\mathcal{O}_{K_\mathfrak{p}}^1=
\mathcal{O}_{K_\mathfrak{P}}^1 \oplus
\mathcal{O}_{K_{\bar{\mathfrak{P}}}}^1 \cong
(\mathcal{O}_{F_\mathfrak{p}}^1)^{\oplus 2}$.

We will define an analytic family of characters of $\Gamma^-_J$.

Let $\sigma$ be a real place of $F$. In the case when
$\mathfrak{p}_\sigma$ is inert or ramified in $K$, let
$\epsilon_\sigma$ denote the character
$$\tilde{\epsilon}_\sigma: \mathcal{O}^1_{K_{\mathfrak{p}_\sigma}}/\mathcal{O}_{F_{\mathfrak{p}_\sigma}}^1 \rightarrow
\BC_p^\times, \ \ a\mapsto
 \frac{\sigma'(a)}{\bar{\sigma}'(a)} .$$ In the case when
$\mathfrak{p}_\sigma$ is split in $K$, let $\epsilon_\sigma$ denote
the character
$$\tilde{\epsilon}_\sigma: \mathcal{O}^1_{K_{\mathfrak{p}_\sigma}}/\mathcal{O}_{F_{\mathfrak{p}_\sigma}}^1
=(\mathcal{O}^1_{K_{\mathfrak{P}_\sigma}}\oplus
\mathcal{O}^1_{K_{\bar{\mathfrak{P}}_\sigma}})/\mathcal{O}_{F_{\mathfrak{p}_\sigma}}^1
\rightarrow \BC_p^\times , \ \ \ \ (a, b)\mapsto
\sigma'\left(\frac{a}{b}\right) .  $$

Let $\mu$ be a sufficiently large positive integer such that for
each $\mathfrak{p}\in J$, roots of unity in $F_\mathfrak{p}$ and
$K_\mathfrak{p} $ (when $\mathfrak{p}$ is inert or ramified in $K$)
are of order dividing $p^\mu$.

\begin{lem}
$\tilde{\epsilon}_\sigma^{p^{\mu}}$ factors through $$
(\mathcal{O}_{K_{\mathfrak{p}_\sigma}}^1/\mathcal{O}_{F_{\mathfrak{p}_\sigma}}^1)/\mathrm{torsion}
$$ and thus gives a character
$$
(\mathcal{O}_{K_{\mathfrak{p}_\sigma}}^1/\mathcal{O}_{F_{\mathfrak{p}_\sigma}}^1)/\mathrm{torsion}
\xrightarrow{\tilde{\epsilon}_\sigma^{p^\mu}} \BC_p^\times  .$$
\end{lem}
\begin{proof} Obvious. \end{proof}

We extend the additive character
$\frac{1}{p^\mu}\log\tilde{\epsilon}_\sigma^{p^\mu}$ linearly to
$\Gamma^-_J$; it is independent of $\mu$, and we denote it by
$\log_\sigma$. Note that the image of $\log_\sigma$ is contained in
$c_0^{-1}p^{-1}\sigma'(\mathcal{O}_{K_{\mathfrak{P}_\sigma}})$.

Let $\LOG_J$ denote the $\BC_p$-vector space spanned by $\log_\sigma$ ($\sigma\in \Sigma_J$). For each $\mathfrak{p}\in J$ we write
$$ \log_\mathfrak{p}=\sum_{\sigma\in\Sigma_\mathfrak{p}}\log_\sigma. $$ Then $\log_\mathfrak{p}\in \LOG_J$.

If $s\in \BC_p$ satisfies $|s|\leq |c_0|_pp^{-2}$, we define the
character $\epsilon_\sigma^{s}$ of $\Gamma^-_J$ by
$$\epsilon_\sigma^{s}(a) := \exp\Big(s\cdot \log_\sigma(a) \Big) .$$
For each $\Sigma_J$-tuple $\vec{s}=(s_\sigma)_{\sigma\in\Sigma_J}$
of elements in $\BC_p$ with $|s_\sigma|\leq |c_0|_pp^{-2}$ for each
$\sigma\in \Sigma_J$, we have the character
$$ \epsilon^{\vec{s}}:=\prod_{\sigma\in
\Sigma_J}\epsilon_\sigma^{s_\sigma} $$ of $\Gamma_J^-$. Hence, we
have an analyitc families of characters $$ \{  \epsilon ^{\vec{s}} :
\vec{s}=(s_{\sigma})_{\sigma\in \Sigma_J}, |s_\sigma|\leq |c_0|_p
p^{-2} \}.
$$

Applying the above construction to the case $F=\BQ$ and
$K=\mathcal{K}$, we obtain the character
$\tilde{\epsilon}_\mathcal{K}$ and an analytic family
$\epsilon_\mathcal{K}^s$ of characters of the group
$\Gamma^-_\mathcal{K}=\mathrm{Gal}(\mathcal{K}^-_\infty/\mathcal{K})$.
The character $\langle \ \cdot \ \rangle_{\mathrm{anti}}$ in our
introduction is exactly $\epsilon_\mathcal{K}^1$ and $\langle \
\cdot \ \rangle_{\mathrm{anti}}^s=\epsilon_\mathcal{K}^s$.

\subsection{Additive characters and measures}

Let $L$ be a number field, $d_L=[L:\BQ]$, $\CD_L$ the different of
$L$. In our application $L=F$ or $K$.

For each finite place $v$ of $L$ above a prime number $\ell$, let
$|\cdot|_v$ or $|\cdot|_{L_v}$ denote the absolute value of $L_v$
defined by $|x|_v=|\mathrm{N}_{L_v/\BQ_\ell}(x)|_{\ell}$. For each
finite place $v$ let $\omega_v$ be a uniformizing element of
$\mathcal{O}_{L_v}$.

For each finite place $v$ of $L$ we define the local zeta function
$\zeta_{L_v}(s)$ by $\zeta_{L_v}(s)=\frac{1}{1-|\omega_v|_v^s}$. For
real archimedean places we define
$\zeta_\BR(s)=\pi^{-\frac{s}{2}}\Gamma(\frac{s}{2})$ and for complex
archimedean places we define $\zeta_\BC(s)=2(2\pi)^{-s}\Gamma(s)$.
Then the zeta function $\zeta_L(s)$ is defined by
$$\zeta_L(s)=\prod_{v} \zeta_{L_v}(s)$$ which is convergent when
$\mathrm{Re}(s)>1$.

Let $\psi_\BQ$ be the standard additive character of $\BA/\BQ$ such
that $\psi_\BQ(x_\infty)=\mathrm{exp}(2\pi i x_\infty)$. We define
the additive character $\psi$ of $\BA_L/L$ by $$\psi=\psi_\BQ\circ
\mathrm{Tr}_{L/\BQ},$$  where $\mathrm{Tr}_{L/\BQ}$ is the trace map
from $\BA_L$ to $\BA$.  We write $\psi=\prod_v \psi_v$, where for
each $v$, $\psi_v$ is an additive character of $L_v$.

If $v$ is finite, let $d_v x$ be the Haar measure on $L_v$ that is
self dual with respect to the pairing $(x, x')\mapsto \psi_v(xx')$.
Then $$\mathrm{vol}(\mathcal{O}_{L_v}, d_v x)=
|\CD_{L_v}|_v^{\frac{1}{2}}.$$ If $v$ is real, we define $d_v x$ to
be the Lebesgue measure on $\BR$. If $v$ is complex, we define $d_v
x$ to be twice the Lebesgue measure on $\BC$.

We define the Haar measure on $L_v^\times$ by $$d_v^\times x=
\zeta_{L_v}(1) \frac{d_vx}{|x|_v}.$$ When $v$ is finite, we have
$$\mathrm{vol}(\mathcal{O}_{L_v}^\times, d^\times_v x)=
|\CD_{L_v}|_v^{\frac{1}{2}}.$$ Taking product we obtain the measure
on $\BA_L^\times$.

Taking $L=F$ or $K$ we obtain measures on $\BA_F^\times$ and
$\BA_K^\times$. We also need the quotient measure on
$\BA_K^\times/\BA_F^\times$. For our convenience we give a
description of it. If $v$ is a finite place of $F$ that is split in
$K$, then $K^\times_v\cong F^\times_v\times F^\times_v$ with
$F_v^\times$ diagonally embedding into $K^\times_v$. In this case we
have $K_v^\times/F_v^\times \cong F_v^\times , \ \ (a,b)\mapsto
ab^{-1}$. The quotient measure on $K_v^\times/F_v^\times$ coincides
with the pullback via this isomorphism of the measure on
$F_v^\times$. When $v$ is inert or ramified in $K$, the quotient
measure on $K_v^\times/F_v^\times$ is the Haar measure such that the
volume of $\mathcal{O}_{K_v}^\times/\mathcal{O}_{F_v}^\times$ is
\begin{equation}\label{eq:local-vol}
\mathrm{vol}
(\mathcal{O}_{K_v}^\times/\mathcal{O}_{F_v}^\times)
=\frac{\mathrm{vol}(\mathcal{O}^\times_{K_v})}{\mathrm{vol}(\mathcal{O}_{F_v}^\times)}
=
\frac{|\mathcal{D}_{K_v}|_{K_v}^{\frac{1}{2}}}{|\mathcal{D}_{F_v}|_v^{\frac{1}{2}}}
\end{equation}

\section{Automorphic representations and Whittaker models}
\label{sec:auto-rep}

\subsection{Admissible representations}

The reference of this subsection is \cite{JL70}.

Let $v$ be a finite place of $F$. If $\mathfrak{n}$ is an ideal of
$\mathcal{O}_{F_v}$, we use $U_0(\mathfrak{n})$ to denote the
subgroup
$$\{ g\in \wvec{a}{b}{c}{d} \in \GL_2(\mathcal{O}_{F_v}) : c \in
\mathfrak{n} \}.$$ When $\mathfrak{n}=(\omega_v)$ we write $
U_0(\omega_v)$ for $U_0(\mathfrak{n})$.

For two characters $\mu_1, \mu_2$ of $F_v^\times$, we have an
induced representation $\mathrm{Ind}(\mu_1\otimes \mu_2)$ of
$\mathrm{GL}_2(F_v)$. It is realized by right translation on the
space of functions $f$ on $\GL_2(F_v)$ satisfying
$$ f(\wvec{a}{b}{0}{d} g)  = \mu_1(a)\mu_2(d) \left|\frac{a}{d}\right|_v^{\frac{1}{2}} f(g)
$$ for all $g\in \GL_2(F_v)$, $a,d \in F_v^\times$ and $b\in F_v$.

When $\mu_1\mu_2^{-1}\neq |\cdot|_v, |\cdot|_v^{-1}$,
$\mathrm{Ind}(\mu_1\otimes \mu_2)$ is irreducible, and we denote it
by $\pi(\mu_1,\mu_2)$. If furthermore $\mu_1$ and $\mu_2$ are
unramified, then there exist nonzero
$\mathrm{GL}_2(\mathcal{O}_{F_v})$-invariant vectors in
$\pi(\mu_1,\mu_2)$. In this case, we say that $\pi(\mu_1,\mu_2)$ is
unraimified.

When $\mu_1\mu_2^{-1}=|\cdot|_v$, $\mathrm{Ind}(\mu_1\otimes \mu_2)$
contains a unique irreducible proper invariant subspace denoted by
$\sigma(\mu_1, \mu_1|\cdot|_v^{-1})$; $\sigma(\mu_1,
\mu_1|\cdot|_v^{-1})$ is of codimension $1$ in
$\mathrm{Ind}(\mu_1\otimes \mu_2)$; we call $\sigma(\mu_1,
\mu_1|\cdot|_v^{-1})$ {\it special}. If furthermore $\mu_1$ is
unramified, then there exist nonzero $U_0(\omega_v)$-invariant
vectors in $\sigma(\mu_1, \mu_1|\cdot|_v^{-1})$. In this case we say
that $\sigma(\mu_1, \mu_1|\cdot|_v^{-1})$ is {\it unramified
special}.

If $\pi=\pi(\mu_1,\mu_2)$ is unramified and $\chi$ is an unramified
character of $F_v^\times$, one defines the $L$-function
$$ L(s, \pi, \chi ) = [(1-\chi(\omega_v)\mu_1(\omega_v)|\omega_v|_v^s) (1-\chi(\omega_v)\mu_2(\omega_v)|\omega_v|_v^s)]^{-1} . $$
If $\pi=\sigma(\mu, \mu|\cdot|_v^{-1})$ is unramified special, and
$\chi$ is unramified, one defines the $L$-function
$$L(s,\pi, \chi) = (1-\chi(\omega_v)\mu(\omega_v)|\omega_v|_v^s)^{-1} .$$

\subsection{Whittaker models}

Fix  a place $v$ of $F$. Let $\pi$ be an admissible irreducible
representation of $\GL_2(F_v)$ with trivial central character,
$\mathfrak{n}=\mathfrak{n}_\pi$ the conductor of $\pi$.

Let $\mathcal{W}(\pi, \psi_v)$ be the Whittaker model of $\pi$
attached to the additive character $\psi_v$ \cite{JL70}. Recall that
$\mathcal{W}(\pi, \psi_v)$ is a subspace of smooth functions $$W:
\GL_2(F_v)\rightarrow \BC$$ that satisfy the following conditions:

$\bullet$ $W(\wvec{1}{x}{0}{1}g)=\psi_v(x)W(g)$ for all $x\in F_v$.

$\bullet$ If $v$ is archimedean, then
$W(\wvec{a}{0}{0}{1})=O(|a|_v^M)$ for some positive number $M$.

\noindent When $\pi$ is not $1$-dimensional, such a  Whittaker model
uniquely exists \cite[Theorem 2.14]{JL70}. The space
$\mathcal{W}(\pi, \psi_v)$ is precisely described by \cite[Lemma
14.3]{Jac72}.

If $v$ is archimedean and $\pi$ is a discrete series of weight $k$,
we take $W_\pi\in \mathcal{W}(\pi, \psi_v)$ to be
$$ W_\pi( z \wvec{a}{x}{0}{1} \wvec{\cos \phi}{\sin \phi}{-\sin \phi}{\cos \phi} ) = a^{\frac{k}{2}} e^{-2\pi a} 1_{\BR^+}(a)\cdot \mathrm{sgn}(z)^k \psi(x)e^{ik \phi}
$$ for $a,z\in \BR^\times, x, \phi\in \BR$. See
\cite[\S 5]{JL70} for the existence of such a $W_\pi$. 

Let $W_{\pi}$ be the Whittaker newform normalized such that
$W_{\pi}(\wvec{d_{F_v}}{0}{0}{1})=1$ and $W_{\pi}(gu)=W_{\pi}(g)$
for all $u\in U_0(\mathfrak{n})$. Here, $d_{F_v}$ a generator of
$\CD_{F_v}$.

\begin{lem}\label{lem:whittaker}
\begin{enumerate}
\item\label{it:don} If $\pi=\pi(\mu_1,\mu_2)$ is unramified, then
$W_\pi(\wvec{a}{0}{0}{1})=W^0_\pi(\wvec{ad_{F_v}^{-1}}{0}{0}{1})$
with
$$ W^0_\pi (\wvec{a}{0}{0}{1})=
\frac{\mu_1(a\omega_v)-\mu_2(a\omega_2)}{\mu_1(\omega_v)-\mu_2(\omega_v)}
|a|_v^{\frac{1}{2}}1_{\mathcal{O}_{F_v}}(a) .
$$
\item\label{it:special} If $\pi=\sigma(\mu, \mu|\cdot|_v^{-1})$ is unramified special, then
$W_\pi(\wvec{a}{0}{0}{1})=W^0_\pi(\wvec{ad_{F_v}^{-1}}{0}{0}{1})$
with
$$ W^0_\pi (\wvec{a}{0}{0}{1})= \mu(a) |a|_v^{\frac{1}{2}}  1_{\mathcal{O}_{F_v}}(a) .
$$
\end{enumerate}
\end{lem}
\begin{proof} Assertion (\ref{it:don}) is just Macdonald's formula
\cite{Mac-1, Mac-2}. See also \cite[Theorem 4.6.5]{Bump}. Note that
the conductor of our $\psi_v$ is $\mathcal{D}_{F_v}$ while the
conductor of $\psi_v$ in loc. cit. is $\mathcal{O}_{F_v}$.

For (\ref{it:special}) we observe that
\begin{equation} \label{eq:unit-invariant}
W_\pi(\wvec{au}{0}{0}{1}) = W_\pi(\wvec{a}{0}{0}{1})
\end{equation} for
$u\in \mathcal{O}^\times_{F_v}$, since $W_\pi$ is right invariant by
the group $\{\wvec{u}{0}{0}{1}: u\in \mathcal{O}^\times_{F_v}\}
\subset U_0(\omega_v)$. Let $\mathrm{U}_v$ be the Hecke operator
defined by
$$(W_\pi|\mathrm{U}_v)(g)=\sum_{b\in
\mathcal{O}_{F_v}/\omega_v\mathcal{O}_{F_v}}W_\pi(g
\wvec{\omega_v}{b}{0}{1}).$$ Then $W_\pi$ is an eigenvector of
$\mathrm{U}_v$. Write $W_\pi|\mathrm{U}_v = \alpha_vW_\pi$. Then
\begin{eqnarray*} \alpha_v W_\pi (\wvec{\omega_v^m}{0}{0}{1}) & = &
\sum_{b\in
\mathcal{O}_{F_v}/\omega_v\mathcal{O}_{F_v}}W_\pi(\wvec{\omega_v^m}{0}{0}{1}
\wvec{\omega_v}{b}{0}{1}) \\ &=& \sum_{b\in
\mathcal{O}_{F_v}/\omega_v\mathcal{O}_{F_v}}
W_\pi(\wvec{1}{\omega_v^{m}b}{0}{1}\wvec{\omega_v^{m+1}}{0}{0}{1})
\\ &=& \sum_{b\in \mathcal{O}_{F_v}/\omega_v\mathcal{O}_{F_v}}
\psi_v( \omega_v^{m}b)W_\pi(\wvec{\omega_v^{m+1}}{0}{0}{1})
\\ &=& \left\{\begin{array}{ll} 0 & \text{ if } m<\ord_v d_{F_v}, \\
|\omega_v|_v^{-1} W_\pi(\wvec{\omega_v^{m+1}}{0}{0}{1}) & \text{ if
} m\geq\ord_v d_{F_v}. \end{array}\right.
\end{eqnarray*} Comparing this with \cite[Lemma 14.3]{Jac72} we
see that $\alpha_v=\mu(\omega_v)|\omega_v|_v^{-\frac{1}{2}}$ and
obtain the desired expression of $W_\pi$.
\end{proof}

\begin{rem}\label{rem:alpha} If further $\pi=\sigma(\mu,\mu|\cdot|_v^{-1})$ is of trivial central character, then $\mu^2=|\cdot|_v$ and $\alpha_v=\pm 1$.
\end{rem}

For each $W\in \mathcal{W}(\pi, \psi_v)$ and each continuous
character $\chi: F_v^\times \rightarrow \BC$ the local zeta integral
is defined by
$$ \Psi(s,W, \chi) =\int_{F_v^\times} W(\wvec{a}{0}{0}{1})\chi(a) |a|_v^{s-\frac{1}{2}}d^\times a , \hskip 10pt (s\in \BC).
$$ Then $\Psi(s,W,\chi)$ converges when $\mathrm{Re}(s) $ is sufficiently
large, and has a meromorphic continuation to the whose $\BC$.

\begin{prop}\label{prop:zeta-L} In the case when $v$ is finite,
if $\pi$ is unramified or unramified special, and if $\chi$ is
unramified, then
$$ \Psi (s,W_{\pi}, \chi ) = L(s,\pi\otimes \chi )
\cdot  \chi(\mathcal{D}_{F_v})|\mathcal{D}_{F_v}|_v^s
$$ \end{prop}
Compare this with \cite[(2.1)]{Hung14} and note that the formula in
loc. cit. misses a factor.
\begin{proof} We have
{\allowdisplaybreaks \begin{eqnarray*} \Psi ( s,W_{\pi}, \chi) &=&
\int_{F_v^\times} W^0(\wvec{ad_{F_v}^{-1}}{0}{0}{1})\chi(a)
|a|_v^{s-\frac{1}{2}}d^\times a \\ &=& \int_{F_v^\times}
W^0(\wvec{a}{0}{0}{1})\chi(ad_{F_v} ) |ad_{F_v}
|_v^{s-\frac{1}{2}}d^\times a \\
&=& |d_{F_v}|_v^{s-\frac{1}{2}}\chi(d_{F_v}) \int_{F_v^\times}
W^0(\wvec{a}{0}{0}{1})\chi(a ) |a|_v^{s-\frac{1}{2}}d^\times a.
\end{eqnarray*} }

In the unramified special case $\pi=\sigma(\mu, \mu|\cdot|_v^{-1})$ we
have {\allowdisplaybreaks \begin{eqnarray*} \int_{F_v^\times}
W^0(\wvec{a}{0}{0}{1})\chi(a) |a|_v^{s-\frac{1}{2}}d^\times a &=&
\mathrm{vol}(\mathcal{O}_{L_v}^\times, d^\times_v x)\sum_{i\geq 0}
\mu(\omega_v)^i\chi(\omega_v)^i |\omega_v|_v^{is} \\
&=& \frac{|d_{F_v}|_v^{1/2}}{1-\mu(\omega_v) \chi(\omega_v)
|\omega_v|_v^{s}} \\ &=& |d_{F_v}|_v^{1/2} L(s,\pi\otimes \chi).
\end{eqnarray*}}
In the unramified case $\pi=\pi(\mu_1,\mu_2)$ we have
{\allowdisplaybreaks \begin{eqnarray*} && \int_{F_v^\times}
W^0(\wvec{a}{0}{0}{1})\chi(a) |a|_v^{s-\frac{1}{2}}d^\times a \\ &=&
\mathrm{vol}(\mathcal{O}_{L_v}^\times, d^\times_v x)\sum_{i\geq
0}\frac{\mu_1(\omega_v)^{1+i}-\mu_2(\omega_v)^{1+i}}{\mu_1(\omega_v)-\mu_2(\omega_v)}
\chi(\omega_v)^i |\omega_v|_v^{is} \\
&=& \frac{|d_{F_v}|_v^{1/2}}{\mu_1(\omega_v)-\mu_2(\omega_v)}
\left(\frac{\mu_1(\omega_v)}{1-\chi(\omega_v)\mu_1(\omega_v)|\omega_v|_v^s}-\frac{\mu_2(\omega_v)}{1-\chi(\omega_v)\mu_2(\omega_v)|\omega_v|_v^s}\right)\\
&=& |d_{F_v}|_v^{1/2} L(s,\pi\otimes \chi),
\end{eqnarray*}} as expected.
\end{proof}


One defines the $\GL_2(F_v)$-invariant pairing
$$ \mathbf{b}_v: \CW(\pi,\psi_v)\times \CW(\pi,\psi_v)\rightarrow
\BC $$ by
$$ \mathbf{b}_v(W_1, W_2) = \int_{F_v^\times} W_1(\wvec{a}{0}{0}{1}) W_2(\wvec{-a}{0}{0}{1}) d^\times a.  $$

\begin{prop} \label{prop:hung} $($\cite[Lemma 4.2]{Hung14}$)$
\begin{enumerate}
\item\label{it:arch} If $v$ is archimedean and $\pi$ is a discrete
series of weight $k$, then
$$ \mathbf{b}_v( W_\pi,
\pi(\wvec{-1}{0}{0}{1})W_\pi ) = (4\pi)^{-k}\Gamma(k). $$
\item\label{it:b-value-unram}
If $v$ is finite and $\pi$ is unramified, then $$
\mathbf{b}_v(W_{\pi}, W_{\pi})=
\frac{\zeta_{F_v}(1)}{\zeta_{F_v}(2)}L( 1,\mathrm{Ad} \pi )
|\CD_{F_v}|_v^{1/2}.$$
\item\label{it:b-value-special}
If $v$ is finite and $\pi=\sigma(\mu,\mu|\cdot|_v^{-1})$ is unramified
special, then
$$\mathbf{b}_v(W_{\pi}, \pi(\wvec{0}{1}{-\omega_v}{0})W_{\pi})=
\epsilon({1}/{2}, \pi, \psi_v) L(1, \mathrm{Ad} \pi)
|\CD_{F_v}|_v^{1/2}.$$
\end{enumerate}

\noindent Here, $\mathrm{Ad} \pi$ is the adjoint representation
associated to $\pi$ \cite{GJ78}.
\end{prop}

Now, let $\pi$ be an irreducible admissible representation of
$\GL_2(\BA _F)$. By tensor product theorem (see \cite[Theorem
3.3.3]{Bump}) $\pi$ has a decomposition $\pi=\otimes_v \pi_v$, where
for each $v$, $\pi_v$ is an admissible representation of
$\GL_2(F_v)$. Then $\pi$ has a Whittaker model $\mathcal{W}$ if and
only if for each place $v$, $\pi_v$ has a Whittaker model
$\mathcal{W}_v$ (see \cite[Theorem 3.5.4]{Bump}); in this case
$\mathcal{W}$ is unique and consists of functions of the from
$W(g)=\prod_{v}W_v(g_v)$, where $W_v\in \mathcal{W}_v$ and
$W_v=W_{\pi_v}$ for almost all $v$.

We write $\mathfrak{n}$, the conductor of $\pi$, in the form
\begin{equation}\label{eq:conductor}
\mathfrak{n}
=\mathfrak{n}^+\mathfrak{n}^-\prod_{\mathfrak{p}|p}\mathfrak{p}^{r_\mathfrak{p}},\end{equation}
where $\mathfrak{n}^+$ (resp. $\mathfrak{n}^-$) is only divided by
primes that are split (resp. inert or ramified) in $K$, and
$(\mathfrak{n}^+\mathfrak{n}^-,p)=1$. We assume that
$\mathfrak{n}^-$ is square-free.

\subsection{Modular forms on definite quaternion
algebras} \label{ss:mod-form}

Let $B$ be a definite quaternion algebra over $F$ with discriminant
$\mathfrak{n}^-_b|\mathfrak{n}^-$. Fix an Eichler order $R$ of $B$
of level $\mathfrak{n}^+$. \label{ss:quaternion}

As ramified primes of $B$ are not split in $K$ (i.e.
$\mathfrak{n}^-_b|\mathfrak{n}^-$), $K$ is isomorphic to a subfield
of $B$.  We embed $K$ into $ B$ such that $\mathcal{O}_K=K\cap R$.
Write $B=K\oplus KI$ with $I^2=\beta\in F^\times$. We may take $I$
such that $\beta$ is totally negative and $\beta_{v}\in
(\mathcal{O}_{F_v}^\times)^2$ for each finite place
$v|p\frac{\mathfrak{n}}{\mathfrak{n}^-_b}\mathfrak{n}_\nu $;
$\beta_{v}\in  \mathcal{O}_{F_v}^\times $ if $v|
N_{K/F}\mathcal{D}_K$. Here $\mathfrak{n}_\nu$ is the conductor of
$\nu$. We will assume that $(\mathfrak{n}_\nu,\mathfrak{n}^-_b)=1$.

We fix an isomorphism $B\otimes_FK\cong M_2(K)$ such that $a\otimes
1\mapsto \wvec{a}{0}{0}{\bar{a}}$ (for  $a\in K)$, and $I\otimes
1\rightarrow \wvec{0}{-\beta}{-1}{0}$. For each $\sigma\in \Sigma$,
we have an inclusion $$i_\sigma: B\hookrightarrow B\otimes_FK\cong
M_2(K) \xrightarrow{\sigma} M_2(\BC).
$$ and extend it to $B_\sigma$.

Fix an element $\vartheta\in K$ such that $\{1,\vartheta_v\}$ is a
basis of $\mathcal{O}_{K_v}$ over $\mathcal{O}_{F_v}$ for each
$v|p\frac{\mathfrak{n}}{\mathfrak{n}^-_b}\mathfrak{n}_\nu$. For such
a $v$ we define the isomorphism $i_{v}: B_v\simeq M_2(F_{v})$  by
$$i_{v}(\vartheta_\mathfrak{p})=\wvec{\mathrm{T}(\vartheta_v)}{-\mathrm{N}(\vartheta_{v})}{1}{0},
\hskip 10pt
i_v(I)=\sqrt{\beta}\wvec{-1}{\mathrm{T}(\vartheta_v)}{0}{1},
$$ where
$\mathrm{T}(\vartheta_v)=\vartheta_v+\bar{\vartheta}_v$ and
$\mathrm{N}(\vartheta_v)=\vartheta_v\bar{\vartheta}_v$.

For each finite place $v\nmid p\mathfrak{n} \mathfrak{n}_\nu$ that
is split in $K$ we fix an isomorphism $i_v: B_v \simeq M_2(F_v)$
such that $i_v(\mathcal{O}_{K})\subset M_2(\mathcal{O}_{F_v})$ and
is diagonally in $M_2(\mathcal{O}_{F_v})$, and such that
$$i_v(I)\in \left\{\begin{array}{ll} F_v^\times
\mathrm{GL}_2(\mathcal{O}_{F_v}) & \text{ if }\mathrm{val}_v(\beta) \text{ is even}, \\
F_v^\times \wvec{\omega_v}{0}{0}{1}\mathrm{GL}_2(\mathcal{O}_{F_v})
& \text{ if }\mathrm{val}_v(\beta) \text{ is odd}.
\end{array}\right.
$$ Here $\mathrm{val}_v$ is the valuation on $F_v^\times$ such that
$\mathrm{val}_v(\omega_v)=1$.

Let $G$ be the algebraic group over $\BQ$ such that for any
$\BQ$-algebra $A$, $$G(A)=(A\otimes_{\BQ}B)^\times.$$ Let $Z$ be the
center of $G$. Then $$ Z(A)=(A\otimes_{\BQ}F)^\times.$$

For each even integer $h\geq 2$ and a ring $A$ let $V_{h}(A)$ denote
the set of homogenous polynomials of degree $h-2$ with coefficients
in $A$. Write
$$ V_{h}(A)= \oplus_{-\frac{h}{2}< m<\frac{h}{2}} A v_m,   \hskip 10pt v_m = X ^ {\frac{h-2}{2}-m} Y^ { \frac{h-2}{2}+m }.
$$ Let $\rho_h$ be the algebraic representation of
$\mathrm{GL}_2(A)$ on $V_{h}(A)$ defined by
$$ (\rho_h(g) P)(X,Y)= \det(g)^{-\frac{h-2}{2}} P((X,Y)g) .$$ Then
$\rho_h$ has a trivial central character.

Let $\vec{k}=(k_\sigma)_{\sigma\in \Sigma_F}$ be a $\Sigma_F$-tuple
of positive even integers. Put $V_{\vec{k}}(\BC)=\otimes_{\sigma\in
\Sigma_F}V_{k_\sigma}(\BC)$ and let $\rho_{\vec{k}}$ be the
representation of $G(\BR)$ defined by
$$ \rho_{\vec{k}}: G(\BR) \xrightarrow{\prod_\sigma \iota_\sigma}
\prod_{\sigma}\mathrm{GL}_2(\BC)
\xrightarrow{\prod_{\sigma}\rho_{k_\sigma}} \mathrm{Aut}(
V_{\vec{k}}(\BC) ). $$ For
$\mathbf{m}=\sum\limits_{\sigma\in\Sigma_K}m_\sigma \sigma$ of
integers with $-\frac{k_\sigma}{2}<m_\sigma<\frac{k_\sigma}{2}$ let
$\mathbf{v}_{\mathbf{m}}$ denote the element $\otimes_\sigma
v_{m_\sigma}$ of $ V_{\vec{k}}(\BC)$. Then $z =(z_\sigma)_{\sigma}
\in \prod_{\sigma}K_\sigma^\times = K_\infty^\times$ acts on
$\mathbf{v}_{\mathbf{m}}$ by $z\cdot \mathbf{v}_{\mathbf{m}} =
\prod_{\sigma}(\frac{z_\sigma}{\bar{z}_\sigma})^{-m_\sigma} \cdot
\mathbf{v}_{\mathbf{m}}$.

The contragredient representation $\check{\rho}_{\vec{k}}$ of
$\rho_{\vec{k}}$ can be realized on the dual $L_{\vec{k}}(\BC)$ of
$V_{\vec{k}}(\BC)$ such that
$$ \langle \check{\rho}_{\vec{k}}(g) \ell, {\rho}_{\vec{k}}(g) v \rangle =\langle\ell, v\rangle
$$ for all $g\in G(\BR)$, $\ell\in L_{\vec{k}}(\BC)$ and $v\in V_{\vec{k}}(\BC)$.
Here $\langle\cdot, \cdot \rangle$ is the pairing $$
\langle\cdot,\cdot\rangle: \hskip 10pt L_{\vec{k}}(\BC)\times
V_{\vec{k}}(\BC)\rightarrow \BC.$$

\begin{defn} Let $U$ be a compact open subgroup of $G(\BA^\infty)$.  A {\it $($algebraic$)$ modular form} on $B^\times$ of {\it trivial central character}, {\it
weight} $\vec{k}$ and {\it level} $U$, is a function $f: \widehat{ B
}^\times\rightarrow L_{\vec{k}}(\BC)$ that satisfies
$$ f( z \gamma b u ) = \check{\rho}_{\vec{k}} (\gamma) f(b) $$
for all $\gamma\in  B ^\times$, $u\in U$, $b\in \widehat{ B
}^\times$ and $z\in \widehat{F}^\times$. Denote by $M^{ B
}_{\vec{k}}(U, \BC)$ the space of such forms. Put
$M^B_{\vec{k}}(\BC)=\lim\limits_{\overrightarrow{\;\; U \;\; }}M^{ B
}_{\vec{k}}(U,\BC) $. Then $G(\BA ^\infty)$ acts on
$M^B_{\vec{k}}(\BC)$ by right translation.
\end{defn}

For each prime $\mathfrak{l}$ of $F$ such that $ B $ splits at
$\mathfrak{l}$, $\mathfrak{l}\nmid p$, and $U_\mathfrak{l}$ is
maximal, we define a Hecke operator $\mathrm{T}_\mathfrak{l}$ on
$M^{ B }_{\mathrm{k}}(U,E)$ as follows.  We may assume that
$U_\mathfrak{l}$ becomes identified with
$\mathrm{GL}_2(\mathcal{O}_{F_\mathfrak{l}})$ via $i_\mathfrak{l}$.
Let $\omega_\mathfrak{l}$ be a uniformizer of
$\mathcal{O}_{F_\mathfrak{l}}$. Given a double coset decomposition
$$ \mathrm{GL}_2(\mathcal{O}_{F_\mathfrak{l}}) \wvec{\omega_\mathfrak{l}}{0}{0}{1} \mathrm{GL}_2(\mathcal{O}_{F_\mathfrak{l}})
=\coprod b_{\mathfrak{l},i}
\mathrm{GL}_2(\mathcal{O}_{F_\mathfrak{l}})
$$ we define the Hecke operator $\mathrm{T}_\mathfrak{l}$ on $S^{ B
}_{\mathrm{k}}(\Sigma)$ by
$$ (\mathrm{T}_\mathfrak{l}  {f})(g) =\sum_i  {f}(g b_{\mathfrak{l},i}) . $$
Assume that $i_p(U_p)\cong
U_{p,0}:=\prod_{\mathfrak{p}|p}U_0(\omega_\mathfrak{p})$. Then for
each $\mathfrak{p}|p$ using the double coset decomposition
$$ U_{\mathfrak{p}} \wvec{\omega_\mathfrak{p}}{0}{0}{1}
U_{\mathfrak{p}} =\coprod b_{\mathfrak{p},i} U_{\mathfrak{p}}$$ we
define $\mathrm{U}_\mathfrak{p}$ by
$$ \mathrm{U}_{\mathfrak{p}}  {f} (g) = \sum_i   {f} (g b_{\mathfrak{p},i}) $$
We define the Atkin operator $w_{\mathfrak{p}}$ by
$$ (w_{\mathfrak{p}} f)(g) =  {f}(g\wvec{0}{1}{\omega_\mathfrak{p}}{0}). $$ Let
$\mathbb{T}_U$ be the Hecke algebra generated by these
$\mathrm{T}_\mathfrak{l}$ and $\mathrm{U}_\mathfrak{p}$,
$w_{\mathfrak{p}}$.

For $v\in V_{\vec{k}}(\BC)$ and $f\in M^B_{\vec{k}}(\BC)$ we can
attach to $v\otimes f$ an automorphic form $\Psi(v\otimes f)$ on
$G(\BA )$ by
$$ \Psi(v\otimes f) (g) : = \langle f(g^\infty), \rho_{\vec{k}}(g_\infty) v  \rangle .
$$

Every automorphic form arises in this way. Indeed, if
$\pi'=\pi'_\infty\otimes \pi'^\infty$ is an automorphic
representation of $G(\BA )$ with $\pi'_\infty \cong \rho_{\vec{k}}$,
then $\pi'^\infty$ appears in $M_{\vec{k}}^B(\BC)$.

\subsection{$p$-adic modular forms on definite quaternion
algebras} \label{ss:p-mod-form}

For each $\sigma$ we put $k_{\sigma'}=k_\sigma$. We form two vector
spaces $V_{\vec{k}}(\BC_p)= V_{\vec{k}}(\BC)\otimes_{\BC,
\jmath}\BC_p$ and $L_{\vec{k}}(\BC_p)=L_{\vec{k}}(\BC)\otimes_{\BC,
\jmath}\BC_p$ over $\BC_p$. So there is a pairing
$L_{\vec{k}}(\BC_p)\times V_{\vec{k}}(\BC_p)\rightarrow \BC_p$. Via
$\jmath$ we have algebraic representations of $\prod_{\sigma}\GL_2(\BC_p)$ on
$V_{\vec{k}}(\BC_p)$ and $L_{\vec{k}}(\BC_p)$. Precisely, for
$g_p\in \GL_2(\BC_p)$ and $f\in L_{\vec{k}}(\BC_p)$ we have $$ g_p \cdot
f = \jmath (\check{\rho}_{\vec{k}}(\jmath^{-1}(g_p))
\jmath^{-1}(f));$$ the same holds for the algebraic representation
on $V_{\vec{k}}(\BC_p)$. By abuse of notation we again use
$\check{\rho}_{\vec{k}}$ and $\rho_{\vec{k}}$ to denote the
resulting algebraic representations.

In Section \ref{ss:mod-form} we fix an isomorphism . For each place
$\mathfrak{p}$ of $F$ above $p$, using the isomorphism
$B\otimes_FK\cong M_2(K)$ given in Section \ref{ss:mod-form} we
obtain an inclusion $$\mathbf{i}_\mathfrak{p}:
B_\mathfrak{p}\hookrightarrow M_2(K_\mathfrak{p});$$ when
$\mathfrak{p}=\mathfrak{P}\bar{\mathfrak{P}}$ is split in $K$, from
this inclusion we obtain an isomorphism
$$\mathbf{i}_\mathfrak{p}: B_\mathfrak{p}\cong M_2(K_\mathfrak{P})\cong M_2(F_\mathfrak{p}).$$

\begin{lem}\label{lem:h-sigma} In the case of
$\mathfrak{p}$ split in $K$,
$\mathfrak{i}_{\mathfrak{p}}i_{\mathfrak{p}}^{-1}=\mathrm{Ad}(\hbar_\mathfrak{p})$
with
$$ \hbar_\mathfrak{p}=\wvec{\sqrt{\beta}}{0}{0}{\sqrt{\beta}^{-1}}\wvec{1}{- \vartheta_{\bar{\mathfrak{P}} }}{1}{- \vartheta_\mathfrak{P}}\in
\GL_2(F_{\mathfrak{p}}). $$
\end{lem}
\begin{proof} This follows from a simple computation.
\end{proof}

For each $\sigma\in \Sigma_F$ we put $i_{\sigma'}=\sigma'\circ
i_{\mathfrak{p}_\sigma}$ and $\mathbf{i}_{\sigma'}=\sigma'\circ
\mathbf{i}_{\mathfrak{p}_\sigma}$. Write
$\mathfrak{i}_p=(\mathfrak{i}_{\sigma'})_{\sigma'}$.

\begin{defn} Let $U=U_pU^p$ be a compact open subgroup of $G(\BA^\infty)$.
A {\it $p$-adic modular form on $ B^\times$, of trivial central
character, weight $\vec{k}=(k_{\sigma'})$ and level $U$}, is a
continuous function $\widehat{f}: \widehat{ B }^\times\rightarrow
L_{\vec{k}}(\BC_p)$ that satisfies
$$ \widehat{f}( z \gamma b u ) =  \check{\rho}_{\vec{k}}(\mathfrak{i}_p(u_p^{-1}))\widehat{f}(b)  $$
for all $\gamma\in  B ^\times$, $u\in U$, $b\in \widehat{ B
}^\times=G(\BA^\infty)$ and $z\in \widehat{F}^\times$. Denote by
$M^{B}_{\vec{k}}(U,\BC_p)$ the space of such forms.
\end{defn}

\begin{lem}\label{lem:control-1} If $\widehat{f}$ is a $p$-adic modular form on $ B^\times$ (of trivial central
character, weight $\vec{k}=(k_{\sigma'})$ and level $U$), then the
image of $\widehat{f}$ lies in an $\mathcal{O}_{\BC_p}$-lattice of
$L_{\vec{k}}(\BC_p)$.
\end{lem}
\begin{proof} By definition, for each $a\in \widehat{B}^\times$,
$B^\times a U$ and $aU$ have the same image by the map
$\widehat{f}$. It is compact since $U$ is compact. By the theorem of
Fujisaki (see \cite{Vig}) we know that the double coset
$B^\times\backslash \widehat{B}^\times/ U$ is finite, which implies
our lemma.
\end{proof}

For each $f\in M^B_{\vec{k}}(U, \BC)$ we can attach to it a $p$-adic
modular form $\widehat{f}\in M^{B}_{\vec{k}}(U,\BC_p)$ defined by
$$ \widehat{f} (b) = \check{\rho}_{\vec{k}}(\mathfrak{i}_p(b_p^{-1})) \jmath ( f (b) ) . $$
Then we get an isomorphism
\begin{equation}\label{eq:mod-p-mod}
M^{B}_{\vec{k}}(U,\BC)\rightarrow M^{B}_{\vec{k}}(U,\BC_p), \hskip
10pt f \mapsto \widehat{f} .
\end{equation}

Let $E$ be a finite extension of $\BQ_p$ in $\BC_p$ that contains
all embeddings of $K_{\sigma'}$ for all $\sigma\in\Sigma_K$. With
$$ L_{\vec{k}}(E)=\bigotimes_{\sigma\in \Sigma_K} L_{k_\sigma}(E) $$
instead of $L_{\vec{k}}(\BC_p)$ we have the space
$M^{B}_{\vec{k}}(U,E)$ of $E$-valued modular form on $ B^\times$ of
trivial central character, weight $\vec{k}=(k_{\sigma'})$ and level
$U$. We also have the notations $V_{k_\sigma}(E)$
($\sigma\in\Sigma_K$) and $V_{\vec{k}}(E)$.

Assume that $U_p=U_{p,0}$. Using notations at the end of Section
\ref{ss:mod-form} we define operators $\mathrm{T}_\mathfrak{l}$,
$\mathrm{U}_\mathfrak{p}$ and $w_\mathfrak{p}$ by
$$ (\mathrm{T}_\mathfrak{l} \widehat{f})(g) =\sum_i \widehat{f}(g b_{\mathfrak{l},i}) , $$
$$ \mathrm{U}_{\mathfrak{p}} \widehat{f} (g) = \sum_i \check{\rho}_{\vec{k}}(\mathfrak{i}_pi_\mathfrak{p}^{-1}(b_{\mathfrak{p},i}))\widehat{f} (g b_{\mathfrak{p},i})
$$ and
$$ (w_{\mathfrak{p}} \widehat{f})(g) = \check{\rho}_{\vec{k}}(\mathfrak{i}_p i_\mathfrak{p}^{-1}\wvec{0}{1}{\omega_\mathfrak{p}}{0}) \widehat{f}(g\wvec{0}{1}{\omega_\mathfrak{p}}{0}).
$$

Let $\mathbb{T}_U$ be the Hecke algebra generated by these
$\mathrm{T}_\mathfrak{l}$ and $\mathrm{U}_\mathfrak{p}$,
$w_{\mathfrak{p}}$. Then the map (\ref{eq:mod-p-mod}) is
$\mathbb{T}_U$-equivariant.

\section{Special values of $L$-function} \label{sec:special-value}

In this section and the next one we extend the result in
\cite{Hung14} in two directions. One is that we allow much bigger
field extensions; such an extension is essential for our work. The
other is that we weaken a local condition in \cite{Hung14}.




\subsection{Waldspurger formula}

Let $\pi$ be an unitary cuspidal automorphic representation on
$\GL_2(\BA_F)$ with trivial central character. We assume that $\pi$
satisfies the following conditions:

$\bullet$ For each $\sigma\in\Sigma_F$, $\pi_\sigma$ is a discrete
series of even weight $k_\sigma$.

$\bullet$ If we write the conductor of $\pi$ as in
(\ref{eq:conductor}), then $r_{\mathfrak{p}}\leq 1$ for each
$\mathfrak{p}|p$.

$\bullet$ If $v|\mathfrak{n}^-$, then $\pi_v$ is a special
representation $\sigma(\mu_v, \mu_v|\cdot|_v^{-1})$ with unramified
character $\mu_v$.

We will fix a decomposition $\mathfrak{n}^+\mathcal{O}_K=\mathfrak{N}^+\bar{\mathfrak{N}}^+$.

Put $W_\pi=\prod_v W_{\pi_v}$. Let $\varphi_\pi$ be the normalized
new form in $\pi$ defined by
$$ \varphi_\pi(g) := \sum_{\alpha\in F} W_\pi ( \wvec{\alpha}{0}{0}{1} g).
$$

Let $\tau^\mathfrak{n}=\prod_v \tau_v^{\mathfrak{n}}$ be the
Atkin-Lehner element defined by
$$ \tau_\sigma^\mathfrak{n} = \wvec{1}{0}{0}{-1}$$ if $\sigma\in \Sigma_F$, and $$ \tau_v^\mathfrak{n}=\wvec{0}{1}{-\varpi_v^{\mathrm{ord}_v\mathfrak{n}}}{0}. $$
if $v$ is finite.

Let $d^tg$ be the Tamagawa measure on $\GL_2$. Put
$$ \langle\varphi_\pi, \varphi_\pi\rangle_{\GL_2} = \int_{\BA^\times_F \GL_2(F)\backslash \GL_2(\BA_F)} \varphi_\pi(g)\varphi_\pi(g\tau^{\mathfrak{n}}) d^tg, $$
and
$$ ||\varphi_\pi||_v := \frac{ \zeta_{F_v}(2) }{\zeta_{F_v}(1)L(\mathrm{Ad}\pi_v,1)} \mathbf{b}_v (W_{\pi_v}, \pi(\tau^{\mathfrak{n}}_v)W_{\pi_v}).
$$

\begin{prop}\label{prop:wald-1} $($Waldspurger formula \cite[Proposition 6]{Wald}$)$ We have
$$ \langle \varphi_\pi,\varphi_\pi \rangle_{\GL_2} = \frac{2L(1,\mathrm{Ad}\pi)}{\zeta_F(2)} \prod_v ||\varphi_\pi||_v .
$$
\end{prop}

Put $$ ||\varphi_\pi||_{\Gamma_0(\mathfrak{n})} = \Big(
\zeta_F(2)\cdot \mathrm{N}\mathfrak{n} \cdot \prod_{v|\mathfrak{n}}
(1+|\varpi_v|_v)\Big)
 \langle\varphi_\pi, \varphi_\pi\rangle_{\GL_2} .$$

\begin{cor}
One has {\small $$ L(1,\mathrm{Ad} \pi) = || \varphi_\pi
||_{\Gamma_0(\mathfrak{n})} \cdot \frac{2^{|k|}
2^{d-1}}{\mathrm{N}\mathfrak{n}} \cdot \prod_{v\nmid \mathfrak{n}:
\text{finite}}\mathrm{N}\CD_{F_v}^{1/2}\cdot
\prod_{v|\mathfrak{n}^-}
\epsilon(\frac{1}{2},\pi_v,\psi_v)\mathrm{N}\CD_{F_v}^{1/2} \cdot
\prod_{v|\frac{\mathfrak{n}}{\mathfrak{n}^-}}
\frac{1}{(1+|\varpi_v|_v)||\varphi_\pi||_v}.
$$} \end{cor}
\begin{proof} This follows from Proposition \ref{prop:wald-1} and
Proposition \ref{prop:hung}. Indeed, by Proposition \ref{prop:hung}
one has
$$||\varphi_\pi||_v = \left\{\begin{array}{ll}  2^{-k_\sigma-1} & \text{ if } v \text{ is archimedean},
\\
\mathrm{N}\CD_{F_v}^{1/2} &  \text{ if } v \text{ is finite and }
v\nmid \mathfrak{n},
\\ \mathrm{N}\CD_{F_v}^{1/2} \epsilon(\frac{1}{2},\pi_v,\psi_v)  &  \text{ if } v|\mathfrak{n}^-. \end{array}\right.$$
For archimedean places $\sigma$ one uses the fact
$L(1,\mathrm{Ad}\pi_\sigma )=
2^{1-k_\sigma}\pi^{-(k_\sigma+1)}\Gamma(k_\sigma)$.
\end{proof}

\subsection{Gross points}

Let $B$ and $G$ be as in Section \ref{ss:quaternion}.

Let $J$ be a set of places of $F$ that are above $p$.

Let $\mathfrak{p}$ be a prime above $p$. If $\mathfrak{p}\notin J$
and $\mathfrak{p}$ is inert or ramified in $K$, we put
$\varsigma_\mathfrak{p}=1$. If $\mathfrak{p}\notin J$ and
$\mathfrak{p}=\mathfrak{P}\bar{\mathfrak{P}}$ is split in $K$, we
put
$$\varsigma_{\mathfrak{p}}=\wvec{\vartheta_\mathfrak{P}}{\vartheta_{\bar{\mathfrak{P}}}}{1}{1}.$$
In this case we have for any $t=(t_1,t_2)\in
K_\mathfrak{p}=K_{\mathfrak{P}}\oplus K_{\bar{\mathfrak{P}}}$,
$$\varsigma_{\mathfrak{p}}^{-1} i_\mathfrak{p}(t) \varsigma_{\mathfrak{p}}=\mathfrak{i}_\mathfrak{p}(t)=\wvec{t_1}{0}{0}{t_2} .$$

When $\mathfrak{p}\in J$, for each integer $n$ we define
$\varsigma_\mathfrak{p}^{(n)}\in G(F_\mathfrak{p})$ as follows. When
$\mathfrak{p}=\mathfrak{P}\bar{\mathfrak{P}}$ splits in $K$, we put
$$ \varsigma_\mathfrak{p}^{(n)} =
\wvec{\vartheta}{-1}{1}{0}\wvec{\varpi_{\mathfrak{p}}^n}{0}{0}{1}
\in \mathrm{GL}_2(F_\mathfrak{p}) .$$ When $\mathfrak{p}$ is inert
or ramified in $K$, we put $$\varsigma_\mathfrak{p}^{(n)} =
\wvec{0}{1}{-1}{0}\wvec{\varpi_\mathfrak{p}^n}{0}{0}{1}.$$ When $n=0$ we also write $\varsigma_\mathfrak{p}$ for $\varsigma_\mathfrak{p}^{(0)}$.

Let $J'$ be a subset of $J_p$ such that $J\subset J'\subset J_p$.
 For each $J$-tuple of integers
$\vec{n}=(n_{\mathfrak{p}})_{\mathfrak{p}\in J}$ we put
$\varsigma_{J,J'}^{(\vec{n})}= \prod_{\mathfrak{p}\in J}
\varsigma_{\mathfrak{p}}^{(n_\mathfrak{p})} \cdot
\prod_{\mathfrak{p} \in J'\backslash J} \varsigma_{\mathfrak{p}} $.
When $J'=J_p$ we just write $\varsigma_{J}^{(\vec{n})}$ for
$\varsigma_{J,J_p}^{(\vec{n})}$. We also write
$\varsigma_J=\varsigma_J^{(\vec{0})}$.

\subsection{$p$-stabilizer} \label{ss:stab}

Let $\pi'=\otimes_v \pi'_v$ be the unitary cuspidal automorphic
representation on $G(\BA_F)$ with trivial central character attached
to $\pi$ via Jacquet-Langlands correspondence.  Then $\pi'_\infty$
is isomorphic to the algebraic representation $\rho_{\vec{k}}$ of
$G(\BR)$. We may assume that $\pi'_v=\pi_v\circ i_v$ for finite
place $v\nmid \mathfrak{n}^-_b$. When $v|\mathfrak{n}^-_b$, as
$\pi_v=\sigma(\mu_v, \mu_v|\cdot|_v^{-1})$ is special, $\pi'_v=
(\mu_v|\cdot|_v^{-1/2})\circ N_{B_v/F_v}$ is $1$-dimensional, where
$N_{B_v/F_v}$ is the reduced norm of $B_v$. For each $v\nmid
\mathfrak{n}$, $\pi'_v$ is unramified. We may assume that
$\pi'_v=\pi_v\circ i_v$ for each finite place $v\nmid
\mathfrak{n}^-_b$.

For each finite place $v$ we define a new vector $\varphi_v\in
\pi'_v$ by the following

$\bullet$ if $v|\mathfrak{n}^-_b$, then $\varphi_v$ is a basis of
the one dimensional representation $\pi'_v$;

$\bullet$ if $v|\mathfrak{n}^-$ but $v\nmid\mathfrak{n}^-_b$, then
$\varphi_v$ is fixed by $U_0(\mathfrak{n})_v$;

$\bullet$ if $v\nmid \mathfrak{n}$, $\varphi_v$ is the normalized
spherical vector of $\pi'_v$ that corresponds to $W_{\pi_v}$ in the
Whittaker model.

Realize $\pi'^{\infty}$ in $S^B_{\vec{k}}(\BC)$ so that
$\varphi^\infty:=\otimes_{v: \text{finite} }\varphi_v$ is an element
of $S^B_{\vec{k}}(\BC)$. Then we put $\varphi_{\pi'}=  \Psi(
\mathbf{v}_{\mathbf{m}} \otimes \varphi^\infty)$.

Next we define $p$-stabilization of $\varphi_v$.

For
$\mathfrak{p}\in J$, if $r_v=1$, we put
$\varphi^\dagger_\mathfrak{p}=\varphi_\mathfrak{p}$.
If
$r_\mathfrak{p}=0$, we choose a Satake parameter $a_\mathfrak{p}$ of $\pi_\mathfrak{p}$
and define $\alpha_{\mathfrak{p}}=a_\mathfrak{p}|\omega_\mathfrak{p}|_{\mathfrak{p}}^{-1/2}$; then we put $\varphi^\dagger_\mathfrak{p}=
\varphi_\mathfrak{p}- \frac{1}{\alpha_\mathfrak{p}}
\pi(\wvec{1}{0}{0}{\varpi_\mathfrak{p}})\varphi_\mathfrak{p}$.

If $\phi\in \pi'$ satisfies $\phi_p=\varphi_p$ (i.e.
$\phi_\mathfrak{p}=\varphi_\mathfrak{p}$ for each $\mathfrak{p}$
above $p$), we put $$ \phi^{\dagger_J}= \phi-
\Big(\prod_{\mathfrak{p}\in J:
r_\mathfrak{p}=0}\frac{1}{\alpha_\mathfrak{p}}
\pi(\wvec{1}{0}{0}{\varpi_\mathfrak{p}})\Big)\phi .$$ From now on,
when the notation $ \phi^{\dagger_J}$ appears, we always means that
$\phi$ satisfies $\phi_p=\varphi_p$ and
$\phi_\infty=\varphi_\infty$.

Define the Atkin-Lehner element $\tau_v^{\mathfrak{n},B}$ of
$G(F_v)$ by
$$ \tau_v^{\mathfrak{n},B}= \left\{ \begin{array}{cl} I & \text{ if } v \text{ is archimedean or } v|\mathfrak{n}^-_b , \\
\wvec{0}{1}{-\varpi_v^{\mathrm{ord}_v \mathfrak{n}}}{0} & \text{ if
} v|\frac{\mathfrak{n}}{\mathfrak{n}^-_b}, \\ 1 & \text{ if } v\nmid
\mathfrak{n}.
\end{array}\right. $$

For $\phi_1, \phi_2 \in \CA(\pi')$, we define the
$G(\BA_F)$-equivalent pairing
$$ \langle \phi_1, \phi_2 \rangle_G := \int_{G(F)Z(\BA_F)\backslash G(\BA_F)} \phi_1(g)\phi_2(g) dg
$$ where $dg$ is the Tamagawa measure on $G/Z$. For each place $v$
as $\pi'_v$ is self-dual, there exists a non-degenerate pairing
$\langle \cdot,\cdot\rangle_v: \pi'_v\times \pi'_v\rightarrow \BC$.
The pairing is unique up to a nonzero scalar. We have
\begin{equation}\label{eq:pairing-loc-gl}
\langle \cdot,\cdot\rangle_G = C_{\pi'}
\bigotimes_v\langle\cdot,\cdot\rangle_v
\end{equation} for some constant $C_{\pi'}$ that only depends on $\pi$.

By Casselman's results \cite[Theorem 1]{Cass} 
we have
$$ \langle \varphi_{\pi'}, \pi'(\tau^{\mathfrak{n},B})\varphi_{\pi'}
\rangle_G\neq 0 $$ and
$$ \langle \varphi_v, \pi'(\tau^{\mathfrak{n},B}_v)\varphi_v
\rangle_v\neq 0 $$ for each $v$.


For $g\in G(F_v)$ and a character $\chi_v: K^\times_v\rightarrow
\BC^\times$ one defines the local toric integral for  $\phi_v$ by
$$ \mathcal{P}( g, \phi_v, \chi_v ) =
\frac{ L( 1,\mathrm{Ad} \pi_v ) L(1, \tau_{K_v/F_v})
}{\zeta_{F_v}(2)L(\frac{1}{2}, \pi_{K_v}\otimes \chi )}
\int_{K^\times_v/F^\times_v} \frac{ \langle \pi'(tg)
\phi^{\dagger_J}_v, \pi'(Ig)\phi^{\dagger_J}_v \rangle_v }{ \langle
\varphi_v, \pi'(\tau^{\mathfrak{n},B}_v)\varphi_v \rangle_v }
\chi_v(t) dt .
$$ 
Let $\chi: K^\times\BA_F^\times \backslash \BA_K^\times \rightarrow
\BC^\times$ be a Hecke character of archimedean weight
$(\mathbf{m},-\mathbf{m})$,
$\mathbf{m}=\sum\limits_{\sigma\in\Sigma_F}m_\sigma\sigma\in\BZ[\Sigma_F]$.

For $\phi\in \CA(G)$ and $g\in G(\BA_F)$ one defines the global
toric period integral by
$$ P(g,\phi,\chi):=\int_{K^\times\BA_F^\times\backslash \BA^\times_K} \phi(tg)\chi(t)
dt,
$$ where $dt$ is the Tamagawa measure.

\begin{prop} \label{prop:wald} For any $\phi$ such that
$\phi_p=\varphi_p$, we have \begin{eqnarray*}
\frac{P(\varsigma_J^{(\vec{n})},
\phi^{\dagger_J},\chi)^2}{\langle\varphi_{\pi'},
\pi'(\tau^{\mathfrak{n},B})\varphi_{\pi'}\rangle_G} &=&
\frac{\zeta_F(2)L(\frac{1}{2}, \pi_K\otimes \chi) }{2
L(1,\mathrm{Ad}\pi )} \cdot \prod_{\mathfrak{p}\in J}
\mathcal{P}(\varsigma_{\mathfrak{p}}^{(\vec{n})}, \phi_\mathfrak{p},
\chi_\mathfrak{p}) \\ &&\cdot \prod_{\mathfrak{p}|p,
\mathfrak{p}\notin J} \mathcal{P}(\varsigma_{\mathfrak{p}},
\phi_\mathfrak{p}, \chi_\mathfrak{p}) \cdot \prod_{v\nmid
p}\mathcal{P}(1, \phi_v, \chi_v).
\end{eqnarray*} \end{prop}
\begin{proof} Set
$\phi_1=\pi'(\varsigma_J^{(\vec{n})})\phi^{\dagger_J}$,
$\phi_2=\pi'(I\varsigma_J^{(\vec{n})})\phi^{\dagger_J}$,
$\phi_3=\varphi_{\pi'}$ and
$\phi_4=\pi'(\tau^{\mathfrak{n},B})\varphi_{\pi'}$. By Waldspurger's
formula (see \cite[Theorem 1.4]{YZZ}), noting that in \cite{YZZ}
$\{\langle\cdot,\cdot\rangle_v\}_v$ are taken such that
$C_{\pi'}=1$, we have
\begin{eqnarray*}\frac{P(1, \phi_1,\chi)P(1,\phi_2,
\chi^{-1})}{\langle \phi_3, \phi_4 \rangle_G} &=&
\frac{\zeta_F(2)L(\frac{1}{2}, \pi_K\otimes \chi)}{ 2
L(1,\mathrm{Ad}\pi)}
 \cdot \prod_{\mathfrak{p}\in J}
\mathcal{P}(\varsigma_{\mathfrak{p}}^{(\vec{n})}, \phi_\mathfrak{p},
\chi_\mathfrak{p})  \\ && \cdot \prod_{\mathfrak{p}|p,
\mathfrak{p}\notin J} \mathcal{P}(\varsigma_{\mathfrak{p}},
\phi_\mathfrak{p}, \chi_\mathfrak{p}) \cdot \prod_{v\nmid
p}\mathcal{P}(1, \phi_v, \chi_v). \end{eqnarray*} Our coefficient
$\frac{\zeta_F(2)L(\frac{1}{2}, \pi_K\otimes \chi) }{2
L(1,\mathrm{Ad}\pi)}$ is different to the coefficient
$\frac{\zeta_F(2)L(\frac{1}{2}, \pi_K\otimes \chi)
}{8L(1,\tau_{K/F})^2 L(1,\mathrm{Ad}\pi)}$ in \cite[Theorem
1.4]{YZZ} is due to the reason that our measure on
$K^\times\BA_F^\times\backslash \BA^\times_K$ is the Tamagawa
measure so that $\mathrm{vol}(K^\times\BA_F^\times\backslash
\BA^\times_K)=2L(1,\tau_{K/F})$, while in \cite{YZZ} one has
$\mathrm{vol}(K^\times\BA_F^\times\backslash \BA^\times_K)=1$.

Clearly $P(1, \phi_1,\chi)=P(\varsigma_J^{(\vec{n})},
\phi^{\dagger_J},\chi)$, and we have
\begin{eqnarray*} P(1,
\phi_2,  \chi^{-1}) & = & \int_{ K^\times\BA _F^\times \backslash
\BA ^\times_K }
\phi^{\dagger_J}(tI\varsigma^{(\vec{n})}_J)\chi(t^{-1}) dt =
\int_{K^\times\BA _F^\times\backslash \BA ^\times_K}
\phi^{\dagger_J}(I\bar{t}\varsigma^{(\vec{n})}_J)\chi(\bar{t}) dt \\
&=& \int_{K^\times\BA _F^\times\backslash \BA ^\times_K}
\phi^{\dagger_J}( t\varsigma_J^{(\vec{n})})\chi(t) dt
=P(\varsigma_J^{(\vec{n})}, \phi^{\dagger_J} ,\chi).
\end{eqnarray*} Here, $ \chi(t^{-1})
=\chi(\bar{t})$ since $\chi$ is anticyclotomic.
\end{proof}

We fix an anticyclotomic character $\nu$ of archimedean type
$(\mathbf{m},-\mathbf{m})$ and will apply Proposition
\ref{prop:wald} to the characters $\chi$ such that

$\bullet$ $\chi$ is of archimedean type $(\mathbf{m},-\mathbf{m})$;

$\bullet$ $\chi\nu^{-1}$ comes from a character of $\Gamma^-_J$.

\noindent Note that  the second condition implies that
$\chi\nu^{-1}$ is unramified outside of $J$.

We assume that $\nu$ is of conductor $\mathfrak{n}_\nu
\mathcal{O}_K$, where $\mathfrak{n}_\nu$ is an ideal of
$\mathcal{O}_F$ coprime to $p$. By technical reason we assume that
all prime factors of $\mathfrak{n}_\nu$ are inert or ramified in
$K$, and that $\mathfrak{n}_\nu$ is coprime to $\mathfrak{n}^-_b$.

\begin{lem} \label{lem:char} \begin{enumerate}
\item If $v\nmid \mathfrak{n}_\nu$ is inert in $K$, then $\nu_v=1$.
\item For each $v\notin J$ inert or ramified in $K$, we have ${\chi}_v={\nu}_v$.\end{enumerate}
\end{lem}
\begin{proof} When $v\nmid \mathfrak{n}_\nu$ is inert in $K$, $\nu_v$ is unramified. From $K_v^\times=F^\times_v\mathcal{O}^\times_{K_v}$ we obtain ${\nu}_v=1$. If $v\notin J$ is inert in $K$,
then ${\chi}_v\cdot {\nu}_v^{-1}$ is unramified and thus we must
have ${\chi}_v\cdot {\nu}_v^{-1}=1$.

If $v\notin J$ is ramified in $K$, as $[K_v^\times:F^\times_v\mathcal{O}^\times_{K_v}]=2$, we have either ${\chi}_v\cdot {\nu}_v^{-1}=1$ or ${\chi}_v\cdot {\nu}_v^{-1}$ is of order $2$. However, ${\chi}_v\cdot {\nu}_v^{-1}$ comes from a character of a pro-$p$ abelian group. Since $p>2$, we must have ${\chi}_v{\nu}^{-1}_v=1$.
\end{proof}

We compute the terms $\mathcal{P}(\varsigma_\mathfrak{p}, \phi_{\mathfrak{p}}, \chi_{\mathfrak{p}})$ ($\mathfrak{p}|p$) and $\mathcal{P}(1, \phi_v, \chi_v)$ ($v\nmid p$).

For each $\mathfrak{p}\in J$ put
\begin{eqnarray*} \bar{e}_{\mathfrak{p}}(\pi,\chi)
= \left\{ \begin{array}{ll} 1 & \text{ if } \chi \text{ is
ramified}; \\ (1-\alpha_\mathfrak{p}^{-1}\chi(\mathfrak{P}))
(1-\alpha_\mathfrak{p}^{-1}\chi(\bar{\mathfrak{P}})) & \text{ if }
\chi \text{ is unramified,} \
\mathfrak{p}=\mathfrak{P}\bar{\mathfrak{P}} \text{ is split}; \\
1-\alpha_\mathfrak{p}^{-2} & \text{ if } \chi \text{ is
unramified,}\
\mathfrak{p} \text{ is inert;} \\
1-\alpha_\mathfrak{p}^{-1}\chi(\mathfrak{P}) & \text{ if } \chi
\text{ is unramified} , \  \mathfrak{p}=\mathfrak{P}^2 \text{ is
ramified. }
\end{array}\right.
\end{eqnarray*}
If $\chi_\mathfrak{p}$ has conductor $\mathfrak{p}^s$,  we put
$$ \tilde{e}_{\mathfrak{p}}(\pi,\chi)= \bar{e}_{\mathfrak{p}}(\pi,\chi)^{2-\mathrm{ord}_\mathfrak{p}\mathfrak{n}}\cdot
\left\{ \begin{array}{ll} \alpha_\mathfrak{p}^2 |\mathfrak{p}|_{\mathfrak{p}}^2 &
\text{ if } s=0, \\ |\mathfrak{p}|_{\mathfrak{p}}^s & \text{ if }  s>0 .
\end{array}\right. $$ Then we define the $\mathfrak{p}$-adic
multiplier $e_\mathfrak{p}(\pi,\chi)$ by $$e_\mathfrak{p}(\pi,
\chi)=(\alpha_\mathfrak{p}^2|\mathfrak{p}|_{\mathfrak{p}})^{-s}\tilde{e}_{\mathfrak{p}}(\pi,\chi).$$

In \cite{Hung14} the $p$-adic multiplier $e_\mathfrak{p}(\pi,\chi)$
is defined to be our  $\bar{e}_\mathfrak{p}(\pi, \chi)$. But the
author think that it is better to define it as above since it
instead of $\bar{e}_\mathfrak{p}(\pi,\chi)$ appears in the
interpolation formula in Theorem \ref{thm:interpol}.

Let $d_{K_\mathfrak{p}}$ be a generator of $\mathcal{D}_{K_\mathfrak{p}}$. Note that $|d_{K_\mathfrak{p}}|_\mathfrak{p}=|d_{K_\mathfrak{p}}|_{\mathfrak{p}\mathcal{O}_K}^{1/2}$.

\begin{prop}$($\cite[Proposition 4.11]{Hung14}$)$ Let $\mathfrak{p}$
be in $J$. Suppose that $\chi_\mathfrak{p}$ has conductor
$\mathfrak{p}^s$. Put $n=\max\{ 1, s\}$. Then we have
$$ \mathcal{P}(\varsigma^{(n)}_{\mathfrak{p}}, \varphi_{\pi,\mathfrak{p}}, \chi_\mathfrak{p}) =
\tilde{e}_\mathfrak{p}(\pi,\chi) \cdot
\frac{1}{||\varphi||_{\mathfrak{p}}}\cdot \left\{\begin{array}{ll}
|d_{F_\mathfrak{p}}|_{\mathfrak{p}}\zeta_{F_\mathfrak{p}}(1)^2 &
\text{ if } \mathfrak{p} \text{ is split in } K,
\\
|d_{K_\mathfrak{p}}|_{\mathfrak{p}}
L(1,\tau_{K_\mathfrak{p}/F_\mathfrak{p}})^2 & \text{ if }
\mathfrak{p} \text{ is inert in } K.
\end{array}\right.
$$
\end{prop}

For $\mathfrak{p}\notin J$ above $p$, $\chi_\mathfrak{p}$ is unramified.

\begin{prop}\label{prop:xie} Let $\mathfrak{p}$
be a prime of $F$ above $p$, $\mathfrak{p}\notin J$.
\begin{enumerate}
\item\label{it:split-P} If $\mathfrak{p}$ is split in $K$, and if $\pi'_\mathfrak{p}$ is a
unramified principal series,
then
$$\mathcal{P}(\varsigma_{\mathfrak{p}},\varphi_\mathfrak{p}, \chi_\mathfrak{p})= |d_{F_\mathfrak{p}}|_{\mathfrak{p}} . $$
\item \label{it:split-S} If $\mathfrak{p}$ is split in $K$, and if $\pi'_\mathfrak{p}$ is a
unramified special representation,
then
$$\mathcal{P}(\varsigma_{\mathfrak{p}},\varphi_\mathfrak{p}, \chi_\mathfrak{p})= \frac{\epsilon(\frac{1}{2}, \pi_\mathfrak{p},\psi_\mathfrak{p})}{||\varphi_\mathfrak{p}||_\mathfrak{p}}
\chi_{\mathfrak{P}}(\omega_\mathfrak{P})|d_{F_\mathfrak{p}}|_{\mathfrak{p}} . $$
\item \label{it:unram-P} If $\mathfrak{p}$ is inert or ramified in $K$, and if $\pi'_\mathfrak{p}$ is a unramified principal series, then
$$ \mathcal{P}(\varsigma_\mathfrak{p},\varphi_\mathfrak{p}, \chi_\mathfrak{p})=  |d_{F_\mathfrak{p}}|_{\mathfrak{p}}^{-\frac{1}{2}}
|d_{K_\mathfrak{p}}|_{\mathfrak{p}}.
$$
\item \label{it:unram-S} If $\mathfrak{p}$ is inert in $K$, and if $\pi'_\mathfrak{p}=\sigma(\mu,\mu|\cdot|_\mathfrak{p}^{-1})$ is a unramified special representation, then
$$\mathcal{P}(\varsigma_\mathfrak{p},\varphi_\mathfrak{p}, \chi_\mathfrak{p})
       = \frac{|d_{K_\mathfrak{p}}|_\mathfrak{p}}{|d_{F_\mathfrak{p}}|_{\mathfrak{p}}^{1/2}}
\frac{|\omega_\mathfrak{p}|_{\mathfrak{p}}(\epsilon(\frac{1}{2},\pi_\mathfrak{p},\psi_\mathfrak{p})
       +\alpha_\mathfrak{p})}{1-|\omega_\mathfrak{p}|^2_{\mathfrak{p}}}, $$ where $\alpha_\mathfrak{p}$ is defined in the proof of Lemma \ref{lem:whittaker} $($\ref{it:special}$)$.
\item \label{it:ram} If $\mathfrak{p}$ is ramified in $K$, and if $\pi'_\mathfrak{p}$ is a unramified special representation, then
$$\mathcal{P}(\varsigma_\mathfrak{p},\varphi_\mathfrak{p}, \chi_\mathfrak{p})
       = \frac{|d_{K_\mathfrak{p}}|_{\mathfrak{p}}}{|d_{F_\mathfrak{p}}|_{\mathfrak{p}}^{1/2}}
\frac{2(1+|\omega_\mathfrak{p}|_{\mathfrak{p}})}{\epsilon(\frac{1}{2},\pi_\mathfrak{p},\psi_\mathfrak{p})
       }.$$
\end{enumerate}
\end{prop}
\begin{proof}
Assertions (\ref{it:split-P}) and (\ref{it:split-S}) follow from \cite[Proposition 4.6, 4.9]{Hung14}.
Assertions (\ref{it:unram-P}) and (\ref{it:ram}) follows from \cite[Proposition 3.8, 3.9]{Hsieh}.

We prove (\ref{it:unram-S}). By Lemma \ref{lem:char},
$\chi_\mathfrak{p}$ is trivial. Since $i_\mathfrak{p}(I)$ is in
$U_0(\omega_\mathfrak{p})$, we have
$$\mathbf{b}_\mathfrak{p}( \pi'_\mathfrak{p}(g)W_\mathfrak{p} , \pi'_\mathfrak{p}(I) W_\mathfrak{p})
= \mathbf{b}_\mathfrak{p}( \pi'_\mathfrak{p}(g)W_\mathfrak{p} ,
W_\mathfrak{p}).$$ Note that
$${m}(g)=\mathbf{b}_\mathfrak{p}(\pi'_\mathfrak{p}(g)W_\mathfrak{p},
W_\mathfrak{p}) $$ only depends on the double coset
$U_0(\mathfrak{p})gU_0(\mathfrak{p})$. For $g=1$ we have
$${m}(1)= |d_{F_\mathfrak{p}}|_{\mathfrak{p}}^{1/2}L(1, \mathrm{Ad} \pi_\mathfrak{p}). $$ For $g=\wvec{0}{1}{1}{0}$, by the fact $\pi'_\mathfrak{p}(\wvec{0}{1}{-\omega_\mathfrak{p}}{0})W_\mathfrak{p}=\epsilon(\frac{1}{2}, \pi_\mathfrak{p},\psi_\mathfrak{p})W_\mathfrak{p}$, we obtain
 \begin{eqnarray*} {m}(\wvec{0}{1}{1}{0}) &=& \epsilon(\frac{1}{2}, \pi'_\mathfrak{p},\psi_\mathfrak{p}) \mathbf{b}_\mathfrak{p}(\pi'_\mathfrak{p}(\wvec{-\omega_\mathfrak{p}}{0}{0}{1})W_\mathfrak{p}, W_\mathfrak{p}) \\ &=& \epsilon (\frac{1}{2},\pi'_\mathfrak{p},\psi_\mathfrak{p}) \mu(\omega_\mathfrak{p})|\omega_\mathfrak{p}|_{\mathfrak{p}}^{1/2}|d_{F_\mathfrak{p}}|_{\mathfrak{p}}^{1/2}L(1, \mathrm{Ad}\pi_\mathfrak{p}).  \end{eqnarray*}

We have the decomposition
$$ K_\mathfrak{p}^\times =
F_\mathfrak{p}^\times (1+\omega_\mathfrak{p}\mathcal{O}_{F_\mathfrak{p}}\vartheta)\sqcup F_\mathfrak{p}^\times (\mathcal{O}_{F_\mathfrak{p}}+\vartheta). $$ For any $a\in \omega_\mathfrak{p}\mathcal{O}_{F_\mathfrak{p}}$ we have $1+a\vartheta \in U_0(\mathfrak{p})$. For any
$a\in \mathcal{O}_{F_\mathfrak{p}}$, we have $a+\vartheta \in U_0(\mathfrak{p})\wvec{0}{1}{1}{0}U_0(\mathfrak{p})$. Now we have
\begin{eqnarray*}  \int_{K^\times_\mathfrak{p}/F^\times_\mathfrak{p}}  {m}(t)\chi_\mathfrak{p}(t) dt  &=&  {m}(1)\mathrm{vol}(1+\omega_\mathfrak{p}\mathcal{O}_{F_\mathfrak{p}}\vartheta) +
       {m}(\wvec{0}{1}{1}{0})\mathrm{vol}(\mathcal{O}_{F_\mathfrak{p}}+\vartheta)\\
       &=&  {m}(1) \frac{|\omega_\mathfrak{p}|_{\mathfrak{p}}}{1+|\omega_\mathfrak{p}|_{\mathfrak{p}}}
       \frac{|d_{K_\mathfrak{p}}|_{\mathfrak{p}}}{|d_{F_\mathfrak{p}}|_{\mathfrak{p}}^{1/2}} +  {m}(\wvec{0}{1}{1}{0}) \frac{1}{1+|\omega_\mathfrak{p}|_{\mathfrak{p}}}
       \frac{|d_{K_\mathfrak{p}}|_{\mathfrak{p}}}{|d_{F_\mathfrak{p}}|_{\mathfrak{p}}^{1/2}} \\
       &=& \frac{|d_{K_\mathfrak{p}}|_{\mathfrak{p}} L(1, \mathrm{Ad} \pi_\mathfrak{p})  }{1+|\omega_\mathfrak{p}|_{\mathfrak{p}}} [|\omega_\mathfrak{p}|_{\mathfrak{p}}
       +\epsilon(\frac{1}{2},\pi_\mathfrak{p},\psi_\mathfrak{p})\mu(\omega_\mathfrak{p})|\omega_\mathfrak{p}|
       _{\mathfrak{p}}^{1/2}],
\end{eqnarray*} and
$$\mathbf{b}_\mathfrak{p}(W_\mathfrak{p}, \wvec{0}{1}{-\omega_\mathfrak{p}}{0}W_\mathfrak{p})=\epsilon(\frac{1}{2},\pi_\mathfrak{p},\psi_\mathfrak{p})
|d_{F_\mathfrak{p}}|_{\mathfrak{p}}^{1/2}L(1,\mathrm{Ad}\pi_\mathfrak{p}).$$
When $\pi_\mathfrak{p}$ is unrmaified special, we have
$$ L(\frac{1}{2}, \pi_{K_\mathfrak{p}}\otimes \chi_\mathfrak{p})=L(\frac{1}{2}, \pi_{K_\mathfrak{p}}\otimes 1)=L(1, \mathrm{Ad} \pi_\mathfrak{p}) .$$ Therefore,
\begin{eqnarray*}\mathcal{P}(\varsigma_\mathfrak{p},\varphi_\mathfrak{p}, \chi_\mathfrak{p}) &=& \frac{L(1,\tau_{K_\mathfrak{p}/F_{\mathfrak{p}}})}{\zeta_{F_\mathfrak{p}}(2)} \frac{|d_{K_\mathfrak{p}}|_\mathfrak{p}}{|d_{F_\mathfrak{p}}|_{\mathfrak{p}}^{1/2}}
\frac{\epsilon(\frac{1}{2},\pi_\mathfrak{p},\psi_\mathfrak{p})|\omega_\mathfrak{p}|_{\mathfrak{p}}
       +\mu(\omega_\mathfrak{p})|\omega_\mathfrak{p}|
       _{\mathfrak{p}}^{1/2}}{1+|\omega_\mathfrak{p}|_{\mathfrak{p}}} \\
       &=& \frac{|d_{K_\mathfrak{p}}|_{\mathfrak{p}}}{|d_{F_\mathfrak{p}}|_{\mathfrak{p}}^{1/2}}
\frac{\epsilon(\frac{1}{2},\pi_\mathfrak{p},\psi_\mathfrak{p})|\omega_\mathfrak{p}|_{\mathfrak{p}}
       +\mu(\omega_\mathfrak{p})|\omega_\mathfrak{p}|
       _{\mathfrak{p}}^{1/2}}{1-|\omega_\mathfrak{p}|^2_{\mathfrak{p}}}, \end{eqnarray*} as expected.
\end{proof}

\begin{cor}\label{cor:P-nonzero}
Let $\mathfrak{p}$ be a prime of $\mathcal{O}_F$,
$\mathfrak{p}\notin J$. In the case that $\mathfrak{p}$ is inert in
$K$ and $\pi_\mathfrak{p}$ is unramified special, we assume that
$\alpha_\mathfrak{p}=-1$. Then
$\mathcal{P}(\varsigma_{\mathfrak{p}},\varphi_\mathfrak{p},
\chi_\mathfrak{p})\neq 0. $ \end{cor}
\begin{proof} The assertion for the case of $\mathfrak{p}$ split in $K$ and the case of $\pi$ unramified follows directly from Proposition \ref{prop:xie}. If $\mathfrak{p}$ is inert in $K$, and $\pi_\mathfrak{p}=\sigma(\mu,\mu|\cdot|_{\mathfrak{p}}^{-1})$ is unramified special, then
$\epsilon(\frac{1}{2},\pi_\mathfrak{p},\psi_\mathfrak{p})=-1$ (see
the proof of \cite[Theorem 4.12]{Hung14}). If
$\alpha_\mathfrak{p}\neq 1$, then by Proposition \ref{prop:xie}
(\ref{it:unram-S}) we have
$\mathcal{P}(\varsigma_{\mathfrak{p}},\varphi_\mathfrak{p},
\chi_\mathfrak{p})\neq 0.$ \end{proof}

Put $c_v=\ord_v \mathfrak{n}_\nu$.

\begin{prop}\label{prop:phi-old} $($\cite[Proposition 3.8, 3.9]{Hsieh}, \cite[Proposition 4.6, 4.9]{Hung14}$)$ Let $v$ be a finite place coprime to
$p\mathfrak{n}^-_b$, and satisfies the condition that if
$v|\frac{\mathfrak{n}^-}{\mathfrak{n}^-_b}$  and is inert in $K$,
then $v|\mathfrak{n}_\nu$.
\begin{enumerate}
\item\label{it:error} If $v$ is split in $K$ and $v\nmid \mathfrak{n}^+$, put $\phi_v=\varphi_v$. Then $$\mathcal{P}(1,\phi_v,
\chi_v)= \left\{\begin{array}{ll}   |d_{F_v}|_v  & \text{ if
}\mathrm{val}_v(\beta) \text{ is even} \\
|d_{F_v}|_v \cdot \chi_v((\omega_v,1))
 & \text{ if }\mathrm{val}_v(\beta) \text{ is odd}.
\end{array} \right.$$ Here $(\omega_v,1)$ is in $K_v\cong F_v\oplus
F_v$.
\item If $v|\mathfrak{n}^+$, writing $v=w\bar{w}$ with $w|\mathfrak{N}^+$, put
$\phi_v=\pi'_v(\wvec{\vartheta_w}{\vartheta_{\bar{w}}}{1}{1})\varphi_v$.
Then $$\mathcal{P}(1,\phi_v,
\chi_v)=\frac{\epsilon(\frac{1}{2},\pi_v,\psi_v)}{||\varphi_\pi||_v}\chi_w(\mathfrak{N}^+)|d_{F_v}|_v.$$
\item If $v$ is inert or ramified in $K$ and $v\nmid\mathfrak{n}^-$, put $\phi_v=\pi'_v(\wvec{\omega_v^{\ord_v\mathfrak{n}_\nu}}{0}{0}{1})\varphi_v$. Then
$$\mathcal{P}(1,\phi_v, \chi_v)=|d_{K_v}|_v|d_{F_v}|_v^{-1/2}\cdot\left\{
\begin{array}{ll} 1 & \text{ if }c_v=0, \\ L(1,\tau_{K_v/F_v})^2|\omega_v|_v^{c_v} & \text{ if }c_v>0.  \end{array}\right.$$
\item If $v|\frac{\mathfrak{n}^-}{\mathfrak{n}^-_b}$, put $\phi_v=\pi'_v(\wvec{\omega_v^{\ord_v\mathfrak{n}_\nu-\ord_v\mathfrak{n}}}{0}{0}{1})\varphi_v$. Then
\begin{eqnarray*} &&\mathcal{P}(1,\phi_v, \chi_v) \\ &=& \frac{1+|\omega_v|_v}{\epsilon(\frac{1}{2},\pi_v,\psi_v)}\frac{|d_{K_v}|_v}{|d_{F_v}|_v^{1/2}}
\cdot \left\{\begin{array}{ll} 2 & \text{if }c_v=0 \text{ and } \\ & v \text{ is ramified in }K, \\
L(1,\tau_{K_v/F_v})^2|\omega_v|_v^{c_v} & \text{if }c_v>0.\end{array}\right.
\end{eqnarray*}
\end{enumerate}
\end{prop}
\begin{proof} In the first case Assertion (\ref{it:error}) is proved in \cite[Proposition
4.6]{Hung14}.  In the second case, one only needs to revise the
argument in loc. cit. slightly.
\end{proof}

\begin{prop}\label{prop:phi-new} If $v|\frac{\mathfrak{n}^-}{\mathfrak{n}^-_b}$, $v\nmid \mathfrak{n}_\nu$, and $v$ is inert in $K$, put $\phi_v=\varphi_v$.  Then
$$\mathcal{P}(1,\phi_v, \chi_v)
       = \frac{|d_{K_v}|_{v}}{|d_{F_v}|_{v}^{1/2}}
\frac{|\omega_v|_{v}(\epsilon(\frac{1}{2},\pi_v,\psi_v)
       +\alpha_v)}{1-|\omega_v|^2_{v}}. $$
\end{prop}
\begin{proof} The argument is the same as the proof of Proposition \ref{prop:xie} (\ref{it:unram-S}).
\end{proof}

Now we consider the case of $v|\mathfrak{n}^-_b$. Note that $\pi'_v$ is the representation $\xi\circ N_{B_v/F_v}$ where $N_{B_v/F_v}$ is the reduced norm, and $\xi$ is an unramified character of $F_\mathfrak{p}^\times$ such that $\xi(\omega_\mathfrak{p})=\alpha_v$.

\begin{prop}\label{prop:nb} If $v|\mathfrak{n}^-_b$, take $\phi_v=\varphi_v$.
\begin{enumerate}
\item If $v$ is inert in $K$, then
$$ \mathcal{P}(1,\phi_v, \chi_v) =\frac{1}{\zeta_{F_v}(1)}\cdot\frac{|d_{K_v}|_v}{|d_{F_v}|_v^{1/2}}. $$
\item In the case of $v$ ramified in $K$, let $\tilde{\omega}_v$ be a unifomrmizing element in $K_v$. If $\nu_v(\tilde{\omega}_v)=\alpha_v$, then
$$\mathcal{P}(1,\phi_v, \chi_v) =\frac{2}{\zeta_{F_v}(1)}\cdot\frac{|d_{K_v}|_v}{|d_{F_v}|_v^{1/2}}. $$ If
$\nu_v(\tilde{\omega}_v)=-\alpha_v$, then $\mathcal{P}(1,\phi_v,\chi_v)=0$.
\end{enumerate}
\end{prop}
\begin{proof}
Note that, if $\pi_v=\sigma(\mu, \mu|\cdot|_v^{-1})$, then $\pi'_v$ is the representation $\xi\circ N_{B_v/F_v}$ where $N_{B_v/F_v}$ is the reduced norm, and $\xi=\mu|\cdot|_v^{-1/2}$. In particular,  $\xi(\omega_v)=\alpha_v$.

In the case when $v$ is inert in $K$, as $\chi_v$ is unramified, $\chi_v=1$. Using $L(1, \mathrm{Ad}\pi_v)=L(\frac{1}{2}, \pi_{K_v}\otimes 1)$ we obtain
$$ \mathcal{P}(1,\phi_v, \chi_v) =\mathcal{P}(1,\phi_v, 1) = \frac{ L(1, \mathrm{Ad} \pi_v) L(1, \tau_{K_v/F_v})
}{\zeta_{F_v}(2)L(\frac{1}{2}, \pi_{K_v}\otimes 1)} \mathrm{vol}(F_v^\times\backslash K_v^\times) = \frac{1}{\zeta_{F_v}(1)}\cdot\frac{|d_{K_v}|_v}{|d_{F_v}|_v^{1/2}}.$$

In the case when $v$ is ramified in $K$, as $\chi_v$ is unramified,
we have $\chi_v(\tilde{\omega}_v^2)=\chi_v(\omega_v)=1$. Thus
$\chi_v(\tilde{\omega}_v)=\pm 1=\pm \alpha_v$. In this case, $L(1,
\mathrm{Ad}
\pi_v)=\frac{1}{1-\mu(\omega_v)^2|\omega_v|_v}=\zeta_{F_v}(2)$ and
$L(1,\tau_{K_v/F_v})=1$, and $L(\frac{1}{2}, \pi_{K_v}\otimes
\chi_v)=\frac{1}{1-\xi(\omega_v)\chi_v(\tilde{\omega}_v)|\omega_v|_v}$.

We have the decomposition $F_v^\times\backslash K^\times_v= \mathcal{O}^\times_{F_v}\backslash\mathcal{O}^\times_{K_v}\sqcup \mathcal{O}^\times_{F_v}\backslash\mathcal{O}^\times_{K_v} \tilde{\omega}_v$. Thus
 \begin{eqnarray*}\mathcal{P}(1,\phi_v, \chi_v)  &=& (1-\xi(\omega_v)\chi_v(\tilde{\omega}_v)|\omega_v|_v) \mathrm{vol}(\mathcal{O}_{F_v}^\times\backslash \mathcal{O}_{K_v}^\times)(1+\chi_v(\tilde{\omega}_v)\xi(\omega_v)) \\ &=&
\left\{\begin{array}{ll} \frac{2}{\zeta_{F_v}(1)}\frac{|d_{K_v}|_v}{|d_{F_v}|_v^{1/2}} & \text{ if } \chi_v(\tilde{\omega}_v)=\xi_v(\omega_v)=\alpha_v \\
0 & \text{ if }\chi_v(\tilde{\omega}_v)=-\xi_v(\omega_v)=-\alpha_v.\end{array}\right.\end{eqnarray*} By Lemma \ref{lem:char}, $\chi_v=\nu_v$. We obtain the assertion in the ramified case.
\end{proof}

\begin{prop}\label{prop:cal-P} For $\sigma\in \Sigma$ we have $$ \mathcal{P}(1, v_{m_\sigma}, \chi_\sigma)=\frac{\Gamma(k_\sigma)}{\pi \Gamma(k_\sigma/2+m_\sigma)\Gamma(k_\sigma/2-m_\sigma)}. $$
\end{prop}
\begin{proof} See the proof of \cite[Theorem 4.12]{Hung14}.
\end{proof}

We make the following assumption on $\pi$, $\nu$ and $B$.

(i) If $\mathfrak{p}\notin J$ is inert in $K$, and if $\pi_{\mathfrak{p}}$ is unramified special, then $\alpha_\mathfrak{p}=-1$. See Remark \ref{rem:alpha}.

(ii) All prime factors of $\mathfrak{n}_\nu$ is inert or ramified in $K$, and $(\mathfrak{n}_\nu,\mathfrak{n}^-_b)=1$.

(iii) If $v|\frac{\mathfrak{n}^-}{\mathfrak{n}^-_b}$ and $v$ is inert in $K$ (in this case $\pi_v$ must be unramified special), then either $v|\mathfrak{n}_\nu$ or $\alpha_v=-1$.

(iv) If $v|\mathfrak{n}^-_b$ is ramified in $K$, then $\nu_v(\tilde{\omega}_v)=\alpha_v$.

\begin{rem} If $\mathfrak{n}^-_b$ is the product of the prime factors of $\mathfrak{n}^-$ that are inert in $K$, then (iii) and (iv) automatically hold.
\end{rem}

\begin{rem}\label{rem:weaken} In \cite{Hung14} it demands (ii), (iv) and a stronger version of (iii) that  if $v|\frac{\mathfrak{n}^-}{\mathfrak{n}^-_b}$,
then $v|\mathfrak{n}_\nu$. These conditions are assumed or implicitly assumed in \cite[Theorem 4.12]{Hung14}.
So we weaken Hung's conditions on $\nu$. It is not harm for us to
assume (iv), since in the branches with
$\nu_v(\tilde{\omega}_v)=-\alpha_v$ for some $v|\mathfrak{n}^-_b$
ramified in $K$, the $p$-adic $L$-functions we will construct are
identically zero.
\end{rem}

Now, we fix our choice of $\phi$. If $v|p$, we take
$\phi_v=\varphi_v$. At other places let $\phi_v$ be as defined in Propositions \ref{prop:phi-old}, \ref{prop:phi-new} and \ref{prop:nb}. \label{ss:phi}

Put \begin{eqnarray*} \widetilde{ \Omega}^-_{J,\phi} &=& \frac{2
L(1, \mathrm{Ad}\pi)}{\zeta_F(2)\langle \varphi_{\pi'},
\pi'(\tau^{\mathfrak{n},B})\varphi_{\pi'}\rangle_{G}}\cdot
\prod_{\sigma\in\Sigma}\frac{1}{\mathcal{P}(1,v_{m_\sigma},
\nu_\sigma)}
 \\ && \cdot \prod_{\mathfrak{p}\in
J: \
\text{inert}}\frac{||\varphi_\mathfrak{p}||_{\mathfrak{p}}}{|d_{K_\mathfrak{p}}|_{\mathfrak{p}}L(1,\tau_{K_\mathfrak{p}/F_{\mathfrak{p}}})^2}
\cdot \prod_{\mathfrak{p}\in J: \
\text{split}}\frac{||\varphi_\mathfrak{p}||_{\mathfrak{p}}}{|d_{F_\mathfrak{p}}|_{\mathfrak{p}}\zeta_{F_\mathfrak{p}}(1)^2}
 \\ && \cdot
\prod_{\begin{array}{c}\mathfrak{p}|p,\mathfrak{p}\notin J: \\
\mathfrak{p} \text{ is split and } r_\mathfrak{p}=1\end{array}}
\frac{||\varphi_v||_{\mathfrak{p}}}{\epsilon(\frac{1}{2},\pi_v,\psi_v)|d_{F_\mathfrak{p}}|_{\mathfrak{p}}}
\cdot \prod_{\begin{array}{c}\mathfrak{p}|p, \mathfrak{p}\notin J:
\\ \mathfrak{p} \text{ is not split  or
}r_\mathfrak{p}=0\end{array}}
\frac{1}{\mathcal{P}(\varsigma_\mathfrak{p},\varphi_\mathfrak{p},
1)}
\\ && \cdot
\prod_{v|\mathfrak{n}^+}\frac{||\varphi_v||_{v}}{\epsilon(\frac{1}{2},\pi_v,\psi_v)|d_{F_v}|_{v}}
\cdot \prod_{\begin{array}{c}v \nmid p \mathfrak{n}^+ : \\ \text{
finite and not split} \end{array}}\frac{1}{\mathcal{P}(1,\phi_v,
\nu_v)} \cdot \prod_{\begin{array}{c}v \nmid p \mathfrak{n}^+ : \\
\text{ finite and split} \end{array}} \frac{1}{|d_{F_v}|_v}
.\end{eqnarray*} This makes sense since almost all factors in the
product are $1$; $\widetilde{\Omega}^-_{J,\phi}$ depends on $J$,
$\pi'$ and $\nu$, but not on $\chi$;
$\widetilde{\Omega}^-_{J,\phi}\neq 0$.

We can state Proposition \ref{prop:wald} as
follows.
\begin{cor}\label{cor:spe-value}
Let $\pi$ and $\chi$ be as above. Put $$ \chi(\mathfrak{N}^+) =
\prod\limits_{w|\mathfrak{N}^+}\chi_v(\omega_w)^{\ord_w\mathfrak{N}^+}.
$$ Then
$$P(\varsigma_J^{(n)}, \phi^{\dagger_J}, \chi)^2
=
\chi(\mathfrak{N}^+)(\prod_{v}\chi_v((\omega_v,1))(\prod_{\mathfrak{p}\in
J} \tilde{e}_{\mathfrak{p}}(\pi,\chi)  ) \cdot \frac{L(\frac{1}{2},
\pi_K\otimes \chi)}{\widetilde{\Omega}^-_{J,\phi}} ,$$ where $v$
runs over the set
\begin{equation}\label{eq:set}
\{ v|p , v \notin J :  v \text{ is split in } K \text{ and } r_v = 1
\} \cup \{ v \text{ is finite and split in }K: \mathrm{val}_v(I)
\text{ is odd} \}. \end{equation}
\end{cor}

\section{Anticyclotomic $p$-adic
$L$-functions} \label{sec:p-L-function}

Fix $\mathbf{m}=\sum\limits_{\sigma\in \Sigma_K}m_{\sigma}\sigma$,
and a continuous $p$-adic character $\widehat{\nu}$ of
$\mathrm{Gal}(\bar{K}/K)$ of type $(\mathbf{m},-\mathbf{m})$. Let
$\nu$ be the complex Hecke character attached to $\widehat{\nu}$; we
assume that $\nu$ is of conductor
$\mathfrak{c}=\mathfrak{n}_\nu\mathcal{O}_K$ ($(\mathfrak{n}_\nu,
p)=1$) and that $\nu_p=1$. We keep the assumption on $\pi$ and $\nu$
at the end of Section \ref{ss:phi}.

\begin{defn} For a continuous $p$-adic character $\widehat{\chi}$ of $\mathrm{Gal}(\bar{K}/K)$, we say
that $\widehat{\chi}$ lies in the {\it $J$-branch of
$\widehat{\nu}$} if $\widehat{\chi}\cdot \widehat{\nu}^{-1}$ comes
from a character of $\Gamma^-_J$.
\end{defn}

Write $$ U^p=\prod_{v: \text{ finite and coprime to }
p}\mathcal{O}_{K_v}^\times. $$ For an ideal $\mathfrak{a}$ of
$\mathcal{O}_K$ (coprime to $p$) let $U^p_{\mathfrak{a}}$ be the
subgroup
$$U^p_{\mathfrak{a}}:=\{x=(x_v)_v \in \BA_K^{\infty,p} :  x_v
\in \widehat{\mathcal{O}}_{K_v}^\times \cap 1+ \mathfrak{a}
\widehat{\mathcal{O}}_{K_v} \}$$ of $U^p$.

We take $\phi$ as in Section \ref{ss:phi}.
There exists an ideal $\tilde{\mathfrak{c}}$ of $\mathcal{O}_K$
coprime to $p$ such that $\phi^\infty$ is
$U^{p}_{\tilde{\mathfrak{c}}}$-invariant. Sharking
$\tilde{\mathfrak{c}}$ if necessary we may assume that
$\mathfrak{c}|\tilde{\mathfrak{c}}$. We put
$$\tilde{\phi}^\infty = \frac{1}{[U^p_{\mathfrak{c}}:U^p_{\widetilde{\mathfrak{c}}}]} \sum_{g\in
U^p_{\mathfrak{c}}/U^p_{\widetilde{\mathfrak{c}}}} g \phi^\infty,
$$ and $\tilde{\phi} =\Psi(\mathbf{v}_{\mathbf{m}}\otimes
\tilde{\phi}^\infty)$. Then $\tilde{\phi}^\infty$, $\tilde{\phi}$
and $\tilde{\phi}^{\dagger_J}$ are $U^{p}_{\mathfrak{c}}$-invariant.
\label{ss:tilde-phi}

\begin{lem}\label{lem:theta} The function $$ a \mapsto \tilde{\phi}^{\dagger_J}
(a\varsigma_J^{(\vec{n})})\nu(a) $$ on $\BA_K^\times$ is
$\BA_F^\times K^\times K_\infty^\times
\widehat{\mathcal{O}}_{\vec{n},\mathfrak{c}}^\times$-invariant.
\end{lem}
\begin{proof} As $K_\infty^\times$ acts on $\mathbf{v}_{m}$ by
$\nu_\infty^{-1}$,  we obtain that $a\mapsto \tilde{\phi}^{\dagger_J}
(a\varsigma_J^{(\vec{n})}) \nu(a)$ is $K_\infty^\times$-invariant.

Both $a \mapsto \tilde{\phi}^{\dagger_J} (a\varsigma_J^{(\vec{n})})$
and $\nu$ are $\BA _F^\times$-invariant and
$K^\times\widehat{\mathcal{O}}_{\vec{n},\mathfrak{c}}^\times$-invariant.
Thus the function $a\mapsto \tilde{\phi}^{\dagger_J}
(a\varsigma_J^{(\vec{n})}) \nu(a)$ is $\BA _F^\times
K^\times\widehat{\mathcal{O}}_{\vec{n},\mathfrak{c}}^\times$-invariant.
\end{proof}

For any $J$-tuple $\vec{n}=(n_\mathfrak{p})_{\mathfrak{p}\in J}$ let
$[ \hskip 10pt ]_{\vec{n},\mathfrak{c}}$ be the composition of maps
$$ \BA_{K}^{\infty,\times}/K^\times \xrightarrow{\mathrm{rec}_K} \mathrm{Gal}(\overline{K}/K) \rightarrow \mathcal{G}_{ \vec{n},\mathfrak{c}}.
$$ Note that $[ \hskip 10pt ]_{ \vec{n},\mathfrak{c}}$ factors through $X_{\vec{n},\mathfrak{c}}:=\BA _F^{\infty,\times} \backslash
\BA _K^{\infty,\times} / K^\times \widehat{\mathcal{O}}_{\vec{n},
\mathfrak{c}} $. We define the theta element $\Theta_{
\vec{n},\mathfrak{c}}$ by
$$ \Theta_{ \vec{n},\mathfrak{c}}
 = \frac{1}{\prod_{\mathfrak{p}\in J}\alpha_\mathfrak{p}^{n_\mathfrak{p}}}
\cdot   \sum_{ a \in X_{\vec{n},\mathfrak{c}} }
\tilde{\phi}^{\dagger_J} (a\varsigma_J^{(\vec{n})})\nu(a)
[a]_{\vec{n},\mathfrak{c}}.
$$

For two $J$-tuples $\vec{n}$ and $\vec{n}'$ we say that
$\vec{n}'\geq \vec{n}$ if $n'_{\mathfrak{p}}\geq n_{\mathfrak{p}}$
for each $\mathfrak{p}\in J$. When $\vec{n}'\geq \vec{n}$ we have
the natural quotient map $\pi_{\vec{n}',\vec{n}}: \mathcal{G}_{
\vec{n}',\mathfrak{c}}\rightarrow
\mathcal{G}_{\vec{n},\mathfrak{c}}$.

The theta elements $\Theta_{\vec{n},\mathfrak{c}}$ ($\vec{n}$
varying) have the following compatible property.

\begin{prop}\label{prop:comp} If $\vec{n}'\geq
\vec{n}$, then
$$ \pi_{\vec{n}',\vec{n}} (\Theta_{ \vec{n}',\mathfrak{c}})=\Theta_{ \vec{n},\mathfrak{c}} . $$
\end{prop}
\begin{proof} We easily reduce to the case that there exists exactly
one $\mathfrak{p}\in J$ such that
$n'_\mathfrak{p}=n_{\mathfrak{p}}+1$ and
$n'_\mathfrak{q}=n_{\mathfrak{q}} $ for $\mathfrak{q}\in J, \
\mathfrak{q}\neq \mathfrak{p}$. In this case our assertion follows
from the relation
$\tilde{\phi}^{\dagger_{J}}|_{U_{\mathfrak{p}}}=\alpha_\mathfrak{p}\tilde{\phi}^{\dagger_{J}}$
and the fact that
$$ (\tilde{\phi}^{\dagger_{J}}|_{U_{\mathfrak{p}}})(a\varsigma_\mathfrak{p}^{(n_\mathfrak{p})})
 = \sum_u \tilde{\phi}^{\dagger_{J}} (au \varsigma_\mathfrak{p}^{(n_\mathfrak{p}+1)})
 $$ where $u$ runs through the coset
$[(1+\mathfrak{p}^{n_{\mathfrak{p}}}\mathcal{O}_{K_\mathfrak{p}})\mathcal{O}^\times_{F_\mathfrak{p}}]
 /[(1+\mathfrak{p}^{n_{\mathfrak{p}}+1}\mathcal{O}_{K_\mathfrak{p}})\mathcal{O}_{F_\mathfrak{p}}^\times
 ]$.
\end{proof}

\begin{prop}\label{prop:theta} If $\widehat{\chi}$ is in the $J$-branch of
$\widehat{\nu}$, and if $\widehat{\chi}\widehat{\nu}^{-1}$ comes
from a character of $\mathcal{G}_{\vec{n},\mathfrak{c}}$, then
$$ \frac{1}{\prod_{\mathfrak{p}\in
J}\alpha_\mathfrak{p}^{n_\mathfrak{p}}} P(\varsigma^{(\vec{n})}_J,
\phi^{\dagger_J},  \chi) =
\mathrm{vol}(\widehat{\mathcal{O}}_{\vec{n},\mathfrak{c}}^\times
/(\widehat{\mathcal{O}}_{\vec{n},\mathfrak{c}}^\times\cap
K^\times)\widehat{\mathcal{O}}^\times_F)
\Theta_{\vec{n},\mathfrak{c}}(\jmath^{-1}(\widehat{\chi}\widehat{\nu}^{-1})
) ,
$$ where $\chi$ is the complex Hecke character attached to
$\widehat{\chi}$.
\end{prop}
\begin{proof}  As
$\chi\nu^{-1}=\jmath^{-1}(\widehat{\chi}\widehat{\nu}^{-1})$ comes
from a character of $\mathcal{G}_{\vec{n},\mathfrak{c}}$, by Lemma
\ref{lem:theta} the function
$$t\mapsto \tilde{\phi}^{\dagger_J} (t\varsigma_J^{(\vec{n})})
\chi(t)$$ is $\BA _F^\times K^\times K_\infty^\times
\widehat{\mathcal{O}}_{\vec{n}, \mathfrak{c}}^\times$-invariant.
Hence, \begin{eqnarray*}
P(\varsigma_J^{(\vec{n})},\phi^{\dagger_J},\chi)&=&\int_{K^\times\BA_F^\times\backslash
\BA^\times_K} \phi^{\dagger_J}(t \varsigma_J^{(\vec{n})})\chi(t) dt
\\ & = &
\mathrm{vol}(\widehat{\mathcal{O}}_{\vec{n},\mathfrak{c}}^\times/(\widehat{\mathcal{O}}_{\vec{n},\mathfrak{c}}^\times\cap K^\times )\widehat{\mathcal{O}}^\times_{F}) \cdot \sum_{a\in \BA _F^{\infty,\times} \backslash
\BA _K^{\infty,\times} / K^\times
\widehat{\mathcal{O}}_{\vec{n}, \mathfrak{c}}^\times }
\tilde{\phi}^{\dagger_J}
(a\varsigma_J^{(\vec{n})})\chi(a)  \\
&=&
\mathrm{vol}(\widehat{\mathcal{O}}_{\vec{n},\mathfrak{c}}^\times/(\widehat{\mathcal{O}}_{\vec{n},\mathfrak{c}}^\times\cap
K^\times)\widehat{\mathcal{O}}^\times_{F} ) \cdot \sum_{a\in \BA
_F^{\infty,\times} \backslash \BA _K^{\infty,\times} / K^\times
\widehat{\mathcal{O}}_{\vec{n}, \mathfrak{c}}^\times }
\tilde{\phi}^{\dagger_J} (a\varsigma_J^{(\vec{n})})\nu(a)
\chi\nu^{-1}([a]_{J,\vec{n}}),
\end{eqnarray*}
as wanted.
\end{proof}

 By abuse of notation we use $\mathrm{N}\mathfrak{c}$ to denote the positive integer generating the ideal $\mathrm{N}_{F/\BQ}(\mathfrak{c})$.
 \begin{lem}\label{lem:volume} We have
\begin{eqnarray*}\mathrm{vol}(\widehat{\mathcal{O}}^\times_{\vec{n},\mathfrak{c}}
 /(\widehat{\mathcal{O}}^\times_{\vec{n},\mathfrak{c}}\cap K^\times )\widehat{\mathcal{O}}^\times_{F})
& =& \frac{1}{[\mathcal{O}^\times_K\cap
\widehat{\mathcal{O}}^\times_{\vec{n},\mathfrak{c}}
:\mathcal{O}^\times_F]}
 \frac{|N_{F/\BQ}\mathcal{D}_F|_\infty^{\frac{1}{2}}}{|N_{K/\BQ}\mathcal{D}_K|_\infty^{\frac{1}{2}}}
 \frac{1}{\mathbf{N}\mathfrak{c}
 \cdot\prod_{\mathfrak{p}\in J}|\omega_\mathfrak{p}|_{\mathfrak{p}}^{-n_\mathfrak{p}}}  \\
 &&\cdot   \prod\limits_{v\in J\text{ or }v|\mathfrak{c}: v \text{ inert }}L(1,\tau_{K_v/F_v})
  \cdot   \prod\limits_{v\in J\text{ or }v|\mathfrak{c}: v \text{ split }}\zeta_{F_v}(1)  .
  \end{eqnarray*}
 \end{lem}
 \begin{proof} By (\ref{eq:local-vol}) we have $$\mathrm{vol}(\widehat{\mathcal{O}}^\times_{K}/\widehat{\mathcal{O}}^\times_{F})
=\prod_{v: \text{ finite}}\mathrm{vol}
(\mathcal{O}_{K_v}^\times/\mathcal{O}_{F_v}^\times)=\frac{|N_{F/\BQ}\mathcal{D}_F|_\infty^{\frac{1}{2}}}{|N_{K/\BQ}\mathcal{D}_K|_\infty^{\frac{1}{2}}}
. $$ Note that
$$[\widehat{\mathcal{O}}^\times_K:\widehat{\mathcal{O}}_{\vec{n},\mathfrak{c}}]=
\prod\limits_{v\in J\text{ or }v|\mathfrak{c}: v \text{ inert
}}L(1,\tau_{K_v/F_v})^{-1} \cdot   \prod\limits_{v\in J\text{ or
}v|\mathfrak{c}: v \text{ split
}}\zeta_{F_v}(1)^{-1}\cdot\mathbf{N}\mathfrak{c}
\cdot\prod_{\mathfrak{p}\in
J}|\omega_\mathfrak{p}|_{\mathfrak{p}}^{-n_\mathfrak{p}}.$$ So
$$\mathrm{vol}(\widehat{\mathcal{O}}^\times_{\vec{n},\mathfrak{c}}/\widehat{\mathcal{O}}^\times_{F})
=\frac{|N_{F/\BQ}\mathcal{D}_F|_\infty^{\frac{1}{2}}}{|N_{K/\BQ}\mathcal{D}_K|_\infty^{\frac{1}{2}}}
\frac{1}{\mathbf{N}\mathfrak{c} \cdot\prod_{\mathfrak{p}\in
J}|\omega_\mathfrak{p}|_{\mathfrak{p}}^{-n_\mathfrak{p}}} \cdot
\prod\limits_{v\in J\text{ or }v|\mathfrak{c}: v \text{ inert
}}L(1,\tau_{K_v/F_v})
  \cdot   \prod\limits_{v\in J\text{ or }v|\mathfrak{c}: v \text{ split }}\zeta_{F_v}(1). $$ Our
lemma follows.
 \end{proof}

Put $\widehat{\mathbf{v}}_{\mathbf{m}}= \jmath
(\mathbf{v}_{\mathbf{m}})$.

\begin{lem}\label{lem:shift-inft-p}
For any $a\in \BA _K^\times$ we have
$$\jmath[\tilde{\phi}^{\dagger_J} (a\varsigma_J^{(\vec{n})})\nu(a)] =
\langle  \rho_{\vec{k}}
(\mathfrak{i}_p(\varsigma_J^{(\vec{n})}))^{-1}
\widehat{\mathbf{v}}_{\mathbf{m}}  ,
\widehat{\tilde{\phi}^{\infty,\dagger_J}} (a\varsigma_J^{(\vec{n})})
\rangle \cdot \widehat{\nu}(a) .
$$
\end{lem}
\begin{proof} Both sides of the equality are
$K_\infty^\times$-invariant. So we may assume $a=(a_p,a^p)\in
\BA _K^{\infty, \times}$. By definition of $\widehat{
\tilde{\phi}^{\infty,\dagger_J} }$ we have
$$ \jmath (\tilde{\phi}^{\infty,\dagger_J}(a\varsigma_J^{(\vec{n})}) )=
\check{\rho}_{\vec{k}} (\mathfrak{i}_p(a_p\varsigma_J^{(\vec{n})}))
\widehat{ \tilde{\phi}^ { \infty,\dagger_J }
}(a\varsigma_J^{(\vec{n})}).
$$ Thus {\allowdisplaybreaks\begin{eqnarray*} \jmath(\tilde{\phi}^{\dagger_J} (a\varsigma_J^{(\vec{n})}))
&=&\langle  \widehat{\mathbf{v}}_{\mathbf{m}},
\check{\rho}_{\vec{k}} (\mathfrak{i}_p(a_p\varsigma_J^{(\vec{n})}))
\tilde{\phi}^{\infty,\dagger_J} (a\varsigma_J^{(\vec{n})}) \rangle
\\ &=& \langle \rho_{\vec{k}}
(\mathfrak{i}_p(a_p\varsigma_J^{(\vec{n})}))^{-1}
\widehat{\mathbf{v}}_{\mathbf{m}}, \tilde{\phi}^{\infty,\dagger_J}
(a\varsigma_J^{(\vec{n})}) \rangle
\\ &=& \prod_{\sigma\in
\Sigma_K}\left(\frac{\sigma' a_{\mathfrak{P}_\sigma}}{\bar{\sigma}'
a_{\bar{\mathfrak{P}}_\sigma}}\right)^{m_\sigma} \cdot \langle
\rho_{\vec{k}} (\mathfrak{i}_p(\varsigma_J^{(\vec{n})}))^{-1}
\widehat{\mathbf{v}}_{\mathbf{m}}, \tilde{\phi}^{\infty,\dagger_J}
(a\varsigma_J^{(\vec{n})}) \rangle,
\end{eqnarray*}}

\noindent where the last equality follows from the fact $$
\rho_{\vec{k}}( \mathfrak{i}_p(a_p)^{-1})
\widehat{\mathbf{v}}_{\mathbf{m}} = \prod_{\sigma\in
\Sigma_K}\left(\frac{\sigma' a_{\mathfrak{P}_\sigma}}{\bar{\sigma}'
a_{\bar{\mathfrak{P}}_\sigma}}\right)^{m_\sigma}
\widehat{\mathbf{v}}_{\mathbf{m}}. $$ Now, our lemma follows from
the relation
$$ \widehat{\nu}(a) = \prod_{\sigma\in \Sigma_K}\left(\frac{\sigma'
a_{\mathfrak{P}_\sigma}}{\bar{\sigma}'
a_{\bar{\mathfrak{P}}_\sigma}}\right)^{m_\sigma} \cdot
\jmath(\nu(a)) $$ for $a\in \BA _K^{\infty,\times}$.
\end{proof}

Put $\widehat{\Theta}_{ \vec{n},\mathfrak{c}}:=\jmath \Theta_{
\vec{n},\mathfrak{c}}$. The following is a direct consequence of
Lemma \ref{lem:shift-inft-p}.

\begin{cor}\label{cor:theta} We have
$$ \widehat{\Theta}_{\vec{n},\mathfrak{c}}=
\frac{1}{\jmath(\prod_{\mathfrak{p}\in
J}\alpha_\mathfrak{p}^{n_\mathfrak{p}})}
 \cdot \sum_{a\in X_{\vec{n},
\mathfrak{c}} } \langle \rho_{\vec{k}}
(\mathfrak{i}_p(\varsigma_J^{(\vec{n})}))^{-1}
\widehat{\mathbf{v}}_{\mathbf{m}}  ,
\widehat{\tilde{\phi}^{\infty,\dagger_J}} (a\varsigma_J^{(\vec{n})})
\rangle \widehat{\nu}(a) [a]_{\vec{n},\mathfrak{c}}. $$
\end{cor}

Let $\mathfrak{o}_{V_{\vec{k}}}$ be the
$\mathcal{O}_{\BC_p}$-lattice in $V_{\vec{k}}(\BC_p)$ generated by
$\{\widehat{\mathbf{v}}_{\vec{s}}:\frac{k_\sigma-2}{2}\leq
s_\sigma\leq \frac{k_\sigma-2}{2}\}$.

\begin{lem}\label{lem:control-2} There is a constant $c\in \BC_p^\times$ such that
$$\rho_{\vec{k}}(\mathfrak{i}_p(\varsigma_J^{(\vec{n})}))^{-1}(\mathfrak{o}_{ V_{\vec{k}} })
\subseteq c \left( \prod_{\mathfrak{p}\in J}
\mathfrak{p}^{n_{\mathfrak{p}}\sum_{\sigma\in\Sigma_\mathfrak{p}}\frac{k_\sigma-2}{2}}
\right) \mathfrak{o}_{ V_{\vec{k}} } . $$
\end{lem}
\begin{proof}
Let $\mathfrak{o}_\sigma$ be the lattice in $\jmath V_{k_\sigma}$
generated by $\{\jmath(v_{s}):\frac{k_\sigma-2}{2}\leq s \leq
\frac{k_\sigma-2}{2}\}$. It is easy to check that if
$\mathfrak{p}_\sigma\in J$, then
$\varsigma_{\mathfrak{p}_\sigma}^{(n_\mathfrak{p})} $ sends
$\mathfrak{o}_\sigma$ to
$\mathfrak{p}_\sigma^{n_{\mathfrak{p}_\sigma}\frac{k_\sigma-2}{2}}
\varsigma_{\mathfrak{p}_\sigma}\mathfrak{o}_\sigma$. Let $c\in
\BC_p^\times$ be such that $\varsigma_J \mathfrak{o}_{ V_{\vec{k}}
}\subset c\mathfrak{o}_{ V_{\vec{k}} }$. We obtain our assertion.
\end{proof}

Now, we assume $\widehat{\varphi^\infty}$ is ``algebraic'' in the
sense that
\begin{quote}(alg) there is a finite extension $E$ of $\BQ_p$ in $\BC_p$
that contains all embeddings of $F_\mathfrak{p}$ ($\mathfrak{p}|p$),
such that $\widehat{\varphi^\infty}$ lies in $M_{\vec{k}}(U, E)$.
\end{quote} We also make the following ``ordinary'' assumption.
\begin{quote}
(ord) For each $\mathfrak{p}\in J$ we have
$$|\alpha_{\mathfrak{p}}|_{\mathfrak{p}}=|\omega_\mathfrak{p}|_{\mathfrak{p}}^{\sum_{\sigma\in\Sigma_\mathfrak{p}}\frac{k_\sigma-2}{2}},$$
i.e. the $\mathfrak{p}$-adic valuation of the
$U_\mathfrak{p}$-eigenvalue $\alpha_\mathfrak{p}$ is
$\sum_{\sigma\in
\Sigma_\mathfrak{p}}\frac{2-k_\sigma}{2}$.\end{quote}

We define the complex period $\Omega^-_{J,\phi}$ by
\begin{eqnarray*} \Omega^-_{J,\phi}: &=& \widetilde{\Omega}^-_{J,\phi} \cdot
\frac{|N_{F/\BQ}\mathcal{D}_F|_\infty}{|N_{K/\BQ}\mathcal{D}_K|_\infty}
\cdot
\frac{\prod\limits_{w|\mathfrak{c}:\text{inert}}L(1,\tau_{K_w/F_w})^2\cdot
\prod\limits_{w|\mathfrak{c}:\text{ split}}\zeta_{F_w}(1)^2}
{[\mathcal{O}^\times_K\cap
\widehat{\mathcal{O}}^\times_{\vec{n},\mathfrak{c}}
:\mathcal{O}^\times_F]^2 \ \mathbf{N}\mathfrak{c}^2}\\
&& \cdot   \prod_{\mathfrak{p}\in J:
\text{inert}}{L(1,\tau_{K_w/F_w})^2} \cdot \prod_{\mathfrak{p}\in J:
\text{split}}{\zeta_{F_w}(1)^2} \\
&=& \frac{2 L(1, \mathrm{Ad}\pi)}{\zeta_F(2)\langle \varphi_{\pi'},
\pi'(\tau^{\mathfrak{n},B})\varphi_{\pi'}\rangle_{G}}\cdot
\prod_{\sigma\in\Sigma}\frac{1}{\mathcal{P}(1,v_{m_\sigma},
\nu_\sigma)}  \cdot
\frac{|N_{F/\BQ}\mathcal{D}_F|_\infty}{|N_{K/\BQ}\mathcal{D}_K|_\infty}
\\ &&   \\
&& \cdot
\prod_{v|\mathfrak{n}^+}\frac{||\varphi_v||_{v}}{\epsilon(\frac{1}{2},\pi_v,\psi_v)|d_{F_v}|_{v}}
\cdot \prod_{v \nmid p \mathfrak{n}^+ :  \text{ finite and not
split} }\frac{1}{\mathcal{P}(1,\phi_v, \nu_v)} \cdot \prod_{v \nmid
p \mathfrak{n}^+ :  \text{ finite and split} } \frac{1}{|d_{F_v}|_v}
\\
&&\cdot
\frac{\prod\limits_{w|\mathfrak{c}:\text{inert}}L(1,\tau_{K_w/F_w})^2\cdot
\prod\limits_{w|\mathfrak{c}:\text{ split}}\zeta_{F_w}(1)^2}
{[\mathcal{O}^\times_K\cap
\widehat{\mathcal{O}}^\times_{\vec{n},\mathfrak{c}}
:\mathcal{O}^\times_F]^2 \ \mathbf{N}\mathfrak{c}^2} \cdot
\prod_{\mathfrak{p}\in J:
\text{inert}}\frac{||\varphi_\mathfrak{p}||_{\mathfrak{p}}}{|d_{K_\mathfrak{p}}|_{\mathfrak{p}}}
\cdot \prod_{\mathfrak{p}\in J:
\text{split}}\frac{||\varphi_\mathfrak{p}||_{\mathfrak{p}}}{|d_{F_\mathfrak{p}}|_{\mathfrak{p}}}\\
&&  \cdot
\prod_{\begin{array}{c}\mathfrak{p}|p,\mathfrak{p}\notin J: \\
\mathfrak{p} \text{ is split and } r_\mathfrak{p}=1\end{array}}
\frac{||\varphi_\mathfrak{p}||_{\mathfrak{p}}}{\epsilon(\frac{1}{2},\pi_\mathfrak{p},\psi_\mathfrak{p})|d_{F_\mathfrak{p}}|_{\mathfrak{p}}}
\cdot \prod_{\begin{array}{c}\mathfrak{p}|p, \mathfrak{p}\notin J:
\\ \mathfrak{p} \text{ is not split  or
}r_\mathfrak{p}=0\end{array}}
\frac{1}{\mathcal{P}(\varsigma_\mathfrak{p},\varphi_\mathfrak{p},
1)}
\end{eqnarray*}

\begin{thm}\label{thm:interpol} Under the conditions $(\mathrm{alg})$ and
$(\mathrm{ord})$ there exists $\widehat{\Theta}_{J }\in
\mathcal{O}_E[[\Gamma^-_J]]\otimes_{\mathcal{O}_E}E$ such that the
following holds:

For each continuous $p$-adic character $\widehat{\chi}$ of
$\mathrm{Gal}(\bar{K}/K)$ of type $(\mathbf{m},-\mathbf{m})$ in the
$J$-branch of $\widehat{\nu}$ such that $ \chi \cdot \nu ^{-1}$ is a
character of conductor $\prod_{\mathfrak{p}\in
J}\mathfrak{p}^{n_\mathfrak{p}}$ for some
$\vec{n}=(n_\mathfrak{p})_{\mathfrak{p}\in J}$, we have
\begin{eqnarray*} \jmath^{-1}\widehat{\Theta}_{J}(\widehat{\chi}
\widehat{\nu}^{-1})^2 =\chi(\mathfrak{N}^+)\cdot
\prod_{v}\chi_v(\omega_v) \cdot \prod_{\mathfrak{p}\in J}
e_{\mathfrak{p}}(\pi,\chi) \cdot \frac{L(\frac{1}{2}, \pi_K\otimes
\chi)}{\Omega^-_{J,\phi}}
\end{eqnarray*} where $v$ runs over the set
$($\ref{eq:set}$)$.
\end{thm}
\begin{proof} By Lemma \ref{lem:control-1} and Lemma
\ref{lem:control-2} there exists some integer $r$ independent of
$\vec{n}$ such that $\widehat{\Theta}_{\vec{n},\mathfrak{c}}\in
\omega_E^r \mathcal{O}_E[[\mathcal{G}_{\vec{n},\mathfrak{c}}]]$. So
by Proposition \ref{prop:comp}, the inverse limit
$\lim\limits_{\overleftarrow{\;\;\vec{n}\:\;}}\widehat{\Theta}_{\vec{n},\mathfrak{c}}$
exists in $\omega_E^r \cdot
\mathcal{O}_E[[\mathcal{G}_{J;\mathfrak{c}}]]$. We denote it as
$\widehat{\Theta}_{J;\mathfrak{c}}$ and let $\widehat{\Theta}_{J}$
be its image in
$\mathcal{O}_E[[\Gamma^-_J]]\otimes_{\mathcal{O}_E}E$ by the
projection $\mathcal{G}_{J;\mathfrak{c}}\rightarrow \Gamma^-_J$. Now
our assertion follows from Corollary \ref{cor:spe-value},
Proposition \ref{prop:theta} and Lemma \ref{lem:volume}.
\end{proof}

\begin{defn} \label{def:measure} The element
$\widehat{\Theta}_{\vec{n},\mathfrak{c}}$ defines a measure $\mu_J$
on $$ \mathcal{G}_{\vec{n},\mathfrak{c}}\cong \BA _F^{\infty,\times}
\backslash \BA _K^{\infty,\times} / K^\times
\widehat{\mathcal{O}}_{\vec{n}, \mathfrak{c}}^\times . $$ The
element $\widehat{\Theta}_{J;\mathfrak{c}}$ defines a measure
$\mu_J$ on $$ \mathcal{G}_{J;\mathfrak{c}}\cong \BA
_F^{\infty,\times} \backslash \BA _K^{\infty,\times} / K^\times
\widehat{\mathcal{O}}_{J; \mathfrak{c}}^\times . $$ Here
$\widehat{\mathcal{O}}_{J; \mathfrak{c}} = \bigcap\limits_{\vec{n}}
\widehat{\mathcal{O}}_{\vec{n}, \mathfrak{c}}$.
\end{defn}

Put $$\mathscr{L}_J(\pi_K, \widehat{\chi})=
\widehat{\Theta}_{J}(\widehat{\chi} \widehat{\nu}^{-1}),$$ where
$\widehat{\chi}$ runs through the $J$-branch of $\widehat{\nu}$.
Then we define the anticyclotomic $p$-adic $L$-function $L_J(\pi_K,
\widehat{\chi})$  by
$$ L_J (\pi_K, \widehat{\chi}) = \mathscr{L}_J(\pi_K,
\widehat{\chi})^2 . $$ In particular, when
$\widehat{\chi}\widehat{\nu}^{-1}$ runs over the analytic family $$
\{  \epsilon ^{\vec{s}} : \vec{s}=(s_{\sigma})_{\sigma\in \Sigma_J},
|s_\sigma|\leq |c_0|_pp^{-2} \} ,
$$ we obtain a $\sharp(\Sigma_J)$-variable analytic function
$$ L_J(\vec{s},\pi_K, \widehat{\nu} ) := L_J(\pi_K,\widehat{\nu}\epsilon ^{\vec{s}} )  $$
on $$ \{(s_{\sigma})_{\sigma\in \Sigma_J}: |s_\sigma|\leq
|c_0|_pp^{-2}\} . $$ We also write $$ \mathscr{L}_J(\vec{s},\pi_K,
\widehat{\nu}) := \mathscr{L}_J(\pi_K,\widehat{\nu}\epsilon
^{\vec{s}} )  . $$

\begin{cor} \label{cor:vanish-value}
If there exists some $\mathfrak{p}\in J$ such that
$\alpha_\mathfrak{p}=1$, and if
$\widehat{\chi}_\mathfrak{p}\widehat{\nu}_\mathfrak{p}^{-1}=1$, then
$$\mathscr{L}_J(\pi_K, \widehat{\chi})=0.$$
\end{cor}
\begin{proof} By continuity we may assume that
$\widehat{\chi}\widehat{\nu}^{-1}$ is of finite order, i.e.
$\widehat{\chi}$ is of $(\mathbf{m},\mathbf{m})$-type. Then
$\chi_\mathfrak{p}=\nu_\mathfrak{p}\jmath^{-1}(\widehat{\chi}\widehat{\mu}^{-1})=1$
and so $\tilde{e}_\mathfrak{p}(\pi, \chi)=0$. Thus our statement
follows from Theorem \ref{thm:interpol}.
\end{proof}

\section{Harmonic cocycle valued and cohomological valued modular
forms}\label{sec:har-coh}

In this section we define Harmonic cocycle valued and cohomological
valued modular forms. Then we construct two maps called the
Schneider morphism and the Coleman morphism from Harmonic cocycle
valued modular forms to cohomological valued modular forms. Finally
we explain $L$-invariant of Teitelbaum type as the ratio of two
cohomological valued modular forms. We also define
``multiple-variable Harmonic cocycle'' valued modular forms and
attach them distributions needed in Section \ref{sec:meas-dist}.

The $L$-invariant of Teitelbaum type was defined by Chida, Mok and
Park \cite{CMP}. They defined $L$-invariant of this type to be the
ratio of two sets of data; those data are evaluations at special
points of our cohomological valued modular forms. We will fill some
details omitted in \cite{CMP}.  For example, to define Hecke
operators on cohomological valued modular forms, we need
corestriction maps for cohomological groups, while this is not
clearly described in \cite{CMP}.

To prove our result in Section \ref{sec:partial}, it is rather
convenient to use our language of cohomological valued modular forms
than that in \cite{CMP},  since we need evaluations of the
cohomological valued modular forms at points other than those chosen
out in \cite{CMP}.

\subsection{Notations}
Let $B$ and $G$ be as in Section \ref{ss:quaternion}. Then $B$
splits at all primes above $p$. Let $J_0$ be the subset of $J$ of
elements $\mathfrak{p}\in J$ such that $\alpha_\mathfrak{p}=1$. We
assume that each $\mathfrak{p}\in J_0$ splits in $K$.

Let $U=\prod_{\mathfrak{l}}U_{\mathfrak{l}}$ be a compact open
subgroup of $G^\infty$.  We assume that for each $\mathfrak{p}\in
J_0$, $U_{\mathfrak{p}}=U_0(\omega_\mathfrak{p})$.

Let $E$ be a sufficiently large extension of $\BQ_p$. In \S
\ref{ss:p-mod-form} for each $\sigma$ we define two vector spaces
$V_{k_\sigma}(E)$ and $L_{k_\sigma}(E)$ over $E$ with actions
$\rho_{k_{\sigma}}\circ \mathfrak{i}_{\sigma'}i_{\sigma'}^{-1}$ and
$\check{\rho}_{k_\sigma}\circ
\mathfrak{i}_{\sigma'}i_{\sigma'}^{-1}$ of
$\mathrm{GL}_2(F_{\mathfrak{p}_\sigma})$. As we will frequently use
these two notations $\rho_{k_{\sigma}}\circ
\mathfrak{i}_{\sigma'}i_{\sigma'}^{-1}$ and
$\check{\rho}_{k_\sigma}\circ
\mathfrak{i}_{\sigma'}i_{\sigma'}^{-1}$, to avoid cumbursomeness we
use $\rho_{k_\sigma}$ and $\check{\rho}_{k_\sigma}$ to denote them.

For each subset $\Sigma$ of $\Sigma_F$ we form two $E$-vector spaces
$$V_\Sigma(E)=\bigotimes_{\sigma\in \Sigma} V_{k_{\sigma}}(E) $$
and
$$V^{\Sigma}(E)=\bigotimes_{\sigma\notin \Sigma} V_{k_\sigma}(E)$$ with tensor product actions of
$G_p=\prod_{\mathfrak{p}|p}\GL_2(F_\mathfrak{p})$. Here we make the
convention that, if there is no $\sigma\in \Sigma$ such that
$\mathfrak{p}_\sigma=\mathfrak{p}$, then $\GL_2(F_\mathfrak{p})$
acts trivially on $V_\Sigma$; and if there is no $\sigma\in
\Sigma_F\backslash \Sigma$ such that
$\mathfrak{p}_\sigma=\mathfrak{p}$ then $\GL_2(F_\mathfrak{p})$ acts
trivially on $V^\Sigma$. When $V_\Sigma$ or $V^\Sigma$ is clear, by
abuse of notation we use $\rho_{\vec{k}}$ to denote the action of
$G_p$. Similarly, we form the $E$-vector spaces $L_\Sigma(E)$ and
$L^\Sigma(E)$ with actions $\check{\rho}_{\vec{k}}$ of $G_p$. When
$\Sigma=\Sigma_{J_1}$ ($J_1\subset J_0$), we also write $J_1$
instead of $\Sigma_{J_1}$ in the superscript and the subscript.

\subsection{Bounded harmonic cocycles and analytic distributions}
\label{ss:har-to-rigid} \label{ss:harmonic}

First we fix a prime $\mathfrak{p}\in J_0$.

Consider the Bruhat-Tits $\mathcal{T}_\mathfrak{p}$ tree for
$\GL_2(F_\mathfrak{p})$. Denote by $\mathcal{V}_\mathfrak{p}$ and
$\mathcal{E}_\mathfrak{p}$ the set of vertices and the set of
oriented edges of $\mathcal{T}_\mathfrak{p}$ respectively. The
source and the target of an oriented edge $e$ are denoted by $s(e)$
and $t(e)$ respectively. A vertex can be represented by a homotopy
class of lattices $L$ in $F_\mathfrak{p}^{\oplus 2}$. An edge can be
represented by a pair of lattices $(L_1,L_2)$ with
$[L_1:L_2]=|\omega_\mathfrak{p}|_{\mathfrak{p}}^{-1}$. Let $e_0$ be
the edge $(L_1,L_2)$ with
$L_1=\omega_\mathfrak{p}^{-1}\mathcal{O}_{F_\mathfrak{p}}\oplus
\mathcal{O}_{F_\mathfrak{p}}$ and
$L_2=\mathcal{O}_{F_\mathfrak{p}}\oplus
\mathcal{O}_{F_\mathfrak{p}}$. Let $\hbar_\mathfrak{p}$ be as in
Lemma \ref{lem:h-sigma}. If
$\hbar_\mathfrak{p}=\wvec{a_0}{b_0}{c_0}{d_0}$, we let $e_*$ be the
edge $(L_3, L_4)$ with
$$L_3=\{(a_0\omega_\mathfrak{p}^{-1}x+b_0y,
c_0\omega_\mathfrak{p}^{-1}x+d_0y):
x,y\in\mathcal{O}_{F_\mathfrak{p}}\}$$ and $$L_4=\{(a_0 x+b_0y, c_0
x+d_0y): x,y\in\mathcal{O}_{F_\mathfrak{p}}\}.$$

We consider the twisted action of $\GL_2(F_\mathfrak{p})$ on
$F_\mathfrak{p}^{\oplus 2}$: if $\mathbf{i}_\mathfrak{p}
i_\mathfrak{p}^{-1}(g)=\wvec{a}{b}{c}{d}$, then $g\cdot
(x,y)=(ax+by, cx+dy)$. This action  induces an action on
$\mathcal{T}_\mathfrak{p}$, and also an action on
$\mathbf{P}^1(F_\mathfrak{p})=F_\mathfrak{p}\cup\{\infty\}$. The
benefit of this twisted action is that
$i_\mathfrak{p}(K^\times_\mathfrak{p})$ acts on
$\mathbf{P}^1(F_\mathfrak{p})$ diagonally. Indeed, for any
$t=(t_\mathfrak{P}, t_{\bar{\mathfrak{P}}})\in
K_\mathfrak{p}^\times$, we have $i_\mathfrak{p}(t)\cdot
x=\frac{t_\mathfrak{P} x}{t_{\bar{\mathfrak{P}}} }$ for any $x\in
\mathbf{P}^1(F_\mathfrak{p})$. Hence, the fixed points in
$\mathbf{P}^1(F_\mathfrak{p})$ of
$i_\mathfrak{p}(K^\times_\mathfrak{p})$ are $0$ and $\infty$.

\begin{rem}\label{rem:twist-2} For the purpose of defining $L$-invariants, it is cumbersome and unnecessary to consider the
twisted action. However, it is rather convenience when we apply the
result in this section to prove the main result in Section
\ref{sec:partial}.
\end{rem}

The isotropy group of $e_*$ is $F_{\mathfrak{p}}^\times
U_\mathfrak{p}$, and the isotropy group of $e_0$ is
$F_{\mathfrak{p}}^\times\mathrm{Ad}(\hbar_\mathfrak{p})^{-1}(U_\mathfrak{p})$.
Thus $\mathcal{E}_\mathfrak{p}$ is isomorphic to the coset
$\GL_2(F_\mathfrak{p})/F_{\mathfrak{p}}^\times U_\mathfrak{p}$. Let
$G_p$ act on $\mathcal{T}_\mathfrak{p}$ by the projection
$G_p\rightarrow \GL_2(F_\mathfrak{p})$.

Let $C^0(\mathcal{T}_\mathfrak{p},L_{\vec{k}}(E))$ be the space of
$L_{\vec{k}}(E)$-valued functions on $\mathcal{V}_\mathfrak{p}$,
$C^1(\mathcal{T}_\mathfrak{p},L_{\vec{k}}(E))$ the space of
$L_{\vec{k}}(E)$-valued functions on $\mathcal{E}_\mathfrak{p}$ such
that $f(e)=-f(\bar{e})$. Let $G_p$ act on
$C^i(\mathcal{T}_\mathfrak{p}, L_{\vec{k}}(E))$ by $\gamma \star f = \gamma
\circ f\circ \gamma_\mathfrak{p}^{-1}$. Then we have a $G_p$-equivariant short
exact sequence
\begin{equation} \label{eq:cover-sq}
\xymatrix{ 0\ar[r] & L_{\vec{k}}(E) \ar[r] &
C^0(\mathcal{T}_\mathfrak{p}, L_{\vec{k}}(E)) \ar[r]^{\partial} &
C^1(\mathcal{T}_\mathfrak{p}, L_{\vec{k}}(E)) \ar[r] & 0
}\end{equation} where $$\partial(f)(e)=f(s(e))-f(t(e)).$$ For any
subgroup $\Gamma$ of $G_p$, from (\ref{eq:cover-sq}) we get the
injective map \begin{equation}\label{eq:delta} \delta_\Gamma:
C^1(\mathcal{T}_{\mathfrak{p}}, L_{\vec{k}}(E))^\Gamma \rightarrow
H^1(\Gamma,L_{\vec{k}}(E)).\end{equation}

Let $C^1_{\mathrm{har}}(\mathcal{T}_\mathfrak{p},L_{\vec{k}}(E))$ be
the space of {\it harmonic forms} $$ \left\{f\in
C^1(\mathcal{T}_\mathfrak{p},L_{\vec{k}}(E)): f(e)=-f(\bar{e}) \
\forall\ e\in \mathcal{E}_\mathfrak{p}, \text{ and }  \sum_{t(e)=v}
f(e)=0 \ \forall\ v \in \mathcal{V}_\mathfrak{p} \right\}.$$ Observe
that $C^1_{\mathrm{har}}(\mathcal{T}_\mathfrak{p},L_{\vec{k}}(E))$
is $G_p$-stable. We again use $\delta_\Gamma$ to denote the
composition
$$ C^1_{\mathrm{har}}(\mathcal{T}_\mathfrak{p},L_{\vec{k}}(E))^{\Gamma}\hookrightarrow
C^1(\mathcal{T}_\mathfrak{p},L_{\vec{k}}(E))^{\Gamma} \rightarrow H^1(\Gamma,L_{\vec{k}}(E)) . $$

For a subgroup $\Gamma$ of $G_p$, and $h\in G_p$, we have a map
\begin{eqnarray*} r_h : H^1(\Gamma, L_{\vec{k}}(E))
&\rightarrow& H^1(h^{-1}\Gamma h, L_{\vec{k}}(E)) , \\
  \phi &\mapsto&
( r_h \phi)(\gamma')= \check{\rho}_{\vec{k}}(h)^{-1}\phi(h \gamma
h^{-1}).
\end{eqnarray*}

We shall use the following lemma later. It comes from homological
theory.

\begin{lem}\label{lem:HA} Let  $\Gamma_1\subset \Gamma_2$ be two subgroups of
$G_p$.
\begin{enumerate} \item\label{it:HA-a} We have the following commutative diagram
\[ \xymatrix{ C^1_{\mathrm{har}}(\mathcal{T}_\mathfrak{p},L_{\vec{k}}(E))^{\Gamma_2} \ar[r]^{\delta_{\Gamma_2}}
\ar@{^(->}[d] & H^1(\Gamma_2,L_{\vec{k}}(E))
\ar[d]_{\mathrm{Res}_{\Gamma_2/\Gamma_1}}
\\ C^1_{\mathrm{har}}(\mathcal{T}_\mathfrak{p},L_{\vec{k}}(E))^{\Gamma_1} \ar[r]^{\delta_{\Gamma_1}}
 & H^1(\Gamma_1,L_{\vec{k}}(E))  . }\]
 \item\label{it:HA-b} When $[\Gamma_2:\Gamma_1]$ is finite, we have the following commutative diagram
\[ \xymatrix{ C^1_{\mathrm{har}}(\mathcal{T}_\mathfrak{p},L_{\vec{k}}(E))^{\Gamma_1} \ar[r]^{\delta_{\Gamma_1}}
\ar[d]_{\mathrm{Cor}_{\Gamma_2/\Gamma_1}}  &
H^1(\Gamma_2,L_{\vec{k}}(E))
\ar[d]_{\mathrm{Cor}_{\Gamma_2/\Gamma_1}}
\\ C^1_{\mathrm{har}}(\mathcal{T}_\mathfrak{p},L_{\vec{k}}(E))^{\Gamma_2} \ar[r]^{\delta_{\Gamma_2}}
 & H^1(\Gamma_2,L_{\vec{k}}(E)) . }\]
 \item\label{it:HA-c} If $\Gamma$ is a subgroup of $G_p$ and $h$ is an element of
 $G_p$, then we have the following commutative diagram
 \[ \xymatrix{ C^1_{\mathrm{har}}(\mathcal{T}_\mathfrak{p},L_{\vec{k}}(E))^{\Gamma} \ar[rr]^{\delta_{\Gamma}}
\ar[d]_{\check{\rho}_{\vec{k}}(h^{-1})} &&
H^1(\Gamma,L_{\vec{k}}(E)) \ar[d]_{r_{h}}
\\ C^1_{\mathrm{har}}(\mathcal{T}_\mathfrak{p},L_{\vec{k}}(E))^{h^{-1}\Gamma h} \ar[rr]^{\delta_{h\Gamma h^{-1}}}
 && H^1(h^{-1}\Gamma h,L_{\vec{k}}(E)).   }\]
\end{enumerate}
\end{lem}

We extend the above notion of harmonic cocycles to multiple
variables.

Fix a nonempty subset $J_1$ of $J_0$. We form
$\mathcal{E}_{J_1}=\prod_{\mathfrak{p}\in J_1}
\mathcal{E}_\mathfrak{p}$. Taking product we obtain action of
$G_{J_1}=\prod_{\mathfrak{p}\in J_1} G_\mathfrak{p}$ on
$\mathcal{E}_{J_1}$. For each element $e\in \mathcal{E}_{J_1}$ we
write $e_\mathfrak{p}(e)$ for the $\mathfrak{p}$-component of $e$
and put $s_\mathfrak{p}(e)=s(e_\mathfrak{p}(e))$. For each
$\mathfrak{p}\in J_1$ let $e_{\mathfrak{p},*}\in
\mathcal{E}_\mathfrak{p}$ be an oriented edge fixed by
$U_{\mathfrak{p}}$. Write
$e_{J_1,*}=(e_{\mathfrak{p},*})_{\mathfrak{p}\in J_1}$.

For any edge $e\in \mathcal{E}_\mathfrak{p}$ let
$\mathcal{U}_e\subset \mathbf{P}^1(F_\mathfrak{p})$ denote the set
of end points of paths (in $\mathcal{T}_\mathfrak{p}$) through $e$.
For $g\in G_\mathfrak{p}$ we have $\mathcal{U}_{ge}=g\mathcal{U}_e$.
If we identify $\mathbf{P}^1(F_\mathfrak{p})$ with
$F_\mathfrak{p}\cup\{\infty\}$ by $[a,b]\mapsto \frac{a}{b}$, then
$\mathcal{U}_{e_{\mathfrak{p},0}}=\mathcal{O}_{F_\mathfrak{p}}$. For
$e=\{e_j\}_{j\in J_1}\in\mathcal{E}_{J_1}$ we put
$\mathcal{U}_e=\prod_{j\in J_1}\mathcal{U}_{e_j}$.

\begin{defn}\label{defn:harmonic}
A {\it harmonic cocycle on $\mathcal{E}_{J_1}$ with values in
$L_{\vec{k}}(E)$} is a function $c: \mathcal{E}_{J_1}\rightarrow
L_{\vec{k}}(E)$ that satisfies the following two conditions:
\begin{enumerate}
\item\label{it:har-a} For each $\mathfrak{p}\in J_1$ if
$e_\mathfrak{p}$ is replaced by $\bar{e}_\mathfrak{p}$, then $c(e)$
is changed by a sign.
\item\label{it:har-b} For each $\mathfrak{p}\in J_1$ and $\sharp(J_1)$ elements $v_\mathfrak{p}\in
\mathcal{V}_\mathfrak{p}$ and  $e_{\mathfrak{p}'}\in
\mathcal{E}_{\mathfrak{p}'}$ ($\mathfrak{p}'\in J_1$,
$\mathfrak{p}'\neq \mathfrak{p}$), we have
$$\sum_{s_\mathfrak{p}(e)=v, e_{\mathfrak{p}'}(e)=e_{\mathfrak{p}'}}
c(e)=0,$$ where the sum runs through all $e\in \mathcal{E}_{J_1}$
such that $s_\mathfrak{p}(e)=v_\mathfrak{p}$ and
$e_{\mathfrak{p}'}(e)=e_{\mathfrak{p}'}$ ($\mathfrak{p}'\neq
\mathfrak{p}$).
\end{enumerate}
\end{defn}

Let $G_p$ act on the space of harmonic cocycles by
$$ (g\star  c) (e) = \check{\rho}_{\vec{k}}(g)  c(g_{J_1}^{-1} e) .$$

\begin{defn} We say that a harmonic cocycle $c$ is {\it bounded}, if $$ \{ (
g\star c)(e_{J_1,*}) : g\in G_{J_1}\} $$ is bounded for any norm
$|\cdot|$ on $L_{\vec{k}}(E)$, i.e.
$$ \sup_{ g\in G_{J_1} } |\check{\rho}_{\vec{k}}(g) c(g^{-1}e_{J_1,*})|
<\infty.
$$
\end{defn}

Next we attach to each bounded harmonic cocycle an analytic (vector
valued) distribution.

For each $\mathfrak{p}\in J_1$ we fix an $\iota_\mathfrak{p}\in
\Sigma_\mathfrak{p}$. Write
$\iota_{J_1}=(\iota_\mathfrak{p})_{\mathfrak{p}\in J_1}$.

Let $\mathrm{LP}_{\iota_\mathfrak{p}}$ be the space of of locally
polynomials on $\mathbf{P}^1(F_\mathfrak{p})$ of degree $\leq
k_{\iota_\mathfrak{p}}-2$. Precisely, a function on
$\mathbf{P}^1(F_\mathfrak{p})$ belongs to
$\mathrm{LP}_{\iota_\mathfrak{p}}$ if and only if for each point $x$
of $\mathbf{P}^1(F_\mathfrak{p})$ there exists an open neighborhood
$U_x$ of $x$ such that $f|_{U_x}$ is a polynomial with coefficients
in $E$ of degree $\leq k_{\iota_\mathfrak{p}}-2$. We define an
action of $\mathrm{GL}(2,F_\mathfrak{p})$ on $\mathrm{LP}_{
\iota_\mathfrak{p} }$ by
$$(\mathrm{Ad}(\hbar_\mathfrak{p}^{-1})\wvec{a}{b}{c}{d})^{-1} f  (x) = \iota_{\mathfrak{p}}\left(\frac{( c x+ d )^{k_{\iota_\mathfrak{p}}-2} }{(ad-bc)^{\frac{k_{\iota_\mathfrak{p}}}{2}-1}} \right)
 f \left(\frac{ax+b}{cx+d}\right) .
$$
Let $\mathrm{LP}_{\iota_{J_1}}$ be the tensor product
$\bigotimes_{\mathfrak{p}\in J_1}\mathrm{LP}_{\iota_\mathfrak{p}}$
(of $E$-vector spaces) with the tensor product action of
$G_{J_1}=\prod_{\mathfrak{p}\in J_1}\GL_2(F_\mathfrak{p})$;
$\mathrm{LP}_{\iota_{J_1}}$ can be considered as a space of
functions on $\mathbf{P}_{J_1}:=\prod_{\mathfrak{p}\in
J_1}\mathbf{P}^1(F_\mathfrak{p})$.

Using the relation (\ref{it:har-b}) in Definition
\ref{defn:harmonic} we attach to each harmonic cocycle $c$ an
$L^{\iota_{J_1}}(E)$-valued linear functional $\mu_c^{\iota_{J_1}}$
of $\mathrm{LP}_{\iota_{J_1}}$ such that
$$ \langle \int_{\mathcal{U}_e} \prod_{\mathfrak{p}\in J_1}x_{\mathfrak{p}}^{j_\mathfrak{p}} \mu_{c}^{\iota_{J_1}} , Q \rangle =
\frac{ \langle  c(e),   (\bigotimes_{\mathfrak{p}\in
J_1}X_{\mathfrak{p}}^{j_{\mathfrak{p}}}Y_{\mathfrak{p}}^{k_{\iota_\mathfrak{p}}-2-j_{\mathfrak{p}}})
\otimes Q\rangle}{\prod\limits_{\mathfrak{p}\in
J_1}\binc{k_{\iota_\mathfrak{p}}-2}{j_{\mathfrak{p}}}}
$$ for $e\in \mathcal{E}_{J_1}$, $Q\in V^{\iota_{J_1}}(E)$,
and  $j_{\mathfrak{p}} \in \{0,1,\cdots, k_{\iota_\mathfrak{p}}-2
\}$.

In the following, for the purpose of simplifying notations, we will
write $\prod_{\mathfrak{p}\in J_1}x_{\mathfrak{p}}^{j_\mathfrak{p}}$
for the element $\bigotimes_{\mathfrak{p}\in
J_1}X_{\mathfrak{p}}^{j_{\mathfrak{p}}}Y_{\mathfrak{p}}^{k_{\iota_\mathfrak{p}}-2-j_{\mathfrak{p}}}$
in $V_{\iota_{J_1}}(E)$.

\begin{lem} For $g\in G_p$ we have
\begin{equation}\label{eq:invariant}
\int_{\mathcal{U}_e} (\rho_{\vec{k}}(g^{-1}) P  )
\mu_{c}^{\iota_{J_1}}= \check{\rho}_{\vec{k}}(g^{-1}) \int_{g
\mathcal{U}_e} P \mu_{g\star c}^{\iota_{J_1}}
\end{equation} for each $P\in V_{\iota_{J_1}}(E)$.
\end{lem}
\begin{proof} If $P=\prod_{\mathfrak{p}\in J_1}x_{\mathfrak{p}}^{j_\mathfrak{p}}\in V_{\iota_J}(E)$ and $Q\in
V^{\iota_{J_1}}(E)$, then \begin{eqnarray*} && \langle
\int_{\mathcal{U}_e}
 \rho_{\vec{k}}(g^{-1}) P   \mu_{c}^{\iota_{J_1}},Q \rangle  = \frac{
\langle  c(e), ( \rho_{\vec{k}}(g^{-1})P) \otimes
Q\rangle}{\prod_{\mathfrak{p}\in
J_1}\binc{k_{\iota_\mathfrak{p}}-2}{j_{\mathfrak{p}}}}  = \frac{
\langle   \check{\rho}_{\vec{k}}(g)  c(e), P \otimes
 \rho_{\vec{k}}(g)Q \rangle}{\prod_{\mathfrak{p}\in
J_1}\binc{k_{\iota_\mathfrak{p}}-2}{j_{\mathfrak{p}}}} \\
&=& \frac{ \langle (g\star c)(g e), P \otimes \rho_{\vec{k}}(g)Q
\rangle}{\prod_{\mathfrak{p}\in
J_1}\binc{k_{\iota_\mathfrak{p}}-2}{j_{\mathfrak{p}}}}  =  \langle
 \int_{\mathcal{U}_{g e}} P \mu_{g\star c}^{\iota_{J_1}} ,\rho_{\vec{k}}(g)Q  \rangle = \langle
\check{\rho}(g^{-1})\int_{\mathcal{U}_{g e}} P \mu_{g\star
c}^{\iota_{J_1}}, Q \rangle ,
\end{eqnarray*} as wanted.
\end{proof}

\begin{rem}\label{rem:measure} If $k_{\mathfrak{p}}=2$ for each $\mathfrak{p}\in J_1$, then for each $g\in G_{J_1}$ we have
$$ \mu^{\iota_{J_1}}_c(g \cdot \mathcal{U}_{e_{J_1,*}}) = c(g\cdot e_{J_1,*}). $$
\end{rem}

For a point $a_{J_1}\in \prod_{\mathfrak{p}\in J_1} F_\mathfrak{p}$,
and a vector of positive integers
$m_{J_1}=(m_\mathfrak{p})_{\mathfrak{p}\in J_1}$, we use
$\mathcal{U}(a_{J_1}, m_{J_1})$ to denote the product of closed
discs
$$ \mathcal{U}(a_{J_1}, m_{J_1})=\prod_{\mathfrak{p}\in J_1}\mathcal{U}(a_{\mathfrak{p}}, m_{\mathfrak{p}}). $$ with the convention
$$\mathcal{U}(\infty, m_\mathfrak{p})=\{x\in \mathbf{P}^1(F_\mathfrak{p}):
|x|\geq |\omega_\mathfrak{p}|^{-m_\mathfrak{p}}\}.$$

For each point $a_{\mathfrak{p}}\in \mathbf{P}^1(F_\mathfrak{p})$,
put $$ A(a_{\mathfrak{p}}) =
\left\{\begin{array}{ll} \max\{ 1, |a_{\mathfrak{p}}|\}^{k_\mathfrak{p}} & \text{ if } a_\mathfrak{p}\neq \infty ,  \\
1 & \text{ if } a_\mathfrak{p} = \infty.
\end{array}\right.
$$ Then we put $A(a_{J_1})=\prod_{\mathfrak{p}\in
J_1}A(a_\mathfrak{p})$.

\begin{prop}\label{prop:estimate}
If $c$ is bounded, then there exists a constant $A>0$ such that
$$\left| \int_{\mathcal{U}(a_{J_1}, m_{J_1})} \prod_{\mathfrak{p}\in J_1}
(x_\mathfrak{p}-a_{\mathfrak{p}})^{j_\mathfrak{p}}
\mu_c^{\iota_{J_1}} \right| \leq A \cdot A(a_{J_1})\cdot
\prod_{\mathfrak{p}\in J_1} |\omega_\mathfrak{p}|
^{m_\mathfrak{p}(j_\mathfrak{p}+1-\frac{k_{\iota_\mathfrak{p}}}{2})}
$$ for each $a_{J_1}$ and $m_{J_1}\geq 0$. Here $m_{J_1}\geq 0$ means that each
component of $m_{J_1}$ is nonnegative.
\end{prop}
\begin{proof} Put $$ g_\mathfrak{p} = \left\{\begin{array}{ll}
{i}_\mathfrak{p}
\mathbf{i}_\mathfrak{p}^{-1}\wvec{1}{-a_\mathfrak{p}}{0}{
\omega_\mathfrak{p}^{ m_\mathfrak{p} } } & \text{ if }
a_\mathfrak{p}\neq \infty , \\ {i}_\mathfrak{p}
\mathbf{i}_\mathfrak{p}^{-1}\wvec{0}{1}{ \omega_\mathfrak{p}^{
m_\mathfrak{p} } }{0} & \text{ if } a_\mathfrak{p} = \infty
\end{array}\right.$$ and $g=(g_{\mathfrak{p}})_{\mathfrak{p}}\in
G_{J_1}$. Then
$g_\mathfrak{p}(\mathcal{U}(a_{\mathfrak{p}},{m_{\mathfrak{p}}}
))=\mathcal{O}_{F_\mathfrak{p}}$, and so $g\mathcal{U}(a_{J_1},
m_{J_1})=\mathcal{U}(0_{J_1},0_{J_1})$.

When none of $a_\mathfrak{p}$ is infinite, as
$$  g  \prod_{\mathfrak{p}\in J_1} x_\mathfrak{p}^{j_\mathfrak{p}}
= \frac{ \omega_{\mathfrak{p}}^{
m_\mathfrak{p}(k_{\iota_\mathfrak{p}} -2)} } {
\omega_{\mathfrak{p}}^{ m_\mathfrak{p}
(\frac{k_{\iota_\mathfrak{p}}}{2}-1) } } \left(\frac{
x_{\mathfrak{p}}-a_{\mathfrak{p}} } { \omega_\mathfrak{p}^{
m_\mathfrak{p}} }\right)^{j_\mathfrak{p}} = \prod_{\mathfrak{p}\in
J_1} \omega_{\mathfrak{p}}^{
m_\mathfrak{p}(\frac{k_{\iota_\mathfrak{p}} }{2}-1-j_\mathfrak{p}) }
(x_\mathfrak{p}-a_\mathfrak{p})^{j_\mathfrak{p}},$$ by
(\ref{eq:invariant}) we have
$$ \int_{\mathcal{U}(a_{J_1}, m_{J_1})} \prod_{\mathfrak{p}\in J_1}(x_\mathfrak{p}-a_\mathfrak{p})^{j_\mathfrak{p}}  \mu_c^{\iota_{J_1}}  =
\prod_{\mathfrak{p}\in J_1}
\omega_\mathfrak{p}^{m_\mathfrak{p}(j_\mathfrak{p}+1-\frac{k_{\iota_\mathfrak{p}}}{2})}
g^{-1}\int_{\mathcal{U}(0_{J_1},0_{J_1})} \prod_{\mathfrak{p}\in
J_1} x_\mathfrak{p}^{j_\mathfrak{p}} \mu^{\iota_{J_1}}_{g\star c}  .
$$ As $c$ is bounded and $g^{-1}$ is bounded by $A(a_{J_1})$, this yields our
assertion. When some of $a_\mathfrak{p}$ is infinite, the argument
is similar, and we omit it.
\end{proof}

\begin{prop}\label{prop:har-to-measure}
If $c$ is bounded, then there is a unique $L^{\iota_{J_1}}$-valued
analytic distribution $\mu_c^{\iota_{J_1}}$ on $\mathbf{P}_{J_1}$
such that
$$ \langle Q, \int_{\mathcal{U}_e} \prod\limits_{\mathfrak{p}\in J_1}x_{\mathfrak{p}}^{j_\mathfrak{p}}  \mu_{c}^{\iota_{J_1}}\rangle =
\frac{ \langle \prod\limits_{\mathfrak{p}\in
J_1}X_{\iota_\mathfrak{p}}^{j_\mathfrak{p}}Y_{\iota_\mathfrak{p}}^{k_{\iota_\mathfrak{p}}-2-j_{\mathfrak{p}}}\otimes
Q, c(e)\rangle}{\prod\limits_{\mathfrak{p}\in
J_1}\binc{k_{\iota_{\mathfrak{p}}}-2}{j_\mathfrak{p}}}, \hskip 20pt
0\leq j_\mathfrak{p} \leq k_{\iota_\mathfrak{p}}-2
$$ for each $Q\in V^{\iota_{J_1}}(E)$.
\end{prop}
\begin{proof}
This follows from Proposition \ref{prop:estimate} and a standard
Amice-Velu and Vishik's argument (\cite{A-V,Vishik}, see also
\cite[\S 11]{MMT}). We only give a sketch below.

For an analytic function $f$ on some disc $\mathcal{U}(a_{J_1},
m_{J_1})$ we write
$$ f= \sum_{\vec{j}=(j_{\mathfrak{p}}): j_{\mathfrak{p}}\geq 0}c_{\vec{j}} (x-a)^{\vec{j}} . $$ Here, we write $(x-a)^{\vec{j}}$ for
$\prod\limits_{\mathfrak{p}\in
J}(x_\mathfrak{p}-a_{\mathfrak{p}})^{j_{\mathfrak{p}}}$; the
coefficients $c_{\vec{j}}$ are in $\BC_p\otimes_E L^{\iota_{J_1}}$.

We write $I(f, a_{J_1}, m_{J_1})$ for the lattice of $\BC_p\otimes
_E L^{\iota_{J_1}}$ generated by $$
c_{\vec{j}}\cdot\prod\limits_{\mathfrak{p}\in
J_1}\iota_{\mathfrak{p}}(\omega_\mathfrak{p})^{m_{\mathfrak{p}}(j_\mathfrak{p}-k_{\iota_\mathfrak{p}}+1)}
$$ for all $\vec{j}$. If $\mathcal{U}(a'_{J_1}, m'_{J_1})\subset \mathcal{U}(a_{J_1}, m_{J_1})$, writing
$$f|_{\mathcal{U}(a'_{J_1},m'_{J_1})}=\sum_{\vec{j}}c'_{\vec{j}} (x-a')^{\vec{j}}  $$
we have $$ c'_{\vec{j}}=
\sum_{\vec{n}}c_{\vec{n}}\prod_{\mathfrak{p}\in
J_1}\binc{n_\mathfrak{p}}{j_{\mathfrak{p}}}(a'_\mathfrak{p}-a_{\mathfrak{p}})^{n_{\mathfrak{p}}-j_{\mathfrak{p}}}.
$$ Since $a'_\mathfrak{p}-a_{\mathfrak{p}}\in
(\omega_\mathfrak{p}^{m_\mathfrak{p}})$, we obtain $I(f, a'_{J_1},
m'_{J_1})\subset I(f, a_{J_1}, m_{J_1})$. In particular, $I(f,
a_{J_1}, m_{J_1})$ is independent of the choice of the center
$a_{J_1}$.

Consider the truncations $$f_{a_{J_1}}=\sum_{\vec{j}:
j_{\mathfrak{p}}\leq k_{\iota_\mathfrak{p}}-2} c_{\vec{j}}
(x-a)^{\vec{j}}$$ and $$f_{a'_{J_1}}=\sum_{\vec{j}:
j_{\mathfrak{p}}\leq k_{\iota_\mathfrak{p}}-2} c_{\vec{j}}
(x-a')^{\vec{j}}.$$ Write
$$ f_{a_{J_1}}-f_{a'_{J_1}} = \sum_{\vec{j}:
j_{\mathfrak{p}}\leq k_{\iota_\mathfrak{p}}-2} b_{\vec{j}}
(x-a)^{\vec{j}} .$$ Then we have
\begin{equation} \label{eq:coef-b}
b_{\vec{j}}\prod\limits_{\mathfrak{p}\in J_1}
\omega_\mathfrak{p}^{m_{\mathfrak{p}}j_\mathfrak{p}} \in
 \prod\limits_{\mathfrak{p}\in
J_1}\iota_\mathfrak{p}(\omega_\mathfrak{p})^{m_{\mathfrak{p}}(k_\mathfrak{p}-1)}
\cdot I(f, a_{J_1}, m_{J_1}).\end{equation}

It follows from (\ref{eq:coef-b}) and the estimate in Proposition
\ref{prop:estimate} that, for any analytic function $f$ on an open
set $U$ and sufficient large $m_{J_1}$, writing $U$ as disjoint
union of $\mathcal{U}(a_{i,J_1}, m_{J_1})$, the ``$m_{J_1}$-th
Riemann sum''
$$ \sum_{i} \int_{\mathcal{U}(a_{i,J_1}, m_{J_1})} f_{a_{i,J}} \mu_c^{\iota_J}  $$
converges when $m_{J_1}\rightarrow \infty$, yielding our integral
$\int_{U}f \mu_c^{\iota_{J_1}}$.
\end{proof}

\begin{cor}\label{cor:measure} When $k_\mathfrak{p}=2$ for each $\mathfrak{p}\in J_1$,
if $c$ is bounded, then $\int ? \mu_c^{\iota_{J_1}}$ extends to all
compactly supported continuous functions on $F_{J_1}$.
\end{cor}
\begin{proof} By Proposition \ref{prop:estimate},
for any fixed compact open subset $U$ of $F_{J_1}$, there exists a
constant $C_U$ only depending on $U$ such that, for any locally
constant function $g$ on $F_{J_1}$ that is supported in $U$, we have
$$|\int g \mu_c^{\iota}|\leq C_U \max_{x\in U}|g(x)|. $$ So, for any
continuous function $g$ supported on $U$, we can take a series of
locally constant functions $g_i$ supported on $U$, such that
$g_i\rightarrow g$. Then we put $\int g
\mu_c^{\iota_{J_1}}=\lim\limits_{\;\; i\;\;}\int g_i
\mu_c^{\iota_{J_1}}$. When $g$ is further locally analytic, the
integral coincides with that defined in the proof of Proposition
\ref{prop:har-to-measure}. Indeed, in this case we may take $g_i$ as
in that proof.
\end{proof}

\begin{prop}\label{prop:int-vanish}
Fix $\mathfrak{p}\in J_1$. For every locally analytic function
$f^{\mathfrak{p}}$ on
$$ \mathbf{P}_{J_1\backslash\{\mathfrak{p}\}}:=\prod_{\mathfrak{p}' \in J_1:
\mathfrak{p}'\neq \mathfrak{p}} \mathbf{P}^1(F_{\mathfrak{p}'}) $$
and each integer $j\in \{0, \cdots, k_{\mathfrak{p}}-2\}$  we have
$$ \int (f^{\mathfrak{p}} \otimes x_\mathfrak{p}^j) \mu_c^{\iota_{J_1}} =0. $$
\end{prop}
\begin{proof}  We easily
reduce to the case when $f^{\mathfrak{p}}$ is an analytic function
on a disc $\mathcal{U}_{e'}$ ($e'\in \prod_{\mathfrak{p}'\neq
\mathfrak{p}} \mathcal{E}_{\mathfrak{p}'}$) and is $0$ outside of
this disc. As the integral is defined as limit of ``Riemann sum'',
it suffices to consider the truncation $(f^{\mathfrak{p}})_{a'}$
(for some $a'\in \mathcal{U}_{e'}$) instead of $f^{\mathfrak{p}}$.
But in this case, the assertion is exactly (\ref{it:har-b}) in
Definition \ref{defn:harmonic}.
\end{proof}

\begin{prop}\label{prop:auto-change}
For any
$g=(i_\mathfrak{p}\mathfrak{i}_\mathfrak{p}^{-1}(\wvec{a_{\mathfrak{p}}}{b_{\mathfrak{p}}}{c_{\mathfrak{p}}}{d_{\mathfrak{p}}}))_\mathfrak{p}\in
G_p$ and any locally analytic function $f$ on $\mathbf{P}_{J_1}$ we
have
$$  \check{\rho}_{\vec{k}}(g) \int f(g x ) \cdot
\prod_{\mathfrak{p}\in
J_1}\iota_{\mathfrak{p}}\left(\frac{(c_\mathfrak{p}x_{\mathfrak{p}}+d_\mathfrak{p})^{k_{\iota_\mathfrak{p}}-2}}
{(a_\mathfrak{p}d_\mathfrak{p}-b_\mathfrak{p}c_\mathfrak{p})^{\frac{
k_{ \iota_\mathfrak{p}} }{2}-1 }}\right) \mu^{\iota_{J_1}}_{c}(x) =
\int f(x) \mu^{\iota_{J_1}}_{g\star c}(x). $$ In particular, if $c$
is $\Gamma$-invariant, then for any
$\gamma=(\wvec{a_{\mathfrak{p}}}{b_{\mathfrak{p}}}{c_{\mathfrak{p}}}{d_{\mathfrak{p}}})_\mathfrak{p}\in
\Gamma$ we have \begin{equation}\label{eq:gamma-inv}
\check{\rho}_{\vec{k}}(\gamma) \int f(\gamma x ) \cdot
\prod_{\mathfrak{p}\in
J_1}\iota_{\mathfrak{p}}\left(\frac{(c_\mathfrak{p}x_{\mathfrak{p}}+d_\mathfrak{p})^{k_{\iota_\mathfrak{p}}-2}}
{(a_\mathfrak{p}d_\mathfrak{p}-b_\mathfrak{p}c_\mathfrak{p})^{\frac{
k_{ \iota_\mathfrak{p}} }{2}-1 }}\right) \mu^{\iota_{J_1}}_{c}(x) =
\int f(x) \mu^{\iota_{J_1}}_{c}(x). \end{equation}
\end{prop}
\begin{proof} This follows from (\ref{eq:invariant})
and a limit argument similar to that in the proof of Proposition
\ref{prop:int-vanish}.
\end{proof}

\begin{rem} \label{rem:gamma-inv} In the case when $k_\mathfrak{p}=2$ for each $\mathfrak{p}\in J_1$, (\ref{eq:gamma-inv}) says
that $\mu^{\iota_{J_1}}_c$ is $\Gamma$-invariant. \end{rem}

\subsection{Harmonic cocycle valued modular forms}
\label{ss:harm-mod-form}

Assume that  $\pi_\mathfrak{p}=\sigma(|\cdot|^{1/2},|\cdot|^{-1/2})$
for each $\mathfrak{p}\in J_1$. In this case we have
$\mathrm{U}_\mathfrak{p}\varphi_\mathfrak{p}=\varphi_\mathfrak{p}$
and $w_\mathfrak{p}\varphi_\mathfrak{p}=-\varphi_\mathfrak{p}$ for
each $\mathfrak{p}\in J_1$.

\begin{defn}\label{defn:J-mod-form} A {\it $J_1$-typle ``harmonic cocycle'' valued modular form of trivial central character, weight $\vec{k}$ and level $U=U^{J_1} \prod\limits_{\mathfrak{p}\in
J_1}U_{0}(\omega_\mathfrak{p})$} is a function $c:
G^{\infty,J_1}\times \mathcal{E}_{J_1}\rightarrow L_{\vec{k}}(E)$
that satisfies the following conditions:
\begin{enumerate}
\item For each $g\in G^{\infty,J_1}$, $c(g,\cdot)$ is a harmonic
cocycle.
\item \label{it:J-center} If $z\in
(F^{\infty,J_1})^{\times}$, then $c(z g,\cdot)= c(g, \cdot)$.
\item\label{it-J-mod-c} If $h\in U^{J_1}  $,
then $$ c( gh, \cdot )= (h_p^{-1}\star c)(g, \cdot). $$ That is  $
c(gh,\cdot)= \check{\rho}_{\vec{k}}(h_p^{-1}) c(g, \cdot). $
\item\label{it:J-mod-d} If $x_0\in G(\BQ)$, then $$ c(x_0 g, \cdot )=(x_{0,J_1}\star c)(g,
\cdot). $$ That is $ c(x_0 g, e)= \check{\rho}_{\vec{k}}(x_{0,J_1})
c(g, x_{0,J_1}^{-1}e). $
\end{enumerate}
Let $C^{1}_{J_1}(U^{J_1}, L_{\vec{k}}(E))$ be the space of such
``harmonic cocycle'' valued modular forms.
\end{defn}

For each $g\in G^{\infty,J_1}$ put $$\widetilde{\Gamma}^{J_1}_g=\{
\gamma \in G(\BQ): \gamma_{\mathfrak{l}} \in  g
U_{\mathfrak{l}}g^{-1} \text{ for }\mathfrak{l}\notin J_1 \}.$$ We
embed $\widetilde{\Gamma}^{J_1}_g$ into $G_p$ by
$$\iota^{J_1}_g: \widetilde{\Gamma}_g^{J_1}\rightarrow G_p, \hskip 10pt \gamma \mapsto g_p^{-1}\gamma_p g_p, $$
and put $\Gamma^{J_1}_p=\iota^{J_1}_g(\widetilde{\Gamma}^{J_1}_g)$.
Here, we consider $g_p$ as an element of $G_p$ whose
$J_1$-components are all $1$.

\begin{rem} Though $\widetilde{\Gamma}^{J_1}_g$ only depends on the coset $g U^{J_1}$,  both $\iota^{J_1}_g$ and $\Gamma^{J_1}_g$ do depend on $g_p$. Indeed, if $y\in U^{J_1}$, then $\Gamma^{J_1}_{gy}=y_p^{-1}\Gamma^{J_1} y_p$. \end{rem}

\begin{lem} For each $g\in G^{\infty,J_1}$, $c(g, \cdot)$ is
$ \Gamma^{J_1}_g $-invariant.
\end{lem}
\begin{proof} Let $\gamma$ be in $\widetilde{\Gamma}_g^{J_1}$. Then $$(\iota^{J_1}_g(\gamma) \star c)(g, e)=
\check{\rho}_{\vec{k}}(g_p^{-1}\gamma_p^{J_1} g_p)
\check{\rho}_{\vec{k}}(\gamma_{J_1})c(g, \gamma_{J_1}^{-1}e)=
\check{\rho}_{\vec{k}}(g_p^{-1}\gamma_p^{J_1} g_p) c(\gamma g, e).$$
The latter equality follows from  Definition \ref{defn:J-mod-form}
(\ref{it:J-mod-d}). As $\gamma\in \widetilde{\Gamma}^{J_1}_g$, we
have $g^{-1}\gamma g \in U_\mathfrak{l}$ for $\mathfrak{l}\notin
J_1$. Thus by Definition \ref{defn:J-mod-form} (\ref{it-J-mod-c}) we
have
$$ \check{\rho}_{\vec{k}}(g_p^{-1}\gamma_p^{J_1} g_p) c(\gamma g,
e)= c(\gamma g (g^{-1}\gamma g)^{-1}, e)=c(g,e),$$ as desired.
\end{proof}

There exist Hecke operators $\mathrm{T}_\mathfrak{l}$
($\mathfrak{l}\nmid p$, $U_\mathfrak{l}$ is maximal compact in
$G_{\mathfrak{l}}$) on $C^{1}_{J_1}(U^{J_1}, L_{\vec{k}}(E))$.  For
such $\mathfrak{l}$ we fix an isomorphism $G_\mathfrak{l}\cong
\mathrm{GL}_2(F_\mathfrak{l})$ such that $\Sigma_\mathfrak{l}\cong
\mathrm{GL}_2(\mathcal{O}_{F_\mathfrak{l}})$ and write
$$ \mathrm{GL}_2(\mathcal{O}_{F_\mathfrak{l}})
\wvec{\omega_\mathfrak{l}}{0}{0}{1}
\mathrm{GL}_2(\mathcal{O}_{F_\mathfrak{l}}) =\prod t_{i}
\mathrm{GL}_2(\mathcal{O}_{F_\mathfrak{l}}) .  $$ Then we define
$$ (\mathrm{T}_\mathfrak{l} \cdot c) (g, e) = \sum_{i} c(g t_i, e) .  $$
Let $\mathbb{T}^{p}_U$ be the algebra generated by
$\mathrm{T}_\mathfrak{l}$ ($\mathfrak{l}\nmid p$, $U_\mathfrak{l}$
is maximal compact in $G_{\mathfrak{l}}$).

We attach to each ``harmonic cocycle'' valued modular form a
$p$-adic modular form. For each $c\in C^{1}_{J_1}(U^{J_1},
L_{\vec{k}}(E))$ let $\widehat{f}_c$ be the function on $G^\infty$
defined by
$$\widehat{f}_c(g) = (g_{J_1}^{-1}\star  c)(g^{J_1}, e_{J_1,*})= \check{\rho}_{\vec{k}}(g_{J_1}^{-1})  c(g^{J_1}, g_{J_1}e_{J_1,*}) .  $$

\begin{prop} \label{prop:har-to-mod}
\begin{enumerate}
\item If $c\in C^{1}_{J}(U^{J_1}, L_{\vec{k}}(E))$, then $\widehat{f}_c \in M_{\vec{k}}(U,
E)$.
\item The map $$C^{1}_{J}(U^{J_1}, L_{\vec{k}}(E))\rightarrow M_{\vec{k}}(U,
E) \hskip 10pt   c\mapsto \widehat{f}_c$$ is injective and
$\mathbb{T}_U^{p}$-equivariant.
\item For $\mathfrak{p}\in J_1$ we have $\mathrm{U}_\mathfrak{p}\widehat{f}_c
= \widehat{f}_c$ and $w_\mathfrak{p}\widehat{f}_c= - \widehat{f}_c$.
\end{enumerate}
\end{prop}

Conversely, if $\widehat{f} \in M_{\vec{k}}(U, E)$ satisfies
\begin{equation} \label{eq:atkin} w_\mathfrak{p} \widehat{f} = -
 \widehat{f},
\end{equation} and
\begin{equation} \label{eq:Up-eigenvalue}
\mathrm{U}_\mathfrak{p} \widehat{f} =  \widehat{f},
\end{equation}
for each $\mathfrak{p}\in J_1$, we associate to it a ``harmonic
cocyle'' valued modular form $c_{\widehat{f}}$ defined by
$$ c_{\widehat{f}}(g^{J_1}, e) = \check{\rho}_{\vec{k}}(g_e)\widehat{f}( g^{J_1}g_e)
, $$ where $g_e$ is an element of $G_{J_1}$ such that $g_e
(e_{J_1,*})=e$. Note that the value $c_{\widehat{f}}(g^{J_1}, e)$
does not depend on the choice of $g_e$.

\begin{prop}\label{prop:mod-to-har}
\begin{enumerate} \item
If $\widehat{f}\in M_{\vec{k}}(U, E)$ satisfies $($\ref{eq:atkin}$)$
and $($\ref{eq:Up-eigenvalue}$)$, then $$ c_{\widehat{f}}\in
C^{1}_{J_1}(U^{J_1}, L_{\vec{k}}(E)). $$
\item For $c\in C^{1}_{J_1}(U^{J_1}, L_{\vec{k}}(E))$ we have
$c_{\widehat{f}_c}=c$.
\end{enumerate}
\end{prop}

\noindent{\it Proof of Propositions \ref{prop:har-to-mod} and
\ref{prop:mod-to-har}.} It is easy to check that $c\mapsto
\widehat{f}_c$ is a one-to-one correspondence between the set of
functions $c: G^{\infty,J_1}\times \mathcal{E}_{J_1}\rightarrow
L_{\vec{k}}(E)$ that satisfy (\ref{it:J-center}), (\ref{it-J-mod-c})
and (\ref{it:J-mod-d}), and the set $M_{\vec{k}}(U, E)$. The
function $c$ is harmonic cocycle valued, i.e. $c(g, \cdot)$ for each
$g\in G^{\infty,J_1}$ is a harmonic cocycle, if and only if
$\mathrm{U}_\mathfrak{p}\widehat{f}_c=\widehat{f}_c$ and
$w_\mathfrak{p}\widehat{f}_c=-\widehat{f}_c$ for each
$\mathfrak{p}\in J_1$. That $c\mapsto \widehat{f}_c$ is
$\mathbb{T}^p_U$-equivariant follows directly from
the definitions of Hecke operators. \qed \\

The following is a direct consequence of Lemma \ref{lem:control-1}
and Proposition \ref{prop:har-to-mod}.

\begin{cor} For any norm $|\cdot|$
on $L_{\vec{k}}(E)$ we have
$$ \sup_{ g=(g_{J_1}, g^{J_1})\in G^\infty } |(g_{J_1}\star c)(g^{J_1}, e_{J_1,*})|
<\infty.
$$ In particular, for each $g^{J_1}\in G^{J_1}$, the harmonic cocycle $c(g^{J_1},
\cdot)$ is bounded.
\end{cor}

Assume that $\emptyset \subsetneq J_2\subsetneq J_1$. Assume that
$\widehat{f}\in M_{\vec{k}}(U, E)$ satisfies (\ref{eq:atkin}) and
(\ref{eq:Up-eigenvalue}) for each $\mathfrak{p}\in J_1$. Let $c_{1}$
and $c_{2}$ be the $J_1$-type ``harmonic cocycle'' valued modular
form and the $J_2$-type one attached to $\widehat{f}$  respectively.

Let $g$ be an element of $G_{J_1\backslash J_2}$.

\begin{prop}\label{prop:vary-J} Suppose $k_\mathfrak{p}=2$ for each $\mathfrak{p}$ above $p$.
For any $g\in G_{J_1\backslash J_2}$, $h\in G^{J_1}$, and any
compactly supported continuous function $\alpha$ on $F_{J_2}\subset
\mathbf{P}_{J_2}$ we have
$$ \int (\alpha \otimes 1_{g \mathcal{U}_{J_1\backslash J_2,*}}) \mu_{c_1(h, \cdot)} = \int \alpha \mu_{c_2(gh, \cdot)} . $$
\end{prop}
\begin{proof} We only need to consider the case that $\alpha$ is locally
constant, say $\alpha$ is of the form $\alpha=1_{x \mathcal{U}_{e_2,
*}}$, $x\in G_{J_2}$. In this case, both sides of the above equality
is $ \widehat{f}(xgh)$.
\end{proof}

\subsection{Cohomological valued modular forms and Schneider's
morphism}\label{ss:coh-sch}

We fix a prime $\mathfrak{p}\in J_0$ and an embedding
$\iota_\mathfrak{p}: F_\mathfrak{p}\hookrightarrow E$. In this
subsection we define cohomological valued modular form (of level
$U$).

For each $g\in G^{\infty,\mathfrak{p}}$ we have the groups
$\widetilde{\Gamma}^{\mathfrak{p}}_g$ and $\Gamma^{\mathfrak{p}}_g$.
As $\mathfrak{p}$ is clear to us, we denote them by
$\widetilde{\Gamma}_g$ and $\Gamma_g$. As $L_{\vec{k}}(E)$ is a
$\Gamma_g$-module, we can form the cohomological group
$H^1(\Gamma_g, L_{\vec{k}}(E))$. This group consists of equivalence
classes of the $L_{\vec{k}}(E)$-valued $1$-cocycles on $\Gamma_g$.

For a subgroup $\Gamma$ of $G_p$, and $h\in G_p$, we have the map
\begin{eqnarray*} r_h : H^1(\Gamma, L_{\vec{k}}(E))
&\rightarrow& H^1(h^{-1}\Gamma h, L_{\vec{k}}(E)) , \\
  \phi &\mapsto&
( r_h \phi)(\gamma')= \check{\rho}_{\vec{k}}(h)^{-1}\phi(h \gamma'
h^{-1}).
\end{eqnarray*}
If  $h\in U^{\mathfrak{p}}$ we have
$\widetilde{\Gamma}_{gh}=\widetilde{\Gamma}_g$ and $\Gamma_{gh} =
h_p^{-1} \Gamma_g h_p .$ So we have the map
$$r_{h_p}: H^1(\Gamma_g, L_{\vec{k}}(E))\rightarrow H^1(\Gamma_{gh}, L_{\vec{k}}(E)). $$
If $g'=xg$ with $x\in G(\BQ)$, then
$\widetilde{\Gamma}_{g'}=x\widetilde{\Gamma}_g x^{-1}$ and $
\Gamma_{g'} =x_{\mathfrak{p}} \Gamma_g x_{\mathfrak{p}}^{-1}$. So we
have the map
$$r_{x_{\mathfrak{p}}^{-1}}: H^1(\Gamma_g, L_{\vec{k}}(E))\rightarrow H^1(\Gamma_{xg}, L_{\vec{k}}(E)).$$

\begin{defn}\label{defn:coh-mod-form}
By a {\it cohomological valued modular form}  $($on
$G^{\infty,\mathfrak{p}}$ of {\it level $U)$} we mean a function $f$
on $G^{\infty,\mathfrak{p}}$ which satisfies the following
conditions:
\begin{enumerate}
\item $f(g)$ is in $H^1(\Gamma_g, V)$.
\item\label{it:coh-center} If $z\in
(F^{\infty,\mathfrak{p}})^{\times}$, then $f(zg)=f(g)$.
\item\label{it:coh-Up} If $h\in U^{\mathfrak{p}}$, then $f(gh)=(r_{h_p}f)(g)$.
\item\label{it:coh-Q} If $x\in G(\BQ)$, then $f(x g)=(r_{x_{\mathfrak{p}}^{-1}}f)(g)$.
\end{enumerate}
Let $MH^1_{\mathfrak{p}}(U^\mathfrak{p}, L_{\vec{k}}(E))$ be the
space of such cohomological valued modular forms.
\end{defn}

There exists a $\mathbb{T}^{p}_U$-action on
$MH^1_{\mathfrak{p}}(U^\mathfrak{p}, L_{\vec{k}}(E))$.

Let $\mathfrak{l}\nmid p$ be a prime of $F$ such that
$U_\mathfrak{l}\cong \mathrm{GL}_2(\mathcal{O}_{F_\mathfrak{l}})$.
Let $b_{\mathfrak{l}}$ be in
$U_{\mathfrak{l}}\wvec{\omega_{\mathfrak{l}}}{0}{0}{1}U_{\mathfrak{l}}$.
For any $g\in G^{\infty,\mathfrak{p}}$, let $\Gamma'_g$ be the
subgroup
$$\Gamma'_g=\Gamma_g \cap  \Gamma_{g b_{\mathfrak{l}}}$$ of $G_p$.  We have the
restriction map
$$\mathrm{Res}_{\Gamma_{gb_{\mathfrak{l}}}/\Gamma'_g}: H^1(\Gamma_{gb_{\mathfrak{l}}}, L_{\vec{k}}(E))\rightarrow H^1(\Gamma'_{g},
L_{\vec{k}}(E))$$ and the corestriction map
$$\mathrm{Cor}_{\Gamma_{g}/\Gamma'_g}: H^1(\Gamma'_g, L_{\vec{k}}(E))\rightarrow H^1(\Gamma_g,
L_{\vec{k}}(E)).$$ We define the action of $\mathrm{T}_\mathfrak{l}$
on $MH^1_{\mathfrak{p}}(U^\mathfrak{p}, L_{\vec{k}}(E))$ by
$$ (\mathrm{T}_\mathfrak{l}  f) (g)  =  \mathrm{Cor}_{\Gamma_g/\Gamma'_g}
\mathrm{Res}_{\Gamma_{gb_{\mathfrak{l}}}/\Gamma'_g}f (g
b_{\mathfrak{l}})    $$  for each $f\in
MH^1_{\mathfrak{p}}(U^\mathfrak{p}, L_{\vec{k}}(E))$.

Now, we define the Schneider morphism
$$\kappa^{\mathrm{sch}}: C^1_{\mathfrak{p}}(U^{\mathfrak{p}}, L_{\vec{k}}(E))\rightarrow MH^1_{\mathfrak{p}}(U^\mathfrak{p},
L_{\vec{k}}(E)).$$ For each $c\in
C^1_{\mathfrak{p}}(U^{\mathfrak{p}}, L_{\vec{k}}(E))$ and $g\in
U^{\mathfrak{p}}$, noting that $c(g, \cdot)$ is $
\Gamma_g$-invariant, we put
$$\kappa^{\mathrm{sch}}_c(g)=\delta_{ \Gamma_g }(c(g,\cdot)) \in H^1(\Gamma_g,
L_{\vec{k}}(E)),$$ where $\delta_{ \Gamma_g }$ is the map
(\ref{eq:delta}).

\begin{prop}\label{prop:schneider}
\begin{enumerate}
\item\label{it:har-coh-a} If $c\in C^1_{\mathfrak{p}}(U^{\mathfrak{p}}, L_{\vec{k}}(E))$, then $\kappa^{\mathrm{sch}}_c$ is in $MH^1_{\mathfrak{p}}(U^\mathfrak{p},
L_{\vec{k}}(E))$.
\item\label{it:har-coh-b} The map \begin{eqnarray*}\kappa^{\mathrm{sch}}: C^1_{\mathfrak{p}}(U^{\mathfrak{p}}, L_{\vec{k}}(E)) &\rightarrow& MH^1_{\mathfrak{p}}(U^\mathfrak{p},
L_{\vec{k}}(E))\\ c & \mapsto & \kappa^{\mathrm{sch}}_c
\end{eqnarray*} is a $\mathbb{T}^p_U$-equivariant isomorphism.
\end{enumerate}
\end{prop}
\begin{proof} First we show that $\kappa^{\mathrm{sch}}_c$ is really in $MH^1_{\mathfrak{p}}(U^\mathfrak{p},
L_{\vec{k}}(E))$. For this we only need to deduce (a)-(d) in
Definition \ref{defn:coh-mod-form} term by term from Definition
\ref{defn:J-mod-form} (with $J=\{\mathfrak{p}\}$ there). By
definition we already have (a). If $z$ is in the center of
$G^{\infty,\mathfrak{p}}$, then $\Gamma_g=\Gamma_{zg}$, so $\delta_{
\Gamma_g }=\delta_{ \Gamma_{zg} }$. By Definition
\ref{defn:J-mod-form} (\ref{it:J-center}) (with $J=\{\mathfrak{p}\}$
there) we have $c(zg, \cdot)=c(g, \cdot)$. Thus
$\kappa^{\mathrm{sch}}_c(g)=\kappa^{\mathrm{sch}}_c(zg)$, which
yields (\ref{it:coh-center}). Items (\ref{it:coh-Up}) and
(\ref{it:coh-Q}) follow from Definition \ref{defn:J-mod-form}
(\ref{it-J-mod-c}, \ref{it:J-mod-d}) and Lemma \ref{lem:HA}
(\ref{it:HA-c}).

Next we prove (\ref{it:har-coh-b}). Put
$$ \widetilde{\Gamma}'_g= \widetilde{\Gamma}_g \cap \widetilde{\Gamma}_{gb_{\mathfrak{l}}}
=\{\gamma\in G(\BQ):  \gamma_{\mathfrak{l}} \in  g_{\mathfrak{l}}
(U_{\mathfrak{l}}\cap b_{\mathfrak{l}}
U_{\mathfrak{l}}b_{\mathfrak{l}}^{-1} ) g_{\mathfrak{l}}^{-1} \text{
and } \gamma_{\mathfrak{l}'} \in  g_{\mathfrak{l}'}
 U_{\mathfrak{l}'}  g_{\mathfrak{l}'}^{-1}
\text{ for }\mathfrak{l}'\neq \mathfrak{l, p} \}.
$$  Then $\Gamma'_g=\iota_g^{\mathfrak{p}}( \widetilde{\Gamma}'_g)$.
We decompose $\widetilde{\Gamma}_g$ into cosets
\begin{equation}\label{eq:decom}
\widetilde{\Gamma}_g = \coprod_i    \alpha_i \widetilde{\Gamma}'_g.
\end{equation} Via $\iota^{\mathfrak{p}}_g$ we obtain
$$  \Gamma _g = \coprod_i   g_{p}^{-1}\alpha_{i,p}g_p    \Gamma '_g
.$$ Thus \begin{eqnarray*} \mathrm{Cor}_{\Gamma_g/\Gamma'_g}
c(gb_{\mathfrak{l}}, e) &=&\sum_i  g_{p}^{-1}\alpha_{i,p}g_p \star
c(gb_{\mathfrak{l}}, e)\\ &=&\sum_i
\check{\rho}_{\vec{k}}(g_p^{-1}\alpha_{i,p}^\mathfrak{p} g_p)
\check{\rho}_{\vec{k}}(\alpha_{i,\mathfrak{p}})c(g b_{\mathfrak{l}},
\alpha_{i,\mathfrak{p}}^{-1}e) \\ &=&
\check{\rho}_{\vec{k}}(g_p^{-1}\alpha_{i,p}^\mathfrak{p} g_p)
c(\alpha_i gb_{\mathfrak{l}}, e).\end{eqnarray*}

By Lemma \ref{lem:coset-isom} below, from (\ref{eq:decom}) the
decomposition of $\widetilde{\Gamma}_g$, we obtain
$$
b_{\mathfrak{l}}^{-1}U^{\mathfrak{p}} b_{\mathfrak{l}} = \coprod_i
b_{\mathfrak{l}}^{-1} g^{-1}\alpha_i g b_{\mathfrak{l}} \cdot
(b_{\mathfrak{l}}^{-1}U^{\mathfrak{p}} b_{\mathfrak{l}}\cap
U^{\mathfrak{p}} ).$$ Then \begin{eqnarray*}
\mathrm{T}_{\mathfrak{l}}c(g,
\cdot) 
&=& \sum_i \check{\rho}_{\vec{k}}(g_p^{-1}\alpha_{i,p}^\mathfrak{p}
g_p) c(g b_{\mathfrak{l}}\cdot b_{\mathfrak{l}}^{-1} g^{-1}\alpha_i
gb_{\mathfrak{l}}, \cdot) \\ &=&\sum_i
\check{\rho}_{\vec{k}}(g_p^{-1}\alpha_{i,p}^\mathfrak{p} g_p)
c(\alpha_i gb_{\mathfrak{l}}, \cdot) =
\mathrm{Cor}_{\Gamma_g/\Gamma'_g} c(gb_{\mathfrak{l}}, e).
\end{eqnarray*} By Lemma \ref{lem:HA} (\ref{it:HA-a}, \ref{it:HA-b})
\begin{eqnarray*} (\mathrm{T}_\mathfrak{l}  \kappa^{\mathrm{sch}}_c) (g)  &=&
\mathrm{Cor}_{\Gamma_g/\Gamma'_g}
\mathrm{Res}_{\Gamma_{gb_{\mathfrak{l}}}/\Gamma'_g}
(\kappa^{\mathrm{sch}}_c)(g b_{\mathfrak{l}}) \\ & =&
\mathrm{Cor}_{\Gamma_g/\Gamma'_g}
\mathrm{Res}_{\Gamma_{gb_{\mathfrak{l}}}/\Gamma'_g}
\delta_{\Gamma_{gb_{\mathfrak{l}}}}(c(gb_{\mathfrak{l}}, \cdot)) \\
&=&\delta_{\Gamma_{g}}(\mathrm{Cor}_{\Gamma_g/\Gamma'_g}
\mathrm{Res}_{\Gamma_{gb_{\mathfrak{l}}}/\Gamma'_g}
(c(gb_{\mathfrak{l}}, \cdot))) = \delta_{\Gamma_{g}}
((\mathrm{T}_{\mathfrak{l}}c)(g , \cdot)).
\end{eqnarray*} This proves that $\kappa^{\mathrm{sch}}$ is
$\mathbb{T}^p_U$-equivariant.

That $\kappa^{\mathrm{sch}}$ is an isomrphism follows from
\cite[Proposition 2.9]{CMP}. Indeed, in loc. cit. it is showed that
$$\delta_{\Gamma_g}: C^1_{\mathrm{har}}(\mathcal{T}_{\mathfrak{p}},
L_{\vec{k}}(E))^{\Gamma_g}\rightarrow H^1(\Gamma_g,
L_{\vec{k}}(E))$$ is an isomorphism.
\end{proof}

\begin{lem}\label{lem:coset-isom} $gU^{\mathfrak{p}}g^{-1}/
g\Big(U^{\mathfrak{p},\mathfrak{l}}(U_{\mathfrak{l}}\cap
b_{\mathfrak{l}}U_{\mathfrak{l}}b_{\mathfrak{l}}^{-1})\Big)g^{-1}\cong
\tilde{\Gamma}_g/\tilde{\Gamma}'_g$.
\end{lem}
\begin{proof} This follows from the strong approximation theorem
for $\mathrm{SL}_1(B)$, the algebraic group with $\BQ$-points
$$B^{\times, 1}=\{x\in B: N_{B/F}(x)=1\},$$ (see \cite{Vig}). Here,
$N_{B/F}: B^\times \rightarrow F^\times$ is the reduced norm. By
this strong approximation theorem we obtain that, for any compact
open subgroup $V$ of $G^{\infty,\mathfrak{p}}$, one has $$
\sharp\Big(G(\BQ)\backslash G(\BA_{\BQ}) /G_\infty G_\mathfrak{p} V
\Big)=\sharp\left(F^\times \backslash
\widehat{F}^\times/F_\mathfrak{p}^\times N_{B/F}(V)\right).$$ So,
from the fact
$$N_{B/F}( gU^{\mathfrak{p}}g^{-1} )=N_{B/F} \Big( g( U^{\mathfrak{p}, \mathfrak{l}}(U_{\mathfrak{l}}\cap
b_{\mathfrak{l}}U_{\mathfrak{l}}b_{\mathfrak{l}}^{-1})
)g^{-1}\Big),$$ we obtain
$$ gU^{\mathfrak{p}}g^{-1} = (gU^{\mathfrak{p}}g^{-1}\cap G(\BQ)) \cdot g\Big(U^{\mathfrak{p},\mathfrak{l}}(U_{\mathfrak{l}}\cap
b_{\mathfrak{l}}U_{\mathfrak{l}}b_{\mathfrak{l}}^{-1})\Big)g^{-1},$$
or equivalently, the coset $gU^{\mathfrak{p}}g^{-1}/
g\Big(U^{\mathfrak{p},\mathfrak{l}}(U_{\mathfrak{l}}\cap
b_{\mathfrak{l}}U_{\mathfrak{l}}b_{\mathfrak{l}}^{-1})\Big)g^{-1}$
has a set of representatives in
$\widetilde{\Gamma}_g=gU^{\mathfrak{p}}g^{-1}\cap G(\BQ)$.
\end{proof}

\subsection{Coleman's morphism and $L$-invariant}\label{ss:coleman}

Let $\mathfrak{p}$ and $\iota_{\mathfrak{p}}$ be as in Section
\ref{ss:coh-sch}. For any $u\in \BC_p^\times$ that is not a unit,
let $\log_u:\BC_p^\times\rightarrow \BC_p$  denote the $p$-adic
logarithm such that $\log_u(u)=0$.

In Section \ref{ss:har-to-rigid} we attach to $c\in
C^1_{\mathrm{har}}(\mathcal{T}_\mathfrak{p}, L_{\vec{k}}(E))$ a
rigid distribution $\mu^{\iota_{\mathfrak{p}}}_{c}$. For a point
$z_0\in
\mathcal{H}_{\iota_\mathfrak{p}}=\BC_p-\iota_\mathfrak{p}(F_\mathfrak{p})$,
we define an $L_{\vec{k}}(E)$-valued function $\lambda_{z_0,c;u}$ on
$G_p$ by
\begin{equation} \label{eq:int-lambda}
\langle\lambda_{z_0,c;u}(g), P\otimes Q \rangle = \langle\int
P(t)\log_u(\frac{t-g_\mathfrak{p}
z_0}{t-z_0})\mu^{\iota_{\mathfrak{p}}}_{c}, Q \rangle \end{equation}
with $P\in V_{\iota_{\mathfrak{p}}}(E)$ and $Q\in
V^{\iota_{\mathfrak{p}}}(E)$. Here $P\mapsto P(t)$ is the linear map
defined by
$X_\mathfrak{p}^jY_\mathfrak{p}^{k_{\iota_\mathfrak{p}}-2-j}\mapsto
t^j$.

\begin{lem} \label{lem:lambda}
\begin{enumerate}
\item\label{it:lambda-0}
For any $x\in G_{\mathfrak{p}}$ we have
$$  \langle \int  P (t)\log_u(\frac{x_{\mathfrak{p}}^{-1}\cdot
t-\gamma_\mathfrak{p} z_0}{x_{\mathfrak{p}}^{-1}\cdot
t-z_0})\mu^{\iota_{\mathfrak{p}}}_{ c } , Q\rangle = \langle \int  P
(t)\log_u(\frac{ t-x_{\mathfrak{p}}\gamma_\mathfrak{p} z_0}{
t-x_{\mathfrak{p}}z_0})\mu^{\iota_{\mathfrak{p}}}_{ c }  , Q\rangle
.$$
\item\label{it:lambda-a}
For $x\in G_p$ we have
$$ \check{\rho}_{\vec{k}}(x) \lambda_{z_0, c;u}(\gamma) = \lambda_{x_\mathfrak{p}\cdot z_0, x\star c;u}(x \gamma x^{-1}). $$
\item\label{it:lambda-b}
If $c$ is $\Gamma$-invariant, then for any $\gamma\in \Gamma$
and $g\in G_p$ we have
$$ \lambda_{z_0,c;u}(\gamma g) =\lambda_{z_0,c;u} (\gamma)+\check{\rho}_{\vec{k}}(\gamma) \lambda_{z_0,c;u}(g).  $$
\item \label{it:lambda-c} Let $z_1, z_2$ be two points of $\mathcal{H}_{\iota_\mathfrak{p}}$. If $c$ is $\Gamma$-invariant, then there
exists $\ell\in L_{\vec{k}}(E)$ that only depends on $z_1, z_2$ and
$c$ such that
$$ \lambda_{z_2, c;u}(\gamma)-\lambda_{z_1, c;u}(\gamma)=(\gamma-1)\ell
$$ for all $\gamma\in \Gamma$.
\end{enumerate}
\end{lem}
\begin{proof}
Write
$\mathfrak{i}_\mathfrak{p}i_\mathfrak{p}^{-1}(x_{\mathfrak{p}}^{-1})=\wvec{a}{b}{c}{d}$,
Then
$$ \frac{x_{\mathfrak{p}}^{-1}\cdot
t- \gamma_\mathfrak{p} z_0}{x_{\mathfrak{p}}^{-1}\cdot t-z_0}=
\frac{ t- x_{\mathfrak{p}} \gamma_{\mathfrak{p}} z_0 }{ t-
x_{\mathfrak{p}} z_0 } \frac{c \gamma_{\mathfrak{p}}z_0+d}{cz_0+d} .
$$ Thus
\begin{eqnarray*} && \langle \int  P (t)\log_u(\frac{x_{\mathfrak{p}}^{-1}\cdot
t-\gamma_\mathfrak{p} z_0}{x_{\mathfrak{p}}^{-1}\cdot
t-z_0})\mu^{\iota_{\mathfrak{p}}}_{  c } , Q\rangle
\\ &=& \langle \int  P (t)\log_u(\frac{
t-x_{\mathfrak{p}}\gamma_\mathfrak{p} z_0}{
t-x_{\mathfrak{p}}z_0})\mu^{\iota_{\mathfrak{p}}}_{  c }  , Q\rangle
+ \log_u(\frac{c \gamma_{\mathfrak{p}}z_0+d}{cz_0+d} ) \langle \int
P (t)\mu^{\iota_{\mathfrak{p}}}_{  c } , Q\rangle .
\end{eqnarray*} By Proposition  \ref{prop:int-vanish} (in the case of
$J=\{\mathfrak{p}\}$), the latter term of the right hand side of
this equality is zero. This proves (\ref{it:lambda-0}).

Let $x$ be in $G_p$.  We have
\begin{eqnarray*}
&&\langle\check{\rho}_{\vec{k}}(x) \lambda_{z_0,c;u}(\gamma)
,P\otimes Q\rangle \\ &=& \langle\check{\rho}_{\vec{k}}(x)\int
(\rho_{k_\mathfrak{p}}(x_{\mathfrak{p}}^{-1})P)(t)\log_u(\frac{t-\gamma_\mathfrak{p}
z_0}{t-z_0})\mu^{\iota_{\mathfrak{p}}}_{c}  , Q\rangle  \\
&=& \langle \int  P (t)\log_u(\frac{x_{\mathfrak{p}^{-1}}\cdot
t-\gamma_\mathfrak{p} z_0}{x_{\mathfrak{p}}^{-1}\cdot
t-z_0})\mu^{\iota_{\mathfrak{p}}}_{x\star c } , Q\rangle \hskip 20pt  \text{by Proposition }\ref{prop:auto-change} \\
&=& \langle \int  P (t)\log_u(\frac{ t- x_{\mathfrak{p}}
\gamma_\mathfrak{p} z_0 }{  t- x_{\mathfrak{p}}
z_0})\mu^{\iota_{\mathfrak{p}}}_{x\star c } , Q \rangle \hskip 20pt \text{by } (\text{\ref{it:lambda-0}}) \\
&=& \langle \lambda_{x_{\mathfrak{p}}\cdot z_0, x\star c;u}(x\gamma
x^{-1}), P\otimes Q\rangle
\end{eqnarray*}  which shows (\ref{it:lambda-a}).

Using (\ref{it:lambda-0}) we obtain \begin{eqnarray*} \langle
\lambda_{z_0,c;u}(\gamma g)- \lambda_{z_0,c;u}(\gamma), P\otimes Q
\rangle &=&  \langle\int P(t)\log_u(\frac{t-\gamma_\mathfrak{p}
g_{\mathfrak{p}} z_0}{t- \gamma_\mathfrak{p}
z_0})\mu^{\iota_{\mathfrak{p}}}_{c}, Q \rangle \\
&=&\langle\int P(t)\log_u(\frac{\gamma_\mathfrak{p}^{-1} t-
g_{\mathfrak{p}} z_0}{\gamma_\mathfrak{p}^{-1} t-
z_0})\mu^{\iota_{\mathfrak{p}}}_{c}, Q \rangle. \end{eqnarray*} By
Proposition \ref{prop:auto-change} (in the case of
$J=\{\mathfrak{p}\}$) and the fact $\gamma\star c=c$, we have
\begin{eqnarray*} && \langle\int P(t)\log_u(\frac{\gamma_\mathfrak{p}^{-1} t-
g_{\mathfrak{p}} z_0}{\gamma_\mathfrak{p}^{-1} t-
z_0})\mu^{\iota_{\mathfrak{p}}}_{c}, Q \rangle \\ &=&
\langle\check{\rho}_{\vec{k}}(\gamma)\int
(\rho_{k_\mathfrak{p}}(\gamma_{\mathfrak{p}}^{-1})P)(t)\log_u(\frac{
t- g_{\mathfrak{p}} z_0}{ t- z_0})\mu^{\iota_{\mathfrak{p}}}_{c}, Q
\rangle \\
&=& \langle \check{\rho}_{\vec{k}}(\gamma)\lambda_{z_0,c;u}(g),
P\otimes Q\rangle.
\end{eqnarray*}
So we get (\ref{it:lambda-b}).

For (\ref{it:lambda-c}) we have
\begin{eqnarray*} && \langle \lambda_{z_2,c;u}(\gamma)-\lambda_{z_1,c;u}(\gamma) , P\otimes Q
\rangle \\ &=& \langle \int P(t) [\log_u(\frac{t-\gamma_\mathfrak{p}
z_2}{t-z_2})-\log_u(\frac{t-\gamma_\mathfrak{p} z_1}{t-z_1})]
\mu_c^{\iota_\mathfrak{p}}, Q \rangle \\
&=& \langle \int P(t) [\log_u(\frac{t-\gamma_\mathfrak{p}
z_2}{t-\gamma_\mathfrak{p} z_1})-\log_u(\frac{t-  z_2}{t-z_1})]
\mu_c^{\iota_\mathfrak{p}}, Q \rangle\\
&=& \langle \int P(t) [\log_u(\frac{\gamma_\mathfrak{p}^{-1}t-
z_2}{\gamma_\mathfrak{p}^{-1}t- z_1})-\log_u(\frac{t-  z_2}{t-z_1})]
\mu_c^{\iota_\mathfrak{p}}, Q \rangle.
\end{eqnarray*} Let $\ell$ be the element of $L_{\vec{k}}(E)$
defined by
$$ \langle \ell, P\otimes Q\rangle = \langle  \int P(t) \log_u(\frac{t-  z_2}{t-z_1})
\mu_c^{\iota_\mathfrak{p}}, Q\rangle .$$ As $\gamma\star c=c$, by
Proposition \ref{prop:auto-change} we have \begin{eqnarray*}  &&
\langle\int P(t) \log_u(\frac{\gamma_\mathfrak{p}^{-1}t-
z_2}{\gamma_\mathfrak{p}^{-1}t- z_1})  \mu_c^{\iota_\mathfrak{p}},
Q\rangle \\ &=& \langle\check{\rho}_{\vec{k}}(\gamma)   \int
(\rho_{k_{\mathfrak{p}}}(\gamma_\mathfrak{p}^{-1})P)(t)
\log_u(\frac{ t- z_2}{ t- z_1})  \mu_c^{\iota_\mathfrak{p}},
Q\rangle \\ & = & \langle \check{\rho}_{\vec{k}}(\gamma) \ell,
P\otimes Q\rangle,
\end{eqnarray*} as expected.
\end{proof}

By Lemma \ref{lem:lambda} (\ref{it:lambda-b}), if $c$ is
$\Gamma$-invariant for a subgroup $\Gamma$ of $G_p$, then the
restriction of $\lambda_{z_0,c;u}$ to $\Gamma$ is  an
$L_{\vec{k}}(E)$-valued $1$-cocycle on $\Gamma$. By Lemma
\ref{lem:lambda} (\ref{it:lambda-c}) its class $[\lambda_{z_0,c;u}]$
in $H^1(\Gamma, L_{\vec{k}}(E))$ does not depend on the choice of
$z_0$. In this way we obtain Coleman's morphism
$$\lambda_{\Gamma;u}: C^1_{\mathrm{har}}(\mathcal{T}_\mathfrak{p}, L_{\vec{k}}(E))^\Gamma \rightarrow H^1(\Gamma, L_{\vec{k}}(E)), \hskip 10pt c\mapsto [\lambda_{z_0,c;u}].
$$

\begin{lem} \label{lem:coleman} Let  $\Gamma_1\subset \Gamma_2$ be two subgroups of
$G_p$.
\begin{enumerate} \item\label{it:col-a} We have the following commutative diagram
\[ \xymatrix{ C^1_{\mathrm{har}}(\mathcal{T}_\mathfrak{p},L_{\vec{k}}(E))^{\Gamma_2} \ar[r]^{\lambda_{\Gamma_2;u}}
\ar@{^(->}[d] & H^1(\Gamma_2,L_{\vec{k}}(E))
\ar[d]_{\mathrm{Res}_{\Gamma_2/\Gamma_1}}
\\ C^1_{\mathrm{har}}(\mathcal{T}_\mathfrak{p},L_{\vec{k}}(E))^{\Gamma_1} \ar[r]^{\lambda_{\Gamma_1;u}}
 & H^1(\Gamma_1,L_{\vec{k}}(E))  . }\]
 \item\label{it:col-b} When $[\Gamma_2:\Gamma_1]$ is finite, we have the following commutative diagram
\[ \xymatrix{ C^1_{\mathrm{har}}(\mathcal{T}_\mathfrak{p},L_{\vec{k}}(E))^{\Gamma_1} \ar[r]^{\lambda_{\Gamma_1;u}}
\ar[d]_{\mathrm{Cor}_{\Gamma_2/\Gamma_1}}  &
H^1(\Gamma_2,L_{\vec{k}}(E))
\ar[d]_{\mathrm{Cor}_{\Gamma_2/\Gamma_1}}
\\ C^1_{\mathrm{har}}(\mathcal{T}_\mathfrak{p},L_{\vec{k}}(E))^{\Gamma_2} \ar[r]^{\lambda_{\Gamma_2;u}}
 & H^1(\Gamma_2,L_{\vec{k}}(E)) . }\]
\end{enumerate}
\end{lem}
\begin{proof} Assertion (\ref{it:col-a}) follows from the definition. We
prove (\ref{it:col-b}). Let $c$ be in
$C^1_{\mathrm{har}}(\mathcal{T}_\mathfrak{p},L_{\vec{k}}(E))^{\Gamma_1}$.
Write $\Gamma_2=\sqcup x_j \Gamma_1$. Then
$\mathrm{Cor}_{\Gamma_2/\Gamma_1}(c)=\sum\limits_j x_j\star c$.

As $u$ in the subscript is clear, we will omit it in the following.

For any $\gamma\in \Gamma_2$ and any $i$, there exists
$j=j_{\gamma,i}$ (depending on $i$ and $\gamma$) such that $\gamma
x_i=x_j h_{\gamma, j}$ with $h_{\gamma,j}\in \Gamma_1$. Then
$$ \gamma\mapsto \sum_j \check{\rho}_{\vec{k}}(x_j)\lambda_{z_0,c}(h_{\gamma,j}) $$ is an $L_{\vec{k}}(E)$-valued $1$-cocycle on
$\Gamma_2$, and its class in $H^1(\Gamma_2, L_{\vec{k}}(E))$ is
exactly $\mathrm{Cor}_{\Gamma_2/\Gamma_1}[\lambda_c]$. We have
{\allowdisplaybreaks
\begin{eqnarray*}
\check{\rho}_{\vec{k}}(x_j)\lambda_{z_0,c}(h_{\gamma,j}) &=&
\lambda_{x_j\cdot z_0, x_j\star c} (x_jh_{\gamma, j}x_j^{-1}) \hskip
20pt \text{by Lemma } \ref{lem:lambda}\ (\text{\ref{it:lambda-a}})
\\
&=& \lambda_{x_j\cdot z_0, x_j\star c} (\gamma x_ix_j^{-1})
\\
&=&\lambda_{z_0, x_j\star c}(\gamma x_i)- \lambda_{z_0, x_j\star
c}(x_j) \\
&=& \lambda_{z_0, x_j\star c}(\gamma x_ix_j^{-1}) +
\check{\rho}_{\vec{k}}(\gamma
x_ix_j^{-1}) \lambda_{z_0, x_j\star c}(x_j ) \\
&&-\lambda_{z_0, x_j\star c}(x_j) \hskip
50pt \text{by Lemma } \ref{lem:lambda} \ (\text{\ref{it:lambda-b}})\\
&=&\lambda_{z_0, x_j\star c}(\gamma) - \check{\rho}_{\vec{k}}(\gamma
x_ix_j^{-1}) \lambda_{z_0, x_j\star c}(x_jx^{-1}_i) \\ &&
+\check{\rho}_{\vec{k}}(\gamma x_ix_j^{-1})  \lambda_{z_0, x_j\star
c}(x_j )-\lambda_{z_0, x_j\star c}(x_j) \hskip
20pt \text{by Lemma } \ref{lem:lambda} \ (\text{\ref{it:lambda-b}}) \\
&=&\lambda_{z_0, x_j\star c}(\gamma) - \check{\rho}_{\vec{k}}(\gamma
) \lambda_{x_ix_j^{-1}z_0, x_i\star c}(x_jx^{-1}_i) \\ &&
+\check{\rho}_{\vec{k}}(\gamma) \lambda_{x_ix_j^{-1} z_0, x_i\star
c}(x_ix_jx_i^{-1} )-\lambda_{z_0, x_j\star c}(x_j) \hskip
20pt \text{by Lemma } \ref{lem:lambda}\ (\text{\ref{it:lambda-a}}) \\
&=&\lambda_{z_0,x_j\star c}(\gamma) -\check{\rho}_{\vec{k}}(\gamma)
(\lambda_{z_0, x_i\star
c}(1)- \lambda_{z_0, x_i\star c}(x_ix_j^{-1})) \\
&& +\check{\rho}_{\vec{k}}(\gamma)(\lambda_{z_0, x_i\star
c}(x_i)-\lambda_{z_0,x_i\star
c}(x_ix_j^{-1}))-\lambda_{z_0, x_j\star c}(x_j) \\
&=&\lambda_{z_0,x_j\star
c}(\gamma)+\check{\rho}_{\vec{k}}(\gamma)\lambda_{z_0, x_i\star
c}(x_i) -\lambda_{z_0, x_j\star c}(x_j).
\end{eqnarray*}
} Here, we have used $\lambda_{z_0, x_i\star c}(1)=0$.
 Taking sum we obtain
$$ \sum_j \check{\rho}_{\vec{k}}(x_j)\lambda_{z_0,c}(h_{\gamma,j}) = \lambda_{z_0,\sum\limits_j x_j\star c}(\gamma) +(\gamma-1)\sum_j\lambda_{z_0, x_j\star c}(x_j).  $$
This shows (\ref{it:col-b}). \end{proof}

We define the Coleman morphism
\begin{eqnarray*}\kappa^{\mathrm{col},\iota_{\mathfrak{p}}}_u:
C^1_{\mathfrak{p}}(U^{\mathfrak{p}}, L_{\vec{k}}(E)) & \rightarrow &
MH^1_{\mathfrak{p}}(U^\mathfrak{p}, L_{\vec{k}}(E))
\\
c &\mapsto &
\kappa_{c;u}^{\mathrm{col},\iota_{\mathfrak{p}}}(g)=\lambda_{\Gamma_g;u}(c(g,\cdot)).\end{eqnarray*}

\begin{prop}\label{prop:coleman}
\begin{enumerate}
\item\label{it:coleman-coh-a} If $c\in C^1_{\mathfrak{p}}(U^{\mathfrak{p}}, L_{\vec{k}}(E))$,
then $\kappa^{\mathrm{col},\iota_\mathfrak{p}}_{c;u}$ is in
$MH^1_{\mathfrak{p}}(U^\mathfrak{p}, L_{\vec{k}}(E))$.
\item\label{it:coleman-coh-b}
The map \begin{eqnarray*}\kappa^{\mathrm{col},\iota_\mathfrak{p}}_u:
C^1_{\mathfrak{p}}(U^{\mathfrak{p}}, L_{\vec{k}}(E))
&\rightarrow&MH^1_{\mathfrak{p}}(U^\mathfrak{p}, L_{\vec{k}}(E)) \\
c & \mapsto & \kappa^{\mathrm{col},\iota_\mathfrak{p}}_{c;u}
\end{eqnarray*} is $\mathbb{T}^p_U$-equivariant.
\end{enumerate}
\end{prop}
\begin{proof} For (a) we need to show that $\kappa^{\mathrm{col},\iota_{\mathfrak{p}}}_{c;u}$ satisfies (a)-(d) in
Definition \ref{defn:coh-mod-form}. Conditions (a) and (b) follows
directly form the definition; (c) and (d) follows form Lemma
\ref{lem:lambda} (\ref{it:lambda-a}).

The proof of (\ref{it:coleman-coh-b}) is similar to the assertion
that $\kappa^{\mathrm{sch}}$ is $\mathbb{T}_U^p$-equivariant in  the
proof of  Proposition \ref{prop:schneider} . We only need to use
Lemma \ref{lem:coleman} instead of Lemma \ref{lem:HA}
(\ref{it:HA-a}, \ref{it:HA-b}) used there. We omit the details.
\end{proof}

To end this section we define $L$-invariants of Teitelbaum type. See
\cite[Definition 3.4]{CMP}.

Now, let $f$ be a newform in $M_{\vec{k}}(U, E)$. Assume that
$\alpha_{\mathfrak{p}}(f)=1$. Then $f$ comes from a harmonic modular
form $c\in C^1_{\mathfrak{p}}(U^{\mathfrak{p}}, L_{\vec{k}}(E))$. We
attached to $f$ two cohomological valued modular forms
$\kappa^{\mathrm{sch}}_c$ and
$\kappa^{\mathrm{col},\iota_\mathfrak{p}}_{c;u}$. By Proposition
\ref{prop:schneider} and Proposition \ref{prop:coleman} both
$\kappa^{\mathrm{sch}}_c$ and
$\kappa^{\mathrm{col},\iota_\mathfrak{p}}_{c;u}$ lie in the
$f$-component of $MH^1(U^\mathfrak{p},L_{\vec{k}}(E))$ with respect
to $\mathbb{T}^p_U$-action. As $f$ is new, the multiplicity one
theorem for $\mathrm{GL}_2$ tells us the $f$-component in
$M_{\vec{k}}(U, E)$ and that in $\in
C^1_{\mathfrak{p}}(U^{\mathfrak{p}}, L_{\vec{k}}(E))$ is
$1$-dimensional over $E$. Then Proposition \ref{prop:schneider}
implies that $f$-component $MH^1(U^\mathfrak{p},L_{\vec{k}}(E))$ is
$1$-dimensional over $E$ and generated by $\kappa_c^{\mathrm{sch}}$.
Therefore there exists a unique
$\mathcal{L}^\Tei_{\iota_{\mathfrak{p}};u}(f) \in E$ such that
$$\kappa^{\mathrm{col}, \iota_{\mathfrak{p}}}_{c;u}=
\mathcal{L}^\Tei_{\iota_{\mathfrak{p}};u}(f)
\kappa_c^{\mathrm{sch}}.$$

Now, let $u$ vary.

Let $\ord_\mathfrak{p}$ be the function on $\BC_p^\times$ such that
$\ord_\mathfrak{p}(\omega_\mathfrak{p})=1$, that
$\ord_\mathfrak{p}(xy)=\ord_\mathfrak{p}(x)+\ord_\mathfrak{p}(y)$
and that $\ord_p(x)=0$ if $x$ is a unit. Here $\omega_\mathfrak{p}$
is a uniformizing element of $F_\mathfrak{p}$. If $u_1,
u_2\in\BC_p^\times$ are not unit, then
$$ \log_{u_1}-\log_{u_2}=\left(\log_{u_1}(\omega_\mathfrak{p})-\log_{u_2}(\omega_\mathfrak{p})\right)\ord_\mathfrak{p}. $$

\begin{prop}\label{prop:diff-L} We have
$$ \mathcal{L}^\Tei_{ \iota_\mathfrak{p};u_1}(f) - \mathcal{L}^\Tei_{ \iota_\mathfrak{p};u_2}(f) = \log_{u_1}(\omega_\mathfrak{p})-\log_{u_2}(\omega_\mathfrak{p}).$$
\end{prop}
\begin{proof} It is easy to check that Lemma
\ref{lem:lambda} again holds when the function $\log_u$ in
(\ref{eq:int-lambda}) is replaced by $\ord_\mathfrak{p}$. In
particular, it defines an element of
$MH^1_{\mathfrak{p}}(U^\mathfrak{p}, L_{\vec{k}}(E))$. As is showed
in \cite[\S 3.1]{DT2008} it is just $\kappa_c^\mathrm{sch}$,
yielding our assertion.
\end{proof}

\begin{defn}\label{def:Tei} We call $\mathcal{L}^{\Tei}_{\iota_{\mathfrak{p}};u}(f)$ the {\it
L-invariant of Teitelbaum type} of $f$ at $\mathfrak{p}$ for the
embedding $\iota_\mathfrak{p}$ and the logarithm $\log_u$.
\end{defn}

\section{Measures and integrals}\label{sec:meas-dist}

Fix a set $J$ of primes of $F$ above $p$. Let $J_0$ be the set of
primes $\mathfrak{p}\in J$ such that $\alpha_{\mathfrak{p}}=1$. We
assume that each $\mathfrak{p}\in J_0$ is split in $K$ and that
$k_\mathfrak{p}=2$ for each $\mathfrak{p}\in J_0$.

\subsection{The measure $\widetilde{\mu}_J$}

In this subsection we write $J'=J\backslash J_0$.

If $\vec{n}$ is a $J'$-tuple of nonnegative integers, we put
$$\widehat{\mathcal{O}}_{J_0;\vec{n}, \mathfrak{c}}
:=\bigcap_{\vec{m}}\widehat{\mathcal{O}}_{(\vec{m},\vec{n}),
\mathfrak{c}}$$ and $$ \widetilde{X}_{J_0;\vec{n},\mathfrak{c}}:=\BA
_F^{\infty,\times} \backslash \BA _K^{\infty,\times} /
\widehat{\mathcal{O}}_{J_0;\vec{n}, \mathfrak{c}}^\times
$$ where $\vec{m}$ runs through all $J_0$-tuples of
nonnegative integers. Similarly, we put
$$\widehat{\mathcal{O}}_{J;\mathfrak{c}}
:=\bigcap_{\vec{m}}\widehat{\mathcal{O}}_{\vec{m}, \mathfrak{c}}$$
and
$$\widetilde{X}_{J;\mathfrak{c}}:=\BA _F^{\infty,\times} \backslash
\BA _K^{\infty,\times} / \widehat{\mathcal{O}}_{J;
\mathfrak{c}}^\times,
$$
where $\vec{m}$ runs through all $J$-tuples of nonnegative integers.

Let $F^\times \backslash K^\times$ act on $
\widetilde{X}_{J_0;\vec{n},\mathfrak{c}}$ and
$\widetilde{X}_{J;\mathfrak{c}}$ by multiplication. The quotient
$(F^\times \backslash K^\times)\backslash
\widetilde{X}_{J;\mathfrak{c}}$ is isomorphic to the group
$\mathcal{G}_{J;\mathfrak{c}}$ defined in Section
\ref{ss:anti-ext-char}.

We write $\widetilde{X}_{J_0;\vec{n}, \mathfrak{c}}$ in the form
$$ \widetilde{X}_{J_0;\vec{n}, \mathfrak{c}} = (F_{J_0}^\times\backslash K_{J_0}^\times) \times (\BA_F^{\infty,J_0,\times}\backslash \BA_K^{\infty,J_0,\times} /
\widehat{\mathcal{O}}_{J_0;\vec{n}, \mathfrak{c}}^\times).$$

For each $\mathfrak{p}\in J_0$ we consider the twisted action of
$i_\mathfrak{p}(K_\mathfrak{p}^\times)$ on
$\mathbf{P}^1(F_\mathfrak{p})=F_\mathfrak{p}\cup\{\infty\}$ given in
Section \ref{ss:harmonic}. For this action
$i_\mathfrak{p}(K_\mathfrak{p}^\times)$ has two fixed points $0,
\infty$. So, for any point $\star\in
\mathbf{P}^1(F_{\mathfrak{p}})\backslash \{0, \infty\}$, the orbit
of $\star$ by $i_\mathfrak{p}(K_{\mathfrak{p}}^\times)$ is exactly
$\mathbf{P}^1(F_{\mathfrak{p}})\backslash \{0, \infty \}$. Taking
$\star=-\beta$ we obtain an isomorphism
$$\eta_\mathfrak{p}: F_\mathfrak{p}^\times \backslash K_\mathfrak{p}^\times   \xrightarrow{\sim}\mathbf{P}^1(F_\mathfrak{p})\backslash \{0, \infty\},
\hskip 10 pt (x,y)\mapsto  \frac{-\beta x }{y}.
$$ Note that
$\eta_\mathfrak{p}(F_\mathfrak{p}^\times\backslash
K_\mathfrak{p}^\times )= F_\mathfrak{p}^\times$ and
$\eta_\mathfrak{p}(\mathcal{O}_{F_\mathfrak{p}}^\times\backslash
\mathcal{O}_{K_\mathfrak{p}}^\times
)=\mathcal{O}_{F_\mathfrak{p}}^\times$. In this way, we identify
$F_{J_0}^\times\backslash
K_{J_0}^\times=\prod\limits_{\mathfrak{p}\in J_0}
F^\times_\mathfrak{p}\backslash K^\times_\mathfrak{p} $ with the
subset $F_{J_0}^\times=\prod\limits_{\mathfrak{p}\in
J_0}F^\times_\mathfrak{p}$ of $\mathbf{P}_{J_0}$.

Thus $\widetilde{X}_{J_0;\vec{n}, \mathfrak{c}}$ is isomorphic to
$$F^\times_{J_0} \times (\BA_F^{\infty,J_0,\times}\backslash \BA_K^{\infty,J_0,\times} /
\widehat{\mathcal{O}}_{J_0;\vec{n}, \mathfrak{c}}^\times).$$ Let
$\overline{X}_{J_0;\vec{n},\mathfrak{c}}$ be the larger set
$$ \overline{X}_{J_0;\vec{n},\mathfrak{c}} := F_{J_0} \times (\BA_F^{\infty,J_0,\times}\backslash \BA_K^{\infty,J_0,\times} / \widehat{\mathcal{O}}_{J_0;\vec{n},
\mathfrak{c}}^\times)\subset \mathbf{P}_{J_0} \times
(\BA_F^{\infty,J_0,\times}\backslash \BA_K^{\infty,J_0,\times} /
\widehat{\mathcal{O}}_{J_0;\vec{n}, \mathfrak{c}}^\times). $$

We pull back the measure $\mu_J$ on $\mathcal{G}_{J;\mathfrak{c}}$
(see Definition \ref{def:measure}) to a measure $\widetilde{\mu}_J$
on $\widetilde{X}_{J; \mathfrak{c}}$, so that for any compactly
supported $p$-adically continuous function $\alpha$ on
$\widetilde{X}_{J;\mathfrak{c}}$, we can define $\int \alpha
\widetilde{\mu}_J$. The function
$$\underline{\alpha} (x)=\sum_{a\in K^\times/F^\times} \alpha(ax) $$
is invariant by $F^\times \backslash K^\times$ and thus can be
viewed as a function on $\mathcal{G}_{J;\mathfrak{c}}$; we have $
\int \alpha \widetilde{\mu}_J =\int \underline{\alpha} \mu_J$. Here,
for $\mu$ being a measure on a $p$-adically topological space $X$,
we mean that  for each compact open subset $U$ of $X$, there exists
a positive constant $C_U$ such that, if $g$ is a $p$-adically
continuous function on $X$ supported on $U$, then $$ | \int g \mu |
\leq C_U ||g||_{\mathrm{Gauss}} ,
$$ where $||g||_{\mathrm{Gauss}} :=\sup\limits_{x\in
U}|g(x)|$.

For any $a\in \BA_K^\times$ and $J_0$-tuple of positive integers
$\vec{m}$ we have
\begin{equation} \label{eq:int-U}
\int 1_{ a \widehat{\mathcal{O}}_{(\vec{m},\vec{n}),
\mathfrak{c}}^\times } \widetilde{\mu}_J =
\frac{1}{\jmath(\prod_{\mathfrak{p}\in
J'}\alpha_\mathfrak{p}^{n_\mathfrak{p}})} \widehat{\nu}(a) \langle
\rho_{\vec{k}}
(\mathfrak{i}_p(\varsigma_J^{(\vec{m},\vec{n})}))^{-1}
\widehat{\mathbf{v}}_{\mathbf{m}}  ,
\widehat{\tilde{\phi}^{\infty,\dagger_J}} (a\varsigma_J^{(\vec{n})})
\rangle .
\end{equation}

We use $C_{J}^\flat$ to denote the space of functions of the form
$g+h$, where $g$ is a compactly supported $p$-adically continuous
function on $\widetilde{X}_{J;\mathfrak{c}}$, and $h$ is a compactly
supported $p$-adically continuous function on $
\overline{X}_{J_0;\vec{n},\mathfrak{c}}$ for some $\vec{n}$
($\vec{n}$ allowed to vary).

Below, we will extend the integral $\int g \widetilde{\mu}_J$ to
$C_{J}^\flat$.

Let $\phi$ and $\phi^{\dagger_J}$ be as in Section \ref{ss:phi},
 $c$ the $J_0$-typle ``harmonic cocycle'' valued modular form
(Definition \ref{defn:J-mod-form}) attached to
$\widehat{\tilde{\phi}^{\infty,\dagger_J}}$.

Let $\mu_{c(a\varsigma_{J',J_p\backslash J_0}^{(\vec{n})},\cdot)}$
be the (vector valued) measure on $$ F_{J_0} \times
(\BA_F^{\infty,J_0,\times} a \widehat{\mathcal{O}}_{J_0;\vec{n},
\mathfrak{c}}^\times) \subset \mathbf{P}_{J_0}
  \times (\BA_F^{\infty,J_0,\times} a
\widehat{\mathcal{O}}_{J_0;\vec{n}, \mathfrak{c}}^\times) \cong
\mathbf{P}_{J_0}  . $$ See Remark \ref{rem:measure} and Corollary
\ref{cor:measure}.

\begin{lem} \label{lem:int-nu-c} For every $a\in \BA_K^{\infty,J_0,\times}$
the integral $$ g\mapsto \frac{1}{\jmath(\prod_{\mathfrak{p}\in
J'}\alpha_\mathfrak{p}^{n_\mathfrak{p}})}\widehat{\nu}(a)
\langle\mathfrak{i}_p(\varsigma_{J'}^{(\vec{n})})^{-1}
\widehat{\mathbf{v}}_{\mathbf{m}}, \int g \: \mu_{c(
a\varsigma_{J',J_p\backslash J_0}^{(\vec{n})},\cdot)} \rangle$$ is
independent of representatives $a$ in the double coset
$\BA_F^{\infty,J_0,\times} a \widehat{\mathcal{O}}_{J_0;\vec{n},
\mathfrak{c}}^\times$.
\end{lem}
\begin{proof} For any $x\in \BA_F^{\infty,J_0,\times}$ and $y\in \widehat{\mathcal{O}}_{J_0;\vec{n},
\mathfrak{c}}^\times$, since $(\varsigma_{J',J_p\backslash
J_0}^{(\vec{n})})^{-1}y\varsigma_{J',J_p\backslash
J_0}^{(\vec{n})}\in U^{J_0}$, by Definition \ref{defn:J-mod-form}
(\ref{it:J-center}, \ref{it-J-mod-c}) we have
$$\mu_{c( xay\varsigma_{J',J_p\backslash J_0}^{(\vec{n})},\cdot)}
 =
\check{\rho}_{\vec{k}}\bigg( \Big((\varsigma_{J',J_p\backslash
J_0}^{(\vec{n})})^{-1} y \varsigma_{J',J_p\backslash
J_0}^{(\vec{n})} \Big)^{-1}_p \bigg)
\mu_{c(a\varsigma_{J',J_p\backslash J_0}^{(\vec{n})},\cdot) } .
$$ Hence
\begin{eqnarray*} && \widehat{\nu}(xay)
\langle\mathfrak{i}_p(\varsigma_{J',J_p\backslash
J_0}^{(\vec{n})})^{-1} \widehat{\mathbf{v}}_{\mathbf{m}}, \int g \:
\mu_{c( xay\varsigma_{J',J_p\backslash J_0}^{(\vec{n})},\cdot)}
\rangle  \\ & = & \widehat{\nu}(xay)
\langle\mathfrak{i}_p((\varsigma_{J',J_p\backslash
J_0}^{(\vec{n})})^{-1}y \varsigma_{\emptyset, J_0}^{-1})
\widehat{\mathbf{v}}_{\mathbf{m}}, \int g \: \mu_{c(
a\varsigma_{J',J_p\backslash J_0}^{(\vec{n})},\cdot)} \rangle \\
& = & \widehat{\nu}(xay) \widehat{\nu}(y)^{-1}
\langle\mathfrak{i}_p(\varsigma_{J',J_p\backslash
J_0}^{(\vec{n})})^{-1} \widehat{\mathbf{v}}_{\mathbf{m}}, \int g \:
\mu_{c( a\varsigma_{J',J_p\backslash J_0}^{(\vec{n})},\cdot)}
\rangle \\ & = & \widehat{\nu}(a)
\langle\mathfrak{i}_p(\varsigma_{J'}^{(\vec{n})})^{-1}
\widehat{\mathbf{v}}_{\mathbf{m}}, \int g \: \mu_{c(
a\varsigma_{J',J_p\backslash J_0}^{(\vec{n})},\cdot)}\rangle.
\end{eqnarray*} Here, we use $\mathfrak{i}_p(\varsigma_{\emptyset, J_0}^{-1})\widehat{\mathbf{v}}_{\mathbf{m}}=\widehat{\mathbf{v}}_{\mathbf{m}}$,
$\mathfrak{i}_p(y)\widehat{\mathbf{v}}_{\mathbf{m}}
=\widehat{\nu}(y)^{-1}\widehat{\mathbf{v}}_{\mathbf{m}}$ and
$\widehat{\nu}(x)=1$.
\end{proof}

Let $\mu_{\nu,c}$ be the measure on $ \overline{X}_{J_0;\vec{n},
\mathfrak{c}}$ whose restriction to $ F_{J_0}  \times
(\BA_F^{\infty,J_0,\times} a \widehat{\mathcal{O}}_{J_0;\vec{n},
\mathfrak{c}}^\times) $ is that given by Lemma \ref{lem:int-nu-c}.

\begin{prop}\label{prop:rest} For any compactly supported continuous function $g$ on $
\widetilde{X}_{J_0;\vec{n}, \mathfrak{c}}$ we have
$$ \int g \mu_{\nu,c} = \int g \widetilde{\mu}_J.$$
\end{prop}
\begin{proof} Since both $\mu_{\nu,c}$ and $\widetilde{\mu}_J$
are measures on $ \widetilde{X}_{J_0;\vec{n}, \mathfrak{c}}$, we
only need to verify the formula when $g$ is locally constant.
Without loss of generality we may assume that $g=1_{U\times
(\BA_F^{J_0,\times} b \widehat{\mathcal{O}}_{J_0;\vec{n},
\mathfrak{c}}^\times)}$ where $U$ is a compact open subset of
$F_{J_0}^\times \cong F_{J_0}^\times\backslash K_{J_0}^\times$.
Since $U$ is a disjoint union of open subsets of the form $a
\mathcal{U}(1, \vec{m})$, we may further assume that
$U=-a\beta\mathcal{U}(1, \vec{m} )$ for some $a\in F_{J_0}^\times$.
Note that
\begin{equation}\label{eq:eta-U}
U=\eta_{J_0}\left((a,1)\prod_{\mathfrak{p}\in
J_0}(1+\mathfrak{p}^{m_{\mathfrak{p}}}\mathcal{O}_{K_\mathfrak{p}})
\cdot F_{J_0}^\times/F_{J_0}^\times\right). \end{equation}

By a simple computation we have $$ \zeta_{J_0,J_0}^{(\vec{
m})}\cdot\mathcal{U}_{e_{J_0,*}}=-\beta\mathcal{U}(1, \vec{m}).$$
Thus
$$ i_{J_0}(a,1)\zeta_{J_0,J_0}^{(\vec{
m})}\cdot  \mathcal{U}_{e_{J_0,
*}} = a\beta
\mathcal{U}(1, \vec{m})=U .$$

Now we have
\begin{eqnarray*}\int g \mu_{\nu,c}  & = & \frac{1}{\jmath(\prod_{\mathfrak{p}\in
J'}\alpha_\mathfrak{p}^{n_\mathfrak{p}})} \widehat{\nu}(b)
\langle\mathfrak{i}_p(\varsigma_{J'}^{(\vec{n})})^{-1}
\widehat{\mathbf{v}}_{\mathbf{m}}, c(b \varsigma_{J',J_p\backslash
J_0}^{(\vec{n})},
i_{J_0}(a,1)\varsigma_{J_0,J_0}^{(\vec{m})}e_{J_0,*}) \rangle \\
&=& \frac{1}{\jmath(\prod_{\mathfrak{p}\in
J'}\alpha_\mathfrak{p}^{n_\mathfrak{p}}\cdot\prod_{\mathfrak{p}\in
J_0}\alpha_\mathfrak{p}^{m_\mathfrak{p}})} \widehat{\nu}(b)
\langle\mathfrak{i}_p(\varsigma_{J}^{(\vec{m},\vec{n})})^{-1}
\widehat{\mathbf{v}}_{\mathbf{m}},
\widehat{\tilde{\phi}^{\infty,\dagger_J}}(i_{J_0}(a,1)b\varsigma_{J}^{(\vec{m},\vec{n})})
\rangle   \end{eqnarray*} where the last equality follows from
Remark \ref{rem:measure}. Here, we use the fact that
$\check{\rho}_{\vec{k}}|_{G_{J_0}}$ is trivial. By (\ref{eq:int-U})
the last term of the above formula is exactly $\int g
\widetilde{\mu}_J$.
\end{proof}

Now, we extend $\widetilde{\mu}_J$ to $C^\flat_{J}$ by putting
$$ \int (g+h) \widetilde{\mu}_J=\int g \widetilde{\mu}_J+\int h \mu_{\nu,c}.$$
This definition is compatible when $\vec{n}$ varies. Indeed, this
follows from an argument similar to the proof of Proposition
\ref{prop:comp}.

The action of $F^\times\backslash K^\times$ on
$\overline{X}_{J_0;\vec{n},\mathfrak{c}}$ and
$\widetilde{X}_{J;\mathfrak{c}}$ induces an action on $C^\flat_{J}$:
$$\gamma^*(f)(x):=f(\gamma x )$$ for $\gamma\in F^\times\backslash
K^\times$ and $f\in C^\flat_{J}$. It induces an action of
$F^\times\backslash K^\times$ on the dual of $C^\flat_{J}$.

\begin{lem} $\widetilde{\mu}_J$ is $F^\times\backslash K^\times$-invariant.
\end{lem}
\begin{proof} This follows from Definition \ref{defn:J-mod-form}
(\ref{it:J-mod-d}) and the fact that
$\check{\rho}_{\vec{k}}|_{G_{J_0}}$ is trivial.
\end{proof}

Let $J_1$ be a subset of $J_0$. With $ J\backslash J_1$ instead of
$J$, we have a distribution $\widetilde{\mu}_{J\backslash J_1}$ on
$\widetilde{X}_{J\backslash J_1;\mathfrak{c}}$ and
$\overline{X}_{J\backslash J_1;\vec{n},\mathfrak{c}}$. We compare
$\widetilde{\mu}_J$ and $\widetilde{\mu}_{J\backslash J_1}$.

\begin{prop}\label{prop:transfer} If $V$ is a compact open subset of
$\BA_F^{\infty,J_1,\times}\backslash
\BA_K^{\infty,J_1,\times}/\widehat{\mathcal{O}}^\times_{J;\mathfrak{c}}$,
then
$$ \widetilde{\mu}_J(\prod_{v\in
J_1}\mathcal{O}_{F_v}  \times  V ) = \widetilde{\mu}_{J\backslash
J_1} (
 [1_{J_1}]   \times V)
, $$ where $[1_{J_1}]$ denotes the point $F_{J_1}^\times \cdot 1
\cdot \mathcal{O}_{K_{J_1}}^\times$ in $F_{J_1}^\times \backslash
K_{J_1}^\times/\mathcal{O}_{K_{J_1}}^\times$.
\end{prop}
\begin{proof} By the same argument as in the proof of Proposition
\ref{prop:rest}, we may take $V$ to be of the form
$$i_{J_0\backslash J_1}(a,1)\zeta_{J_0\backslash J_1,J_0\backslash
J_1}^{(\vec{ m})}\cdot \mathcal{U}_{e_{J_0\backslash J_1,
*}}\times \BA_F^{J_0,\times} b
\widehat{\mathcal{O}}_{J_0;\vec{n}, \mathfrak{c}}^\times$$ where
$\vec{m}$ and $\vec{n}$ are $J_0\backslash J_1$-tuple and $J'$-tuple
of positive integers respectively. Note that $$\prod_{v\in
J_1}\mathcal{O}_{F_v}=\mathcal{U}_{e_{J_1,0}}=\zeta_{\emptyset,J_1}\cdot\mathcal{U}_{e_{J_1,*}}.$$

Thus we have {\allowdisplaybreaks
\begin{eqnarray*}&&\widetilde{\mu}_J( \prod_{v\in
J_1}\mathcal{O}_{F_v} \times V)\\ &=& \mu_{\nu, c}( i_{J_0\backslash
J_1}(a,1)\zeta_{J_0\backslash J_1,J_0}^{(\vec{ m})}\cdot
U_{e_{J_0},*} \times \BA_F^{J_0,\times} b
\widehat{\mathcal{O}}_{J_0;\vec{n},
\mathfrak{c}}^\times) \\
 &=&  \frac{1}{\jmath(\prod_{\mathfrak{p}\in
J'}\alpha_\mathfrak{p}^{n_\mathfrak{p}})} \widehat{\nu}(b)
\langle\mathfrak{i}_p(\varsigma_{J'}^{(\vec{n})})^{-1}
\widehat{\mathbf{v}}_{\mathbf{m}},
\widehat{\tilde{\phi}^{\infty,\dagger_{J}}}(i_{J_0\backslash
J_1}(a,1)\zeta_{J_0\backslash J_1,J_0}^{(\vec{ m})}
b\varsigma_{J',J_p\backslash J_0}^{( \vec{n})}) \rangle \\  &=&
\frac{1}{\jmath(\prod_{\mathfrak{p}\in
J'}\alpha_\mathfrak{p}^{n_\mathfrak{p}}\cdot \prod_{\mathfrak{p}\in
J_0\backslash J_1}\alpha_\mathfrak{p}^{m_\mathfrak{p}})}
\widehat{\nu}(b) \langle\mathfrak{i}_p(\varsigma_{J\backslash
J_1}^{(\vec{n})})^{-1} \widehat{\mathbf{v}}_{\mathbf{m}},
\widehat{\tilde{\phi}^{\infty,\dagger_{J\backslash J_1}}}(
i_{J_0\backslash
J_1}(a,1) b\varsigma_{J\backslash J_1}^{( \vec{m},\vec{n})}) \rangle \\
&=& \widetilde{\mu}_{J\backslash J_1} ( [1_{J_1}] \times V) .
\end{eqnarray*}}

\noindent Here, the first equality is just Proposition
\ref{prop:rest}, the second equality follows from the definition of
$\mu_{\nu,c}$, the third equality follows from the fact that
$\widehat{\tilde{\phi}^{\infty,\dagger_{J}}}=\widehat{\tilde{\phi}^{\infty,\dagger_{J\backslash
J_1}}}$ and that the action of
$\mathfrak{i}_p(\varsigma_{J_0\backslash J_1,J_0}^{(\vec{m})})$ on
$L_{\vec{k}}(\BC_p)$ is trivial, and the last equality follows from
the definition of $\widetilde{\mu}_{J\backslash J_1}$.
\end{proof}

\subsection{Some computation (I): vanishing of some integrals}
\label{ss:comp-I}

For each $\mathfrak{p}\in J_0$
($\mathfrak{p}\mathcal{O}_K=\mathfrak{P}\overline{\mathfrak{P}}$) we
choose $\beta_\mathfrak{p}\in K^\times$ such that
$(\beta_\mathfrak{p})=\mathfrak{P}^{-h_\mathfrak{p}}$ with
$h_\mathfrak{p}>0$. \label{sec:gamma} We may even choose
$\beta_\mathfrak{p}$ such that $h_\mathfrak{p}$ is the smallest with
this property. For any subset $J_1$ of $J_0$, let $\Delta_{J_1}$ be
the subgroup of $K^\times/F^\times$ generated by
$\beta_\mathfrak{p}$ ($\mathfrak{p}\in J_1$).

The $\mathfrak{p}$-component of $\beta_\mathfrak{p}$, denoted by
$\beta_{\mathfrak{p},\mathfrak{p}}$, acts on
\begin{eqnarray*}
K_\mathfrak{p}^\times/F_\mathfrak{p}^\times &\cong&
F_\mathfrak{p}^\times,
\\ (x,y)F_\mathfrak{p}^\times &\mapsto& \frac{-\beta x}{y} ;
\end{eqnarray*} the
fundamental domain of this action is
$$ D_\mathfrak{p}:= \{x\in F^\times_\mathfrak{p}: 0\leq v_\mathfrak{p}(x) < h_\mathfrak{p}\}.
$$ We put $\underline{D}_\mathfrak{p}:= \mathcal{O}^\times_{F_\mathfrak{p}}$ and $$ B_\mathfrak{p}:=\{x\in
F^\times_\mathfrak{p}: v_\mathfrak{p}(x)\geq 0 \} .$$ Then
$$ (1-\beta_{\mathfrak{p},\mathfrak{p}}^*)
1_{B_\mathfrak{p}}=1_{D_{\mathfrak{p}}} . $$ For each finite place
$v\notin J_0$ we set $$ D_v=\mathcal{O}_{K_v}^\times/
\mathcal{O}_{F_{v}}^\times .
$$

For any element $a\in \BA_K^{\infty,\times}$ with $a_p=1$, and two
disjoint subsets $J_1$ and $J_2 $ of $J_0$, we put
$$ Z_{J_1,J_2,a} =  ( \prod_{\mathfrak{p}\in J_1} B_{\mathfrak{p}})\times (\prod_{\mathfrak{p}\in J_2}\underline{D}_\mathfrak{p}) \times a(\prod_{v\notin J_1\cup J_2} D_v)
.$$

We form the indexed set
$$ I = \LOG_J \times \{\text{finite places of } F\}. $$  If $i=(l, v)$,
we put $\mathrm{pr}_1(i)=l$ and $\mathrm{pr}_2(i)=v$.

For each pair $i=(l, v)$ let  $_vl$ denote the function on
$\BA_K^{\infty,\times}/\BA_F^{\infty,\times}$ such that
$$ _v l(x)=l(x_v) ,$$ where  $x_v$ denotes the element of
$\BA_K^{\infty,\times}/\BA_F^{\infty,\times}$ whose $v$-component is
$x_v$ and other components are all $1$. Then $l =\sum_v \: _vl$. We
also write $\ell_i$ for $_vl$.

\begin{lem} \label{lem:log-value} If $\mathrm{pr}_2(i)\notin \Sigma_J$, then
$\ell_i(\beta_\mathfrak{p})=0$ for each $\mathfrak{p}\in J_0$.
\end{lem}
\begin{proof} Write $i=(l,v)$. As $v\neq \mathfrak{p}$,  $\beta_{\mathfrak{p},
v}\in \mathcal{O}_{K_v}^\times$. For each $\sigma\in \Sigma_J$  we
have
$_v\log_{\sigma}|_{\mathcal{O}_{K_v}^\times/\mathcal{O}_{F_v}^\times}=0$.
In particular $_v\log_{\sigma}(\beta_\mathfrak{p})=0$. Thus
${_v}l(\beta_\mathfrak{p}) =0$.
\end{proof}

For a vector $\vec{t}=(t_i)_{i\in I}$ of nonnegative integers (with
$t_i=0$ for almost all $i\in I$) and two disjoint subsets
$J_1\subset J_0$ and $J_2\subset J$, we write
$$\ell^{\vec{t}}:= \prod_{i\in I}\ell_i^{t_i}, \hskip 10pt \lambda(\vec{t},J_1,J_2,a):=
1_{Z_{J_1,J_2,a}}\cdot \ell^{\vec{t}}.$$ If $(\vec{t}, J_1,J_2)$
satisfies the condition

($\flat$) $t_i=0$ when $\mathrm{pr}_2(i)\in \Sigma_{J_1}$,

\noindent  then $\lambda(\vec{t}, J_1, J_2,a)$ is in $C^\flat_J$.
When $J_1=\emptyset$, $(\flat)$ automatically holds.

Put $$ |\vec{t}|=\sum\limits_{ i \in I}t_i   \hskip 10pt  \text{ and
} \hskip 10pt |\vec{t}|_{J}=\sum\limits_{\mathrm{pr}_2(i)\in
\Sigma_{J}}t_i .
$$ If $t'_i\leq t_i$ for each $i\in I$, and if
$|\vec{t}'|_{J}<|\vec{t}|_{J}$, we write $\vec{t}'<_{J}\vec{t}$.

\begin{lem}\label{lem:1-gamma} Let $\mathfrak{p}$ be in $J_0$. For any $\vec{t}$ with $t_i=0$ for almost all $i\in I$,
$(1-\beta_\mathfrak{p}^*)\ell^{\vec{t}}$ is a linear combination of
$\ell^{\vec{t}'}$ with $\vec{t}'<_{J}\vec{t}$. Writing
$$(1-\beta_\mathfrak{p}^*)\ell^{\vec{t}}=\sum_{\vec{t'}<_J\vec{t}}c_{\vec{t'}}\ell^{\vec{t'}};$$
if $|\vec{t'}|=|\vec{t}|-1$ with $t'_i=t_i-1$, then
$c_{\vec{t'}}=-t_i\ell_i(\beta_\mathfrak{p})$.
\end{lem}
\begin{proof} We use the relation
\begin{equation}\label{eq:1-gamma}
(1-\beta_\mathfrak{p}^*)(fg)=(1-\beta_\mathfrak{p}^*)f \cdot g
+f\cdot (1-\beta_\mathfrak{p}^*)g - (1-\beta_\mathfrak{p}^*)f \cdot
(1-\beta_\mathfrak{p}^*)g\end{equation} and the fact that
$(1-\beta_\mathfrak{p}^*)\ell_i=-\ell_i(\beta_\mathfrak{p})$ is a
constant.
\end{proof}

\begin{lem}\label{lem:int-0} If $(\vec{t}, J_1, J_2)$ satisfies $(\flat)$,
if $|\vec{t}|_{J}+\sharp(J_1)+\sharp(J_2)<\sharp(J_0)$, and if
$\widehat{\chi}$ is a $\Delta_{J_0\backslash (J_1\cup
J_2)}$-invariant element of $C^\flat_J$, then
$$ \int \widehat{\chi} \cdot \lambda(\vec{t}, J_1, J_2,a) \widetilde{\mu}_J =0. $$
\end{lem}
\begin{proof} We prove the assertion by induction on $|\vec{t}|_{J}$.
Since $|\vec{t}|_{J}+\sharp(J_1)+\sharp(J_2)<\sharp(J_0)$, there
exists $\mathfrak{p}\in J_0\backslash (J_1\cup J_2)$ such that
$t_{i}=0$ when $\mathrm{pr}_2(i)=\mathfrak{p}$.  Then $(\vec{t},
J_1\cup\{\mathfrak{p}\},J_2)$ satisfies ($\flat$) and thus
$\lambda(\vec{t}, J_1\cup\{\mathfrak{p}\}, J_2,a)$ is in
$C_J^\flat$.

 By Lemma \ref{lem:log-value}, when $|\vec{t}|_{J}=0$, $\ell^{\vec{t}}$
is $\beta_\mathfrak{p}^*$-invariant. Then
$$\lambda(\vec{t}, J_1, J_2,a)=
(1-\beta_\mathfrak{p}^*)\lambda(\vec{t}, J_1\cup\{\mathfrak{p}\},
J_2,a).$$ As $\widetilde{\mu}_J$ and $\widehat{\chi}$ are
$\beta_\mathfrak{p}$-invariant, we have
$$ \int \widehat{\chi} \cdot \lambda(\vec{t}, J_1, J_2,a) \widetilde{\mu}_J = 0 . $$

Now we assume that $|\vec{t}|_{J}>0$.   As $\widetilde{\mu}_J$ is
$\beta_\mathfrak{p}$-invariant, we have
$$ \int (1-\beta_\mathfrak{p}^*)\Big(\widehat{\chi}\cdot \lambda(\vec{t}, J_1\cup\{\mathfrak{p}\}, J_2, a)\Big) \widetilde{\mu}_J =0. $$
By (\ref{eq:1-gamma}) we obtain
\begin{eqnarray*}
&&(1-\beta_\mathfrak{p}^*) \Big(\widehat{\chi}\cdot\lambda(\vec{t},
J_1\cup\{\mathfrak{p}\},J_2,a)\Big) \\ &
= & \widehat{\chi}\cdot (1-\beta_\mathfrak{p}^*) ( 1_{Z_{J_1 \cup\{\mathfrak{p}\}, J_2,a}} \ell^{\vec{t}} ) \\
& = & \widehat{\chi}\cdot [ ((1-\beta_\mathfrak{p}^*)
\ell^{\vec{t}} )\cdot 1_{Z_{J_1 \cup\{\mathfrak{p}\}, J_2,a}} +
\ell^{\vec{t}} \cdot
(1-\beta_\mathfrak{p}^*) 1_{Z_{J_1 \cup\{\mathfrak{p}\},J_2,a}} 
- (1-\beta_\mathfrak{p}^*)\ell^{\vec{t}} \cdot
(1-\beta_\mathfrak{p}^*)
1_{Z_{J_1 \cup\{\mathfrak{p}\},J_2,a}} ] \\
& = &  \widehat{\chi}\cdot((1-\beta_\mathfrak{p}^*)  \ell^{\vec{t}}
)\cdot 1_{Z_{J_1 \cup\{\mathfrak{p}\},J_2,a}}+
\widehat{\chi}\cdot\ell^{\vec{t}} \cdot
 1_{Z_{J_1,J_2,a}} 
+ \widehat{\chi}\cdot ((1-\beta_\mathfrak{p}^*) \ell^{\vec{t}})\cdot
1_{Z_{J_1,J_2,a}}
\end{eqnarray*}

By Lemma \ref{lem:1-gamma}, $(1-\beta_\mathfrak{p}^*)
\ell^{\vec{t}}$ is a linear combination of $ \ell^{\vec{t}'}$ with
$\vec{t}'<_{J}\vec{t}$. Since $(\vec{t}, J_1\cup\{\mathfrak{p}\},
J_2)$ satisfies ($\flat$) and since $\vec{t}'<_{J}\vec{t}$, both
$(\vec{t}', J_1\cup\{\mathfrak{p}\},J_2)$ and $(\vec{t}', J_1, J_2)$
satisfy ($\flat$). On the other hand
$$|\vec{t}'|_{J}+\sharp(J_1)+\sharp(J_2)<|\vec{t}'|_{J}+\sharp(J_1\cup
\{\mathfrak{p}\})+\sharp(J_2)<|\vec{t}|_{J} +\sharp(J_1\cup
\{\mathfrak{p}\})+\sharp(J_2)\leq \sharp(J_0).$$ Thus by the
inductive assumption we have
$$ \int \widehat{\chi}\cdot((1-\beta_\mathfrak{p}^*)  \ell^{\vec{t}} )\cdot
1_{Z_{J_1\cup \{\mathfrak{p}\},J_2,a}} \widetilde{\mu}_J = \int
\widehat{\chi}\cdot ((1-\beta_\mathfrak{p}^*)  \ell^{\vec{t}} )\cdot
1_{Z_{J_1,J_2,a}}\widetilde{\mu}_J=0.
$$ It follows that
$\int \widehat{\chi}\cdot \ell^{\vec{t}} \cdot
 1_{Z_{J_1,J_2,a}} \widetilde{\mu}_J
=0$. \end{proof}

When $|\vec{t}|=0$ and $\sharp(J_1)+\sharp(J_2)=\sharp(J_0)$, we
have the following useful average vanishing result.

\begin{lem}\label{lem:average-vanishing} Let $\{a_k\}$ be a set of representatives in
$\BA^{\infty,\times}_K $ of the coset $\BA^{\infty,\times}_F
K^\times \backslash \BA^{\infty,\times}_K/
\widehat{\mathcal{O}}^\times_K$. Assume that $J_1\bigsqcup J_2 =
J_0$ and that $J_2$ is nonempty. Let $\widehat{\chi}$ be a character
of $\Gamma^-_J$ that factors through $\Gamma^-_{J\backslash J_1}$,
and assume that $\widehat{\chi}_\mathfrak{p}=1$ for some
$\mathfrak{p}\in J_2$. Then
$$ \sum_{k} \int \widehat{\chi}\cdot 1_{Z_{J_1,J_2,a_k}} \widetilde{\mu}_{J} =0. $$
\end{lem}
\begin{proof}
Note that the map $$\bigcup_k Z_{\emptyset,J_0,a_k} \rightarrow
K^\times \BA^{\infty,\times}_F  \backslash \BA^{\infty,\times}_K
$$ is surjective and sends
$[ \mathcal{O}^\times_K: \mathcal{O}^\times_F]$ elements to one
element. Thus, it follows from Proposition \ref{prop:transfer} that
$$ \sum_{a_k} \int \widehat{\chi}\cdot 1_{Z_{J_1,J_2,a_k}} \widetilde{\mu}_{J} = \sum_{a_k}  \int \widehat{\chi}\cdot 1_{{Z}_{\emptyset,J_0, a_k}}
\widetilde{\mu}_{J\backslash J_1} = [ \mathcal{O}^\times_K:
\mathcal{O}^\times_F] \int \widehat{\chi} \:\mu_{J\backslash J_1} =
[ \mathcal{O}^\times_K: \mathcal{O}^\times_F]
\mathscr{L}_{J\backslash J_1} (\pi,\widehat{\nu}\widehat{\chi}) .$$
Applying Theorem \ref{thm:interpol} to $J\backslash J_1$ instead of
$J$, we obtain $\mathscr{L}_{J\backslash J_1}
(\pi,\widehat{\nu}\widehat{\chi}) =0$. Note that
$\widehat{\nu}_\mathfrak{p}=1$ since $k_\mathfrak{p}=2$.
\end{proof}

\subsection{Some computation (II)} \label{ss:comp-II}

In this subsection we fix an element $l$ of $\LOG_J$, and a subset
$J_2$ of $J_0$.

If $\Xi$ is a subset of $J_0\backslash J_2$, we use
$\vec{t}_{l,\Xi}$ to denote the vector with
$$\vec{t}_{l,\Xi,i}=\left\{\begin{array}{ll} 1 & \text{ if
}i\in\{l\}\times \Xi, \\ 0 & \text{ otherwise.}
\end{array}\right.$$ For any element $a\in \BA_K^{\infty,\times}$ with $a_p=1$, we put
$$ \Lambda_{l,\Xi, J_2, a}=\lambda(\vec{t}_{l,\Xi}, (J_0\backslash J_2)\backslash  \Xi, J_2, a) .  $$ When $\Xi$ is the empty set $\emptyset$,
$\Lambda_{l, \emptyset, J_2,a}$ is independent of $l$, so we write
it by  $\Lambda_{\emptyset, J_2, a}$.

For every $J'\subseteq J_0\backslash J_2$ we denote by $M(J')$ the
set of maps $f: J' \rightarrow J\backslash J_2$ such that
\begin{quote} $f(S)\nsubseteq S$ for all $S\subseteq J'$, $S\neq
\emptyset$. \end{quote}  Let $\vec{t}_{l,f}$
be the vector with $$\vec{t}_{l, f,i}=\left\{\begin{array}{ll} \sharp(f^{-1}(v)) & \text{ if }i=(l, v), \\
0 & \text{ if }\mathrm{pr}_1(i)\neq l. \end{array}\right.$$

\begin{lem}\label{lem:complex} If $(\vec{t}, J_1,J_2)$ satisfies that $t_i=0$ unless $i\in \{l\} \times  (J\backslash (J_1\cup J_2))$, if $|\vec{t}|+\sharp(J_1)+\sharp(J_2)=\sharp(J_0)$,
and if $\widehat{\chi}$ is a $\Delta_{J_0\backslash (J_1\cup J_2)}$-invariant
element of $C^\flat_J$,
then
$$ \int \widehat{\chi} \cdot \lambda(\vec{t}, J_1, J_2, a) \widetilde{\mu}_J = \vec{t}\:!
\sum_{(\Xi, f)}  \left(\prod_{v \in J_0\backslash (J_1\cup J_2\cup
\Xi) }{_{f(v)}}l(\beta_v) \right)\int \widehat{\chi}\cdot
\Lambda_{l,\Xi,J_2,a}\widetilde{\mu}_J, $$ where the sum runs
through all pairs $(\Xi, f)$ with $\Xi\subseteq J_0\backslash
(J_1\cup J_2)$ and $f\in M((J_0\backslash (J_1\cup J_2 \cup \Xi))$
such that $\vec{t}_{l,f}+\vec{t}_{l,\Xi}=\vec{t}$.
\end{lem}
\begin{proof}
In the case  $t_i>0$ for each $i\in \{l\} \times  (J_0\backslash
(J_1\cup J_2))$, both sides of the equality are equal to $$\int
\widehat{\chi} \cdot \Lambda_{l, J_0\backslash (J_1\cup J_2) , J_2,
a}\widetilde{\mu}_J.$$  Indeed, so is the left hand side because we
must have $\vec{t}=\vec{t}_{l, J_0\backslash (J_1\cup J_2)}$. For
$(\Xi,f)$ appeared in the right hand side, we have
$$\vec{t}_{l,f}=\vec{t}_{l, J_0\backslash (J_1\cup
J_2)}-\vec{t}_{l, \Xi}=\vec{t}_{l, J_0\backslash (J_1\cup J_2\cup
\Xi)}.$$ This implies that $f(J_0\backslash (J_1\cup J_2 \cup
\Xi))=J_0\backslash (J_1\cup J_2 \cup \Xi)$. So the condition $f\in
M((J_0\backslash (J_1\cup J_2 \cup \Xi))$ forces
$J_0\backslash(J_1\cup J_2 \cup \Xi)=\emptyset$ or equivalently
$\Xi=J_0\backslash(J_1\cup J_2)$.

Suppose that $t_i=0$ for some $i=(\ell, \mathfrak{p})$,
$\mathfrak{p}\in J_0\backslash (J_1\cup J_2)$. By Lemma
\ref{lem:1-gamma} and Lemma \ref{lem:int-0} we have
\begin{eqnarray*} \int \widehat{\chi}\cdot \lambda(\vec{t}, J_1, J_2, a) \widetilde{\mu}_J
&=&- \int\widehat{\chi}\cdot
\Big((1-\beta_\mathfrak{p}^*)\prod_{v\in J} \ell ^{\vec{t}}
\Big)\cdot 1_{Z_{J_1\cup\{\mathfrak{p}\}, J_2,a}}
\widetilde{\mu}_J  \\
&=&   \int \widehat{\chi}\cdot \sum_{(f,
\vec{t}')}t_{l,f(\mathfrak{p})}\:{_{f(\mathfrak{p})}}l(\beta_\mathfrak{p})
\cdot \lambda(\vec{t}', J_1\cup\{\mathfrak{p}\},
J_2,a)\widetilde{\mu}_J ,
\end{eqnarray*} where $(f, \vec{t}')$ runs over the pairs with $f\in
M(\{\mathfrak{p}\})$ and $t_{l,f}+\vec{t}'=\vec{t}$. Now our
assertion follows by induction.
\end{proof}

We apply Lemma \ref{lem:complex} to the case that  $l$ satisfies
\begin{equation}\label{eq:sum-zero}
\sum_{v\in J\backslash J_2}\:_vl (\beta_\mathfrak{p})=0
\end{equation}
for each $\mathfrak{p}\in J_0\backslash J_2$. Put $h=|J_0|-|J_2|$,
and write $J_0\backslash J_2=\{\mathfrak{p}_1, \cdots,
\mathfrak{p}_h\}$.

\begin{prop}\label{prop:ell}
If $\widehat{\chi}$ is a $\Delta_{J_0\backslash   J_2}$-invariant
element of $C^\flat_J$, then {\tiny \begin{eqnarray*} && \int
\widehat{\chi}\left(\sum_{v\in J\backslash J_2} {
_vl}\right)^{h}\cdot 1_{Z_{\emptyset, J_2,a}}\widetilde{\mu}_J \\
&=& h! \int \widehat{\chi} \det \left(\begin{array}{cccc}
\Lambda_{l,\mathfrak{p}_1,J_2,a} -
{_{\mathfrak{p}_1}l(\beta_{\mathfrak{p}_1})}\Lambda_{\emptyset,J_2,a}
&
-{_{\mathfrak{p}_1}l(\beta_{\mathfrak{p}_2})}\Lambda_{\emptyset,J_2,a}
& \cdots & -{_{\mathfrak{p}_1}l(\beta_{\mathfrak{p}_h})}\Lambda_{\emptyset,J_2,a} \\
-{_{\mathfrak{p}_2}l(\beta_{\mathfrak{p}_1})}\Lambda_{\emptyset,J_2,a}
& \Lambda_{l,\mathfrak{p}_2,J_2,a} -
{_{\mathfrak{p}_2}l(\beta_{\mathfrak{p}_2})}\Lambda_{\emptyset,J_2,a}
& \cdots & -{_{\mathfrak{p}_2}l(\beta_{\mathfrak{p}_h})}\Lambda_{\emptyset,J_2,a} \\
\vdots & \vdots & \ddots & \vdots \\
-{_{\mathfrak{p}_h}l(\beta_{\mathfrak{p}_1})}\Lambda_{\emptyset,J_2,a}
&
-{_{\mathfrak{p}_h}l(\beta_{\mathfrak{p}_2})}\Lambda_{\emptyset,J_2,a}
& \cdots & \Lambda_{l,\mathfrak{p}_h,J_2,a}
-{_{\mathfrak{p}_h}l(\beta_{\mathfrak{p}_h})}\Lambda_{\emptyset,J_2,a}
 \end{array} \right) \widetilde{\mu}_J \end{eqnarray*}}
\end{prop} Here we denote $\Lambda_{l,\{\mathfrak{p}_i\},J_2,a}$
by $\Lambda_{l, \mathfrak{p}_i ,J_2,a}$.
\begin{proof} Appliying Lemma \ref{lem:complex} to the case of $J_1=\emptyset$ we
obtain
$$\int \widehat{\chi}\cdot\left(\sum_{v\in J\backslash J_2} { _vl}\right)^{h}\cdot 1_{Z_{\emptyset, J_2},a}\widetilde{\mu}_J =
h!\sum_{(\Xi, f)} \left(\prod_{v \in (J_0\backslash J_2)\backslash
\Xi} {_{f(v)}l}(\beta_v) \right)\int
\widehat{\chi}\cdot\Lambda_{l,\Xi,J_2,a}\widetilde{\mu}_J,$$ where
the sum runs through all pairs $(\Xi, f)$ with $\Xi\subseteq
J_0\backslash J_2$ and $f\in M(J_0\backslash (J_2 \cup \Xi) )$.

On the other hand,{\small
\begin{eqnarray*} & & \det \left(\begin{array}{cccc} \Lambda_{l,\mathfrak{p}_1,J_2,a}
-
{_{\mathfrak{p}_1}l(\beta_{\mathfrak{p}_1})}\Lambda_{\emptyset,J_2,a}
&
-{_{\mathfrak{p}_1}l(\beta_{\mathfrak{p}_2})}\Lambda_{\emptyset,J_2,a}
& \cdots &
-{_{\mathfrak{p}_1}l(\beta_{\mathfrak{p}_h})}\Lambda_{\emptyset,J_2,a}
\\ -{_{\mathfrak{p}_2}l(\beta_{\mathfrak{p}_1})}\Lambda_{\emptyset,J_2,a} & \Lambda_{l,\mathfrak{p}_2,J_2,a}
- {_{\mathfrak{p}_2}l(\beta_{\mathfrak{p}_2})}\Lambda_{\emptyset,J_2,a} & \cdots & -{_{\mathfrak{p}_2}l(\beta_{\mathfrak{p}_h})}\Lambda_{\emptyset,J_2,a} \\
\vdots & \vdots & \ddots & \vdots \\
-{_{\mathfrak{p}_h}l(\beta_{\mathfrak{p}_1})}\Lambda_{\emptyset,J_2,a}
& -{_{\mathfrak{p}_h}l(\beta_{\mathfrak{p}_2})}
\Lambda_{\emptyset,J_2,a} & \cdots & \Lambda_{l,
\mathfrak{p}_h,J_2,a}
-{_{\mathfrak{p}_h}l(\beta_{\mathfrak{p}_h})}\Lambda_{\emptyset,J_2,a}
 \end{array} \right) \\ & & =  \sum_{\Xi\subseteq J_0\backslash J_2} d_\Xi \Lambda_{l,\Xi,J_2, a}
 \end{eqnarray*}}
 where
 $$d_\Xi=\det(-{ _{\mathfrak{p}_j} }\ell(\beta_{\mathfrak{p}_i}))_{\mathfrak{p}_i,\mathfrak{p}_j\in J_0\backslash (J_2\cup \Xi)}.$$

As (\ref{eq:sum-zero}) holds for each $\mathfrak{p}\in J_0\backslash
J_2$, our assertion follows from the following lemma.
\end{proof}

\begin{lem} \cite[Lemma 4.9]{Spiess}  Let $k\leq m$ be positive integers, and let $(c_{ij})_{1\leq i\leq k, 1\leq j\leq m}$
be a $k\times m$-matrix with entries in a commutative ring such that $\sum_{j=1}^m c_{ij}=0$ for all $i=1,\cdots, k$. Then
$$ \det(-c_{ij})_{1\leq i,j\leq k} = \sum_f \prod_{i=1}^k c_{i f(i)} $$ where the sum runs through
all maps $f: \{1,\cdots, k\}\rightarrow \{1,\cdots, m\}$ with $f(S)\nsubseteq S$ for all $S\subseteq \{1,\cdots, k\}$, $S\neq \emptyset$.
\end{lem}

\section{$L$-invariants and group cohomology} \label{sec:L-inv-gh}

\subsection{Notations}

For each $\mathfrak{p}\in J_0$ let $\beta_\mathfrak{p}$ be the
element in $ K^\times/F^\times$ provided in Section \ref{ss:comp-I}.
Let $\Delta=\Delta_{J_0}$ be the subgroup of $ K^\times/F^\times$
generated by $\beta_\mathfrak{p}$ ($\mathfrak{p}\in J_0$). Then
$\Delta$ is a free abelian group with generators
$\beta_\mathfrak{p}$ ($\mathfrak{p}\in J_0$).

We fix a $J\backslash J_0$-tuple $\vec{n}$.

For a subset $J_2\subset J_0$ we consider the spaces
$$ \widetilde{X}^{J_2}_{J_0;\vec{n}, \mathfrak{c}} =
(F_{J_0\backslash J_2}^\times\backslash K_{J_0\backslash
J_2}^\times) \times (\BA_F^{\infty,J_0,\times}\backslash
\BA_K^{\infty,J_0,\times} / \widehat{\mathcal{O}}_{J_0;\vec{n},
\mathfrak{c}}^\times) $$ and
$$\overline{X}^{J_2}_{J_0;\vec{n},\mathfrak{c}} :=
F_{J_0\backslash J_2}\times (\BA_F^{\infty,J_0,\times}\backslash
\BA_K^{\infty,J_0,\times} / \widehat{\mathcal{O}}_{J_0;\vec{n},
\mathfrak{c}}^\times). $$ Via $\eta_{J_2}$ we identify
$\widetilde{X}^{J_2}_{J_0;\vec{n}, \mathfrak{c}}$ with a subset of
$\overline{X}^{J_2}_{J_0;\vec{n},\mathfrak{c}}$. If $J_0\backslash
J_2=J_{1,1}\bigsqcup J_{1,2}$ is a disjoint union, we have a set
$$\overline{X}^{J_2}_{J_0;\vec{n},\mathfrak{c}}(J_{1,1}, J_{1,2}) :=
(F_{J_{1,1}} \times   F_{J_{1,2}}^\times )\times
(\BA_F^{\infty,J_0,\times}\backslash \BA_K^{\infty,J_0,\times} /
\widehat{\mathcal{O}}_{J_0;\vec{n}, \mathfrak{c}}^\times). $$

Let $\bar{C}^{J_2}(J_{1,1}, J_{1,2})$ be the completion of the space
of compact supported functions on
$\overline{X}^{J_2}_{J_0;\vec{n},\mathfrak{c}}(J_{1,1}, J_{1,2})$
for the Gauss norm. When $J_{1,1}=J_0\backslash J_2$ and
$J_{1,2}=\emptyset$, we denote it by $\bar{C}^{J_2}$. Then
$\bar{C}^{J_2}(J_{1,1}, J_{1,2})$ is closed in $\bar{C}^{J_2}$ for
each pair $(J_{1,1}, J_{1,2})$.

Let $\Delta$ act on the spaces
$\overline{X}^{J_2}_{J_0;\vec{n},\mathfrak{c}}$ and
$\overline{X}^{J_2}_{J_0;\vec{n},\mathfrak{c}}(J_{1,1}, J_{1,2})$.
These actions induce actions of $\Delta$ on $\bar{C}^{J_2}$ and
$\bar{C}^{J_2}(J_{1,1}, J_{1,2})$.

Let $I_{J_{1,2}}$ be the subset of $I$ consisting of
$i=(\log_\iota,\mathfrak{p})$ with $\mathfrak{p}\in J_{1,2}$ and
$\iota\in \Sigma_\mathfrak{p}$. Let $\mathcal{C}^{J_2,\leq N} $
denote the subspace
$$ \bigoplus_{J_{1,2}\subset J_0\backslash J_2} \bigoplus_{\tiny\begin{array}{c} \vec{m}:
I_{J_{1,2}}\text{-tuple}, \\  \vec{m}>0, |\vec{m}|\leq N
\end{array}} \left(\bar{C}^{J_2}(J_{1,1}, J_{1,2})\right)^{\Delta_{J_2}} \cdot
\ell^{\vec{m}}
$$ endowed with the product topology, where $\vec{m}>0$
means each component $m_{\iota,\mathfrak{p}}$ of $\vec{m}$ is
positive. There exists a natural action of $\Delta$ on
$\mathcal{C}^{J_2,\leq N}$.

Similarly, let $\bar{C}(F_{J_2})$ denote the completion of the space
of compactly supported functions on $F_{J_2}=\prod_{\mathfrak{p}\in
J_2}F_{\mathfrak{p}}$. Then $\Delta$ acts on $\bar{C}(F_{J_2})$. Let
$\mathcal{C}_{J_2}^{\leq N}$ be the space
$$  \bigoplus_{J'\subset J_2} \bigoplus_{\tiny\begin{array}{c} \vec{m}:
I_{J'}\text{-tuple}, \\  \vec{m}>0, |\vec{m}|\leq N
\end{array}} \bar{C}(F_{J_2\backslash J'}\times F_{J'}^\times)^{\Delta_{J_0\backslash J_2}} \cdot
\ell^{\vec{m}} $$ equipped with the product topology.

\subsection{$L$-invariants revisited} \label{sec:partial}

Let $\mathfrak{p}$ be in $J_0$. Write
$\mathfrak{p}\mathcal{O}_K=\mathfrak{P}\bar{\mathfrak{P}}$. Then
$K_\mathfrak{p}=K_\mathfrak{P}\oplus K_{\bar{\mathfrak{P}}}$.  Write
$$ \beta_{\mathfrak{p},\mathfrak{p}}=(u_1\omega_{\mathfrak{p}}^{-h_\mathfrak{p}},
u_2)\in K_\mathfrak{p}=K_{\mathfrak{P}}\oplus
K_{\bar{\mathfrak{P}}}\cong F_\mathfrak{p}\oplus F_{\mathfrak{p}} $$
for $u_1,u_2\in \mathcal{O}_{F_\mathfrak{p}}^\times$. Put
$u_\mathfrak{p}=\frac{u_1}{u_2\omega_\mathfrak{p}^{h_\mathfrak{p}}}$.

\begin{prop}\label{prop:deriv-induction}
If $\xi\in (\bar{C}^\mathfrak{p})^{\Delta_\mathfrak{p}}$, then for
each $\iota\in\Sigma_\mathfrak{p}$ we have
\begin{equation}\label{eq:key}  \int
 { _\mathfrak{p}}\log_{\iota }\cdot 1_{D_\mathfrak{p}}\cdot
 \xi \: \widetilde{\mu}_J =  \mathcal
{L}^\Tei_{\iota, u_\mathfrak{p}^\iota}(f) \int \left(
\sum_{i=1}^{h_\mathfrak{p}} 1_{\omega_\mathfrak{p}^{i}
B_\mathfrak{p}}\right) \cdot \xi \:\widetilde{\mu}_{J}.
\end{equation}
\end{prop}

Note that ${ _\mathfrak{p}}\log_{\iota }$ is exactly
$$\log_{u_\mathfrak{p}^\iota}^{\iota}: K_\mathfrak{p}^\times \rightarrow \BC_p ,\ \ \ \ (a_1,a_2)\mapsto
\log_{u_\mathfrak{p}^\iota}((a_1/a_2)^\iota).$$ Indeed, both of them
are additive characters that coincide with each other on
$\mathcal{O}^\times_{K_\mathfrak{p}}$ and vanish at
$\beta_{\mathfrak{p},\mathfrak{p}}$.

Since $\widetilde{\mu}_J$ is a measure, to prove (\ref{eq:key}) we
may restrict to the case that $\xi$ is locally constant. Put
$J'_0=J_0\backslash \{\mathfrak{p}\}$. We write $$ \xi=\sum_{k} d_k
1_{\mathcal{U}(a_k, m_k)\times \BA_F^{\infty,J_0,\times} t_k
\widehat{\mathcal{O}}_{J_0;\vec{n}, \mathfrak{c}}^\times}, $$ where
$\mathcal{U}(a_k, m_k)$ are open discs of $F_{J'_0}$, and $d_k$ are
constants in $\BC_p$. Let $h_k$ be an element of $\GL_2(F_{J'_0})$
such that $\mathcal{U}(a_k, m_k)=h_k \mathcal{U}_{e_{J'_0,*}}$.

Write $\tilde{\phi} = \sum_j b_j\pi'(g_j)\varphi$ with $b_j\in E$
and $(g_j)_{\infty,p}=1$. Then $\tilde{\phi}^{\infty,\dagger_J}=
\sum_j b_j\pi'(g_j)\varphi^{\infty,\dagger_J}$. Let $c(\cdot,\cdot)$
be the $\mathfrak{p}$-type ``harmonic cocycle'' valued modular form
attached to $\widehat{\varphi^{\infty,\dagger_J}}$ (Section
\ref{ss:harm-mod-form}).

Let $V$ be a compact open subgroup of
$\BA_{K}^{\infty,\mathfrak{p},\times}$ such that
$$i_p^\mathfrak{p}(V)\subset \varsigma_{J\backslash J_0}^{(\vec{n})}h_kg_jU^{\mathfrak{p}}g_j^{-1}h_k^{-1}(\varsigma_{J\backslash J_0}^{(\vec{n})})^{-1}$$ for any $j, k$.
Let $r$ be a positive integer such that $\beta_\mathfrak{p}^r$ is
contained in $V$. Enlarging $r$ if necessary we may assume that the
out-$p$ part $(\beta_\mathfrak{p}^r)^p$ of $\beta_\mathfrak{p}^r$ is
in $\widehat{\mathcal{O}}^\times_{\mathfrak{c}}$.

For $\ell=t_k\varsigma_{J\backslash J_0}^{(\vec{n})} h_k g_j$ let
$\gamma_\ell$ be the element
$\iota_\ell^{\mathfrak{p}}(\beta_{\mathfrak{p}}^r)$ in
$\Gamma^\mathfrak{p}_\ell$. Note that the $\mathfrak{p}$-component
of $\gamma_{\ell}$ is $\beta_{\mathfrak{p},\mathfrak{p}}^r$,
independent of $j$ and $k$. We write $\gamma_\mathfrak{p}$ for it.
The out-$J_0$-part of $\gamma_{\ell}$ is also independent of $j$ and
$k$; we write $\gamma^{J_0}$ for it.

Fix a point $z_0\in \mathcal{H}_\mathfrak{p}$. Put $$I_{ \ell
}:=\int_{\mathbf{P}^1(F_\mathfrak{p})} \log_{
u_\mathfrak{p}^\iota}(\frac{ \iota (x)-\gamma_\mathfrak{p}
z_0}{\iota (x)-z_0})\mu _{c(\ell ,\cdot)}.$$

\begin{lem}\label{lem:key-compute} We have
$$I_{\ell}=\int_{\mathcal{F}} ( \log_{u_\mathfrak{p}^\iota}( z_0) - \log_{u_\mathfrak{p}^\iota} \iota(x) )
 \mu_{c(\ell,\cdot)}.$$
\end{lem}
\begin{proof}  Let
$\mathcal{F}=\mathcal{F}_{\gamma_\mathfrak{p}}$ be a fundament
domain in $\mathbf{P}^1(F_\mathfrak{p})\backslash \{0, \infty\}$ for
the action of $\gamma_\mathfrak{p}$. We may take
$\mathcal{F}=\bigcup\limits_{i=0}^{r-1} \beta_{\mathfrak{p}
,\mathfrak{p}}^{-i}D_\mathfrak{p}$.

By Remark \ref{rem:gamma-inv}, $\mu_{c(\ell,\cdot)}$ is
$\Gamma^\mathfrak{p}_{\ell}$-invariant. So
\begin{eqnarray*}   I_{\ell} &=& \int_{\mathbf{P}^1(F_\mathfrak{p})} \log_{
u_\mathfrak{p}^\iota}(\frac{ \iota (x)-\gamma_\mathfrak{p}
z_0}{\iota (x)-
z_0})\mu _{c(\ell,\cdot)} \\
&=& \int_{\mathcal{F}} \sum_{j=-\infty}^{\infty}
\log_{u_\mathfrak{p}^\iota}(\frac{\iota(\gamma_\mathfrak{p}^j
x)-\gamma_\mathfrak{p} z_0}{\iota (\gamma_\mathfrak{p}^j x)-
z_0})\mu _{c(\ell,\cdot)}
\end{eqnarray*} As $
\mathfrak{i}_\mathfrak{p}i_\mathfrak{p}^{-1}(\gamma_\mathfrak{p})
=\wvec{u_1^r\omega_\mathfrak{p}^{-rh_\mathfrak{p}}}{0}{0}{u_2^r},$
we have {\allowdisplaybreaks
\begin{eqnarray*} && \sum_{j=-\infty}^{\infty}
\log_{u_\mathfrak{p}^\iota}(\frac{\iota( \gamma_\mathfrak{p}^j x)-
 \gamma_\mathfrak{p} z_0}{\iota ( \gamma_\mathfrak{p}^j
x)- z_0}) \hskip 5pt = \hskip 5pt \sum_{j=-\infty}^{\infty}
\log_{u_\mathfrak{p}^\iota}(\frac{\iota(u_\mathfrak{p}^{rj}  x) -
\iota(u_\mathfrak{p}^r) z_0}{\iota(u_\mathfrak{p}^{rj}  x)-z_0}) \\
&=& \sum_{j=-\infty}^{\infty}
\log_{u_\mathfrak{p}^\iota}(\frac{\iota(u_\mathfrak{p}^{r(j-1)} x) -
z_0}{\iota(u_\mathfrak{p}^{rj} x)-z_0}) \hskip 5pt = \hskip 5pt
\lim_{N\rightarrow +\infty}\log_{u_\mathfrak{p}^\iota}
(\frac{\iota(u_\mathfrak{p}^{-r(N+1)}
x)-z_0}{\iota(u_\mathfrak{p}^{rN} x)-z_0}) \\ &=& \lim_{N\rightarrow
+\infty}\log_{u_\mathfrak{p}^\iota}
(\frac{\iota(u_\mathfrak{p}^{-r(N+1)}
x)-z_0}{\iota(x)-\iota(u_\mathfrak{p}^{-rN} )z_0}) \hskip 5pt =
\hskip 5pt
\log_{u_\mathfrak{p}^\iota}z_0-\log_{u_\mathfrak{p}^\iota}\iota(x).
\end{eqnarray*}}

\noindent Thus we obtain the equality in the lemma.
\end{proof}

\begin{lem}\label{lem:fixed} $\mathfrak{i}_p(\varsigma_{J\backslash J_0}^{(\vec{n})})^{-1}
\widehat{\mathbf{v}}_{\mathbf{m}}$ is fixed by
$\mathfrak{i}_p(\gamma_\ell)$.
\end{lem}
\begin{proof} Since $k_\mathfrak{q}=2$ for each $\mathfrak{q}\in J_0$, $J_0$-part
of $\mathfrak{i}_p(\gamma_\ell)$ acts trivially on $V_{\vec{k}}(E)$.
So,
$$ \mathfrak{i}_p(\gamma_\ell)\mathfrak{i}_p(\varsigma_{J\backslash J_0}^{(\vec{n})})^{-1}
\widehat{\mathbf{v}}_{\mathbf{m}}=\mathfrak{i}_p(\gamma^{J_0})\mathfrak{i}_p(\varsigma_{J\backslash
J_0}^{(\vec{n})})^{-1} \widehat{\mathbf{v}}_{\mathbf{m}} =
\mathfrak{i}_p(\varsigma_{J\backslash J_0}^{(\vec{n})})^{-1}
\mathfrak{i}_p(\beta^{J_0}_{\mathfrak{p},p})^r
\widehat{\mathbf{v}}_{\mathbf{m}} =
\widehat{\nu}(\beta^r_{\mathfrak{p},p})\mathfrak{i}_p(\varsigma_{J\backslash
J_0}^{(\vec{n})})^{-1} \widehat{\mathbf{v}}_{\mathbf{m}}.  $$ Here
$\beta^{J_0}_{\mathfrak{p},p}$ denotes the $J_p\backslash J_0$-part
of $\beta$. But
$\widehat{\nu}(\beta^r_{\mathfrak{p},p})=\widehat{\nu}((\beta_\mathfrak{p}^r)^p)^{-1}=1$
since $(\beta_\mathfrak{p}^r)^p$ is in
$\widehat{O}_{\mathfrak{c}}^\times$ and since $\widehat{\nu}$ is
trivial on $\widehat{O}_{\mathfrak{c}}^\times$. The reader should be
cautious that $(\beta_\mathfrak{p}^r)^p$ denotes the out-$p$ part of
$\beta_\mathfrak{p}^r$, not the $p$-th power of
$\beta_\mathfrak{p}^r$.
\end{proof}

For a choice $z'_0$ different from $z_0$, by Lemma \ref{lem:lambda}
(\ref{it:lambda-c}) there exists   $v\in L_{\vec{k}}(E)$ such that
\begin{eqnarray*} (\gamma-1)v &=& \int_{\mathbf{P}^1(F_\mathfrak{p})} \log_{
u_\mathfrak{p}^\iota}(\frac{ \iota (x)-\gamma z_0}{\iota (x)-
z_0})\mu _{c(\ell,\cdot)}- \int_{\mathbf{P}^1(F_\mathfrak{p})}
\log_{ u_\mathfrak{p}^\iota}(\frac{ \iota (x)-\gamma z'_0}{\iota
(x)- z'_0})\mu
_{c(\ell,\cdot)} \\
&=&
(\log_{u_\mathfrak{p}^\iota}z_0-\log_{u_\mathfrak{p}^\iota}(z'_0))\int
1_\mathcal{F} \mu _{c(\ell,\cdot)}.
\end{eqnarray*} Thus by Lemma \ref{lem:fixed} we have
$$ \langle \mathfrak{i}_p(\varsigma_{J\backslash J_0}^{(\vec{n})})^{-1}
\widehat{\mathbf{v}}_{\mathbf{m}} ,\int 1_\mathcal{F} \mu
_{c(\ell,\cdot)}\rangle =0 .$$ Applying Proposition
\ref{prop:vary-J} we obtain
$$   \int 1_{\mathcal{F}}\cdot
 { _\mathfrak{p}}\ell_{\iota }\cdot \xi \: \widetilde{\mu}_J  = - \frac{1}{\jmath(\prod_{\mathfrak{q}\in
J\backslash J_0}\alpha_\mathfrak{q}^{n_\mathfrak{q}})}\sum_{j,k}
b_jd_k\widehat{\nu}(t_k)\langle
\mathfrak{i}_p(\varsigma_{J\backslash J_0}^{(\vec{n})})^{-1}
\widehat{\mathbf{v}}_{\mathbf{m}}, I_{t_k \varsigma_{J\backslash
J_0}^{(\vec{n})}h_kg_j} \rangle.
$$

\begin{lem}\label{lem:repeat-L-inv} For each $\ell=t_k
\varsigma_{J\backslash J_0}^{(\vec{n})}h_kg_j$, $$\langle
\mathfrak{i}_p(\varsigma_{J\backslash J_0}^{(\vec{n})}))^{-1}
\widehat{\mathbf{v}}_{\mathbf{m}},I_\ell\rangle
=\mathcal{L}^\Tei_{\iota, u_\mathfrak{p}^\iota}(f)\langle
\mathfrak{i}_p(\varsigma_{J\backslash J_0}^{(\vec{n})}))^{-1}
\widehat{\mathbf{v}}_{\mathbf{m}},
\delta_{\Gamma^\mathfrak{p}_\ell}c(\ell,\cdot)(\gamma_\ell)\rangle.$$
\end{lem}
\begin{proof} In Section \ref{ss:coleman} we proved that there exists $y\in L^{\iota}(E)$ such that
$$ I_{\ell} = \mathcal{L}_{\iota, u_\mathfrak{p}}(f)\delta_{\Gamma^\mathfrak{p}_\ell}c(\ell,\cdot)(\gamma_\ell) + (\mathfrak{i}_p(\gamma_\ell)-1)y. $$
By Lemma \ref{lem:fixed} we have $\langle
\mathfrak{i}_p(\varsigma_{J'}^{(\vec{n})})^{-1}
\widehat{\mathbf{v}}_{\mathbf{m}}, (\mathfrak{i}_p(\gamma_\ell)-1)y
\rangle = \langle
(\mathfrak{i}_p(\gamma_\ell)-1)\mathfrak{i}_p(\varsigma_{J'}^{(\vec{n})})^{-1}\widehat{\mathbf{v}}_{\mathbf{m}},
y\rangle=0 $.
\end{proof}

\noindent{\it Proof of Proposition \ref{prop:deriv-induction}.}
Write $\ell=t_k \varsigma_{J\backslash J_0}^{(\vec{n})}h_kg_j$. For
any vertex $v$ in the tree $\mathcal{T}_{\mathfrak{p}}$,
$$\delta_{\Gamma_\ell}c(\ell,\cdot)(\gamma_\ell)= \sum_{e\in v\rightarrow \gamma_\mathfrak{p} v} c(\ell,e).
$$ We may take $v= \omega_\mathfrak{p}^{-rh_\mathfrak{p}}\mathcal{O}_{F_\mathfrak{p}}\oplus
\mathcal{O}_{F_\mathfrak{p}}$. Then the edges in the chain
$v\rightarrow \gamma_\mathfrak{p} v$ are
$$i_\mathfrak{p}((\omega_{\mathfrak{p}}^{i},1))
\bar{e}_{\mathfrak{p},0} \hskip 10pt (1\leq i\leq 2rh_\mathfrak{p}).
$$ Thus when $\ell$ runs over the set $\{t_k \varsigma_{J\backslash
J_0}^{(\vec{n})}h_kg_j\}_{j,k,t}$, the sum
$$\frac{1}{\jmath(\prod_{\mathfrak{q}\in
J\backslash J_0}\alpha_\mathfrak{q}^{n_\mathfrak{q}})}\sum_{j,k}
b_jd_k\widehat{\nu}(t_k)\langle
\mathfrak{i}_p(\varsigma_{J'}^{(\vec{n})}))^{-1}
\widehat{\mathbf{v}}_{\mathbf{m}}, \delta_{\Gamma^\mathfrak{p}_{t_k
\varsigma_{J\backslash J_0}^{(\vec{n})}h_kg_j}}c(t_k
\varsigma_{J\backslash J_0}^{(\vec{n})}h_kg_j,\cdot)(\gamma_{t_k
\varsigma_{J\backslash J_0}^{(\vec{n})}h_kg_j})\rangle$$ is exactly
$-\int (\sum\limits_{i=1}^{rh_\mathfrak{p}}1_{
\omega_\mathfrak{p}^{i} B_\mathfrak{p}})
 \xi \:\widetilde{\mu}_{J}$. As
$\beta_\mathfrak{p}^*(\xi)=\xi$, we have
$$ \beta_\mathfrak{p}^*(\sum_{i=(s-1)h_\mathfrak{p}+1}^{sh_\mathfrak{p}}1_{
\omega_\mathfrak{p}^{i} B_\mathfrak{p}}
 \cdot \xi) = \sum_{i=sh_\mathfrak{p}+1}^{(s+1)h_\mathfrak{p}} 1_{\omega_\mathfrak{p}^{i}
B_\mathfrak{p}}
 \cdot \xi $$ and thus
$$\int
(\sum_{i=1}^{rh_\mathfrak{p}} 1_{ \omega_\mathfrak{p}^{i}
B_\mathfrak{p}})
 \xi \:\widetilde{\mu}_{J}=r\int (\sum_{i=1}^{h_\mathfrak{p}} 1_{ \omega_\mathfrak{p}^{i} B_\mathfrak{p}}) \xi
\:\widetilde{\mu}_{J}.$$ Similarly,
$$\int
1_{\mathcal{F}}\cdot  { _\mathfrak{p}}\log_{\iota } \cdot
 \xi \:\widetilde{\mu}_{J}=r\int 1_{D_\mathfrak{p}}\cdot  { _\mathfrak{p}}\log_{\iota } \cdot
 \xi
\:\widetilde{\mu}_{J}.$$   Now Proposition
\ref{prop:deriv-induction} follows from Lemma
\ref{lem:repeat-L-inv}. \qed

\subsection{Cohomology of $\Delta$}

\begin{lem}
\begin{enumerate}
\item\label{it:strict-a} For each $\gamma\in \Delta_{J_2}$, $\gamma^*-1$ is strict on
$\mathcal{C}^{J_2,\leq N}$ for the given topology.
\item\label{it:strict-b} $\gamma^*-1: \mathcal{C}^{J_2,\leq N+1}\rightarrow \mathcal{C}^{J_2,\leq
N}$ is surjective.
\end{enumerate}\end{lem}
\begin{proof}
We consider the big set
$$ \widetilde{\mathcal{C}}^{J_2,\leq N}:=\bigoplus_{\tiny\begin{array}{c}\vec{m}:I_{J_0\backslash J_2}\text{-tuple}, \\ |\vec{m}|\leq N \end{array}}
(\bar{C}^{J_2})^{\Delta_{J_2}}\cdot\ell^{\vec{m}}
$$ endowed  with the product
topology. The map
$$\gamma^*-1: \widetilde{\mathcal{C}}^{J_2,\leq N}\rightarrow \widetilde{\mathcal{C}}^{J_2,\leq N}$$ is a homomorphism of finite rank
$(\bar{C}^{J_2})^{\Delta_{J_2}} $-module, and the matrix of
$\gamma^*-1$ with respect to the basis $$\{\ell^{\vec{m}}: \vec{m}
\text{ is an }I_{J_0\backslash J_2}\text{-tuple},  |\vec{m}|\leq N
\}$$ has entries in $\BC_p$. Thus $\gamma^*-1$ is strict.

When $J'\supset J_{1,2}$, $\bar{C}^{J_2}(J_{1,1},
J_{1,2})^{\Delta_{J_2}} $ is  closed in
$\bar{C}^{J_2}((J_0\backslash J_2)\backslash J', J')^{\Delta_{J_2}}
$. Thus $\mathcal{C}^{J_2,\leq N}$ is  closed in
$\widetilde{\mathcal{C}}^{J_2,\leq N}$. It follows that $\gamma^*-1:
\mathcal{C}^{J_2,\leq N} \rightarrow \mathcal{C}^{J_2,\leq N}$ is
strict. This proves (\ref{it:strict-a}).

The proof of (\ref{it:strict-b}) is easy and omitted.
\end{proof}

\begin{rem} $\widetilde{\mathcal{C}}^{J_2,\leq N}$ is not a subset of
$\bar{C}^{J_2}$.
\end{rem}

Let $\mathcal{D}^{J_2,\leq N}$ and $\Dist^{J_2}$ be the dual spaces
of $\mathcal{C}^{J_2,\leq N}$ and $\bar{C}^{J_2}$ respectively. The
natural map $\mathcal{C}^{J_2,\leq N}\rightarrow \bar{C}^{J_2}$ is
continuous and $\Delta$-equivariant. So by restricting we obtain a
$\Delta$-equivariant continuous map $\Dist^{J_2}\rightarrow
\mathcal{D}^{J_2,\leq N}$.

If $J_3$ is a subset of $J_0$ disjoint from $J_2$, then we have a
$\Delta$-equivariant map
$$ \mathcal{C}_{J_3}^{\leq M}\times \mathcal{D}^{J_2,\leq M+N}  \rightarrow \mathcal{D}^{J_2\cup J_3, \leq N}.
$$ We derive from it a pairing
$$ H^i(\Delta, \mathcal{C}_{J_3}^{\leq M}) \times H^j(\Delta, \mathcal{D}^{J_2,\leq M+N}) \rightarrow H^{i+j}(\Delta, \mathcal{D}^{J_2\cup J_3, \leq N})
.$$ Similarly, we have a $\Delta$-equivariant map
$$ \mathcal{C}_{J_3}^{\leq M}\times \mathcal{C}_{J_2}^{\leq N}  \rightarrow \mathcal{C}_{J_2\cup J_3}^{ \leq M+N}.
$$ We derive from it a pairing
$$ H^i(\Delta, \mathcal{C}_{J_3}^{\leq M}) \times H^j(\Delta, \mathcal{C}_{J_2}^{\leq N}) \rightarrow H^{i+j}(\Delta, \mathcal{C}_{J_2\cup J_3}^{ \leq M+N})
.$$

For each $\mathfrak{p}\in J_0$ let $\ord_\mathfrak{p}$ be the
function on $ F_{\mathfrak{p}}^\times$ such that
$\mathrm{ord}_\mathfrak{p}(u\omega_\mathfrak{p}^n)=n $ for every
$u\in \mathcal{O}_{F_{\mathfrak{p} } }^\times $. We attach to
$\mathrm{ord}_\mathfrak{p}$ the $1$-cocycle on $\Delta$ with values
in $\mathcal{C}_\mathfrak{p}^{\leq 1}$: For each $\gamma\in \Delta$
we put
$$ c_{\mathfrak{p},\ord}(\gamma)=(\gamma^*-1)(-\mathrm{ord}_\mathfrak{p}\cdot 1_{\mathcal{O}_\mathfrak{p}}). $$
For $l\in \LOG_J$ we define $c_{\mathfrak{p},l}$ by $$
c_{\mathfrak{p},l}(\gamma)=(\gamma^*-1)(-{_\mathfrak{p}}l\cdot
1_{\mathcal{O}_\mathfrak{p}}) .$$ When $l=\log_\sigma$
($\sigma\in\Sigma_J$), we also write $c_{\mathfrak{p},\sigma,\log}$
for $c_{\mathfrak{p},\log_\sigma}$. If
$l=\sum_{\sigma\in\Sigma_J}s_\sigma \log_\sigma$, then
$$c_{\mathfrak{p},l}=\sum_{\sigma\in \Sigma_J}s_\sigma c_{\mathfrak{p},\sigma,\log}.$$

We have
$$ c_{\mathfrak{p},\ord}(\beta_\mathfrak{q})
= \left \{ \begin{array}{cl} \sum\limits_{i=1}^{h_\mathfrak{p}} 1_{\omega_\mathfrak{p}^i B_\mathfrak{p}} & \text{if }\mathfrak{q}=\mathfrak{p} \\
0 & \text{if }\mathfrak{q}\neq \mathfrak{p}.
\end{array} \right.
$$ When $\sigma\in \Sigma_\mathfrak{p}$, $$ c_{\mathfrak{p},\sigma,\log}(\beta_\mathfrak{q})
= \left \{ \begin{array}{cl}{ _\mathfrak{p}}\log_\iota \cdot 1_{D_\mathfrak{p}} & \text{if }\mathfrak{q}=\mathfrak{p} \\
 -{ _\mathfrak{p}}\log_\iota(\beta_\mathfrak{q})\cdot 1_{B_\mathfrak{p}} & \text{if }\mathfrak{q}\neq
  \mathfrak{p}.
\end{array} \right.
$$ When $\sigma\notin \Sigma_\mathfrak{p}$,
\begin{equation}\label{eq:trivial-ord} c_{\mathfrak{p},\sigma,\log}=
{_\mathfrak{p}}\log_\sigma(\omega_\mathfrak{p})c_{\mathfrak{p},\ord}.
\end{equation}
Hence $[c_{\mathfrak{p},\ord}]$ and $[c_{\mathfrak{p},\iota,\log}]$
($l\in \LOG_J$) are all in $H^1(\Delta,
\mathcal{C}_\mathfrak{p}^{\leq 1})$.

The $\Delta$-invariant distribution $\widetilde{\mu}_J$ induces an
element $[\widetilde{\mu}_J]$ in $H^0(\Delta,
\mathrm{Dist}^{\emptyset})$. Consider $[c_{\mathfrak{p},\ord}]$ and
$[c_{\mathfrak{p}, l}]$ as elements in $H^1(\Delta,
\mathcal{C}_\mathfrak{p}^{\leq N})$. Then we obtain elements
$[c_{\mathfrak{p},\ord}] \cup [\widetilde{\mu}_J]$ and
$[c_{\mathfrak{p},l}] \cup [\widetilde{\mu}_J]$ in $H^1(\Delta,
\mathcal{D}^{\mathfrak{p},\leq N} )$.


\begin{prop}\label{prop:coh-1} If $\mathfrak{p}\in J_0$ and if $\sigma\in\Sigma_\mathfrak{p}$, then for each $N\geq 1$
we have $$ [c_{\mathfrak{p},\sigma,\log}] \cup [\widetilde{\mu}_J] =
\mathcal{L}^\Tei_{\sigma, u_\mathfrak{p}^{\sigma}}(f)
[c_{\mathfrak{p},\ord}] \cup [\widetilde{\mu}_J] $$ in $H^1(\Delta,
\mathcal{D}^{\mathfrak{p},\leq N} )$.
\end{prop}
\begin{proof} Let $\mu_0$ denote the distribution $$\left(c_{\mathfrak{p},\sigma,\log} (\beta_\mathfrak{p}) - \mathcal{L}^\Tei_{\sigma, u_\mathfrak{p}^{\sigma}}(f)
c_{\mathfrak{p},\ord} (\beta_\mathfrak{p}) \right) \widetilde{\mu}_J
\in \mathcal{D}^{\mathfrak{p},\leq N+1}.$$ By Proposition
\ref{prop:deriv-induction}, $\mu_0$ vanishes on
$(\mathcal{C}^{\mathfrak{p},\leq
N+1})^{\Delta_\mathfrak{p}}=(\bar{C}^{\mathfrak{p}})^{\Delta_\mathfrak{p}}
 $. As
$\beta^*_\mathfrak{p}-1$ is strict, there exists
$\mu_1\in\mathcal{D}^{\mathfrak{p},\leq N+1}$ such that
$$ \langle \mu_1, (\beta_\mathfrak{p}^*-1) \alpha \rangle = \langle \mu_0,
\alpha \rangle$$ for each $\alpha\in \mathcal{C}^{\mathfrak{p},\leq
N+1}$. Equivalently $(\beta_{\mathfrak{p}*}-1)\mu_1= \mu_0$.

Now, for any $\tau\in \Delta$ and $\alpha\in
\mathcal{C}^{\mathfrak{p},\leq N+1}$ we have {\allowdisplaybreaks
\begin{eqnarray*}
&& \langle (\tau_*-1)\mu_1, (\beta_\mathfrak{p}^*-1)\alpha\rangle \\
& = & \langle (\beta_{\mathfrak{p}*}-1)
(\tau_*-1)\mu_1, \alpha \rangle \\
& = & \langle  (\tau_*-1)(\beta_{\mathfrak{p}*}-1)\mu_1, \alpha \rangle \\
&=& \langle (\tau_*-1)(c_{\mathfrak{p},\sigma,\log}
(\beta_\mathfrak{p}) - \mathcal{L}^\Tei_{\sigma,
u_\mathfrak{p}^{\sigma}}(f) c_{\mathfrak{p},\ord}
(\beta_\mathfrak{p}) ) \widetilde{\mu}_J
, \alpha \rangle \\
&=& \langle (\beta_{\mathfrak{p}*}-1)(c_{\mathfrak{p},\sigma,\log}
(\tau) - \mathcal{L}^\Tei_{\sigma, u_\mathfrak{p}^{\sigma}}(f)
c_{\mathfrak{p},\ord} (\tau) ) \widetilde{\mu}_J , \alpha \rangle \\
&=&\langle (c_{\mathfrak{p},\sigma,\log} (\tau) -
\mathcal{L}^\Tei_{\sigma, u_\mathfrak{p}^{\sigma}}(f)
c_{\mathfrak{p},\ord} (\tau) ) \widetilde{\mu}_J ,
(\beta_\mathfrak{p}^*-1)\alpha \rangle.
\end{eqnarray*}} As
$(\beta_\mathfrak{p}^*-1)\mathcal{C}^{\mathfrak{p},\leq
N+1}=\mathcal{C}^{\mathfrak{p},\leq N}$, we obtain
$$ (\tau_*-1)\mu_1 = (c_{\mathfrak{p},\sigma,\log} (\tau) - \mathcal{L}^\Tei_{\sigma,
u_\mathfrak{p}^{\sigma}}(f) c_{\mathfrak{p},\ord} (\tau) )
\widetilde{\mu}_J $$ in $\mathcal{D}^{\mathfrak{p},\leq N}$.
\end{proof}

\begin{defn}\label{defn:L-inv} For $\sigma\in \Sigma_J$ and $\mathfrak{p}\in J_0$, we put
$$ \mathcal{L}_{\mathfrak{p},\sigma  }(f) = \left\{\begin{array}{ll}
\mathcal{L}^\Tei_{\sigma, u_\mathfrak{p}^\sigma}(f) &  \text{ if }\sigma\in \Sigma_\mathfrak{p}, \\
{_\mathfrak{p}}\log_\sigma(\omega_\mathfrak{p}) &  \text{ if
}\sigma\notin \Sigma_\mathfrak{p}.
\end{array}\right.
$$ For $l=\sum\limits_{\sigma\in\Sigma_J}s_\sigma \log_\sigma$, we set
$$\mathcal{L}_{\mathfrak{p}, l }(f)=\sum_{\sigma\in\Sigma_J}s_\sigma
\mathcal{L}_{\mathfrak{p},\sigma}(f).$$
\end{defn}

The following is a direct consequence of Proposition
\ref{prop:coh-1} and (\ref{eq:trivial-ord}).

\begin{cor} For each $\mathfrak{p}\in J_0$, each $l\in\LOG_J$ and each $N\geq 1$ we have $$
[c_{\mathfrak{p},l}] \cup [\widetilde{\mu}_J] = \mathcal{L}_{
\mathfrak{p},l}(f) [c_{\mathfrak{p},\ord}] \cup [\widetilde{\mu}_J]
$$ in $H^1(\Delta, \mathcal{D}^{\mathfrak{p},\leq N} )$.
\end{cor}

\begin{thm}\label{thm:coh} Write $J_0=\{\mathfrak{p}_1, \cdots
\mathfrak{p}_r\}$. Let $l_i$ $(i=1,\cdots, r)$ be in $\LOG_J$. Then
for each $N\geq 1$ we have
$$ [c_{\mathfrak{p}_1,l_1}]\cup \cdots\cup [c_{\mathfrak{p}_r,l_r}] \cup [\widetilde{\mu}_J]
=  \left(\prod_{i=1}^r \mathcal{L}_{ \mathfrak{p}_{i},l_i}(f)\right)
[c_{\mathfrak{p}_1,\ord}] \cup \cdots \cup [c_{\mathfrak{p}_r,\ord}]
\cup [\widetilde{\mu}_J]
$$ in $H^r(\Delta, \mathcal{D}^{J_0, \leq N})$.
\end{thm}
\begin{proof} For $j=1,\cdots, r$ we write $J(j)=\{\mathfrak{p}_1, \cdots \mathfrak{p}_j\}$.
We prove $$ [c_{\mathfrak{p}_{1},l_{1}}]\cup \cdots\cup
[c_{\mathfrak{p}_j, l_j}] \cup [\widetilde{\mu}_J] =
\left(\prod_{i=1}^j \mathcal{L}_{\mathfrak{p}_{i},l_i }(f)\right)
[c_{\mathfrak{p}_{1},\ord}] \cup \cdots \cup
[c_{\mathfrak{p}_j,\ord}] \cup [\widetilde{\mu}_J]
$$ in $H^j(\Delta, \mathcal{D}^{J(j),\leq N})$ by induction on $j$.

When $j=1$,
this is just Proposition \ref{prop:coh-1}. Assume $j\geq 2$ and the assertion
holds when $j$ is replaced by $j-1$. Then we have
$$ [c_{\mathfrak{p}_{1},l_{1}}]\cup \cdots\cup
[c_{\mathfrak{p}_{j-1},l_{j-1}}] \cup [\widetilde{\mu}_J] =
\left(\prod_{i=1}^{j-1} \mathcal{L}_{\mathfrak{p}_i,l_i}(f)\right)
[c_{\mathfrak{p}_{1},\ord}] \cup \cdots \cup
[c_{\mathfrak{p}_{j-1},\ord}] \cup [\widetilde{\mu}_J]
$$ in $H^{j-1}(\Delta, \mathcal{D}^{J(j-1),\leq N+1} )$. So from the pairing
$$ H^1(\Delta, \mathcal{C}_{\mathfrak{p}_j}^{\leq 1})\times H^{j-1}(\Delta, \mathcal{D}^{J(j-1),\leq N+1} )
\rightarrow H^{j}(\Delta, \mathcal{D}^{J(j),\leq N}) $$ we obtain
\begin{eqnarray*}&& [c_{\mathfrak{p}_{1},l_{1}}]\cup \cdots\cup
[c_{\mathfrak{p}_j,l_j}] \cup [\widetilde{\mu}_J] \\ &=&
(-1)^{j-1}[c_{\mathfrak{p}_j,l_j}]\cup\left(
[c_{\mathfrak{p}_{1},l_{1}}]\cup \cdots \cup
[c_{\mathfrak{p}_{j-1},l_{j-1}}]
 \cup [\widetilde{\mu}_J]\right) \\ &=&
\left(\prod_{i=1}^{j-1} \mathcal{L}_{
\mathfrak{p}_{i},l_i}(f)\right)(-1)^{j-1}[c_{\mathfrak{p}_j,l_j}]\cup
\left([c_{\mathfrak{p}_{1},\ord}] \cup \cdots \cup
[c_{\mathfrak{p}_{j-1},\ord}] \cup [\widetilde{\mu}_J]\right) \\
&=& \left(\prod_{i=1}^{j-1} \mathcal{L}_{
\mathfrak{p}_i,l_i}(f)\right) [c_{\mathfrak{p}_{1},\ord}] \cup
\cdots \cup [c_{\mathfrak{p}_{j-1},\ord}]\cup
[c_{\mathfrak{p}_j,l_j}]\cup [\widetilde{\mu}_J]
\end{eqnarray*} in $H^j(\Delta, \mathcal{D}^{J(j),\leq N} )$.

Since
$$[c_{\mathfrak{p}_j,l_j}]\cup [\widetilde{\mu}_J] = \mathcal{L}_{ \mathfrak{p}_j,l_j}(f) [c_{\mathfrak{p}_j,\ord}] \cup
[\widetilde{\mu}_J]$$ in $H^1(\Delta,
\mathcal{D}^{\mathfrak{p}_j,\leq N+j-1})$ which is ensured by
Proposition \ref{prop:coh-1}, from the pairing
$$H^{j-1}(\Delta, \mathcal{C}_{J(j-1)}^{\leq j-1} )\times H^1(\Delta, \mathcal{D}^{\mathfrak{p}_j,\leq N+j-1})\rightarrow H^j(\Delta, \mathcal{D}^{J(j),\leq N})
$$ we obtain the relation
$$[c_{\mathfrak{p}_{1},\ord}] \cup \cdots
\cup [c_{\mathfrak{p}_{j-1},\ord}]\cup [c_{\mathfrak{p}_j,l_j}]\cup
[\widetilde{\mu}_J] = \mathcal{L}_{\mathfrak{p}_j,l_j}(f)
[c_{\mathfrak{p}_{1},\ord}] \cup \cdots \cup
[c_{\mathfrak{p}_{j-1},\ord}] \cup [c_{\mathfrak{p}_{j},\ord}] \cup
[\widetilde{\mu}_J]
$$ in $H^j(\Delta, \mathcal{D}^{J(j),\leq N})$. This finishes the
inductive proof.
\end{proof}

\section{Exceptional zero formula}\label{sec:proof-main}

We fix a disjoint decomposition $J=J'\bigsqcup J''$ of $J$; $J''$ may be empty. Put $J_1=J'\cap J_0$ and $J_2=J''\cap J_0$.
Assume that $r:=\sharp (J_1) >0$.

For each $\sigma\in \Sigma_{J'}$ we take a constant $s_\sigma$. We
take $l$ in Proposition \ref{prop:ell} to be
$$l= \sum_{\sigma\in \Sigma_{J'}} s_\sigma
\log_\sigma.$$ Then for any $\mathfrak{p}\in J'$ we have
$$ \sum_{v\in J\backslash J'' } {_vl}(\beta_{\mathfrak{p}})=\sum_{v\in J\backslash J_2 } {_vl}(\beta_{\mathfrak{p}})=0.
$$

We take a set of representatives $\{a: a_p=1\}$ in
$\BA^{\infty,\times}_K $ of the coset $\BA^{\infty,\times}_F
K^\times \backslash \BA^{\infty,\times}_K/
\widehat{\mathcal{O}}^\times_K$.

\begin{prop} \label{prop:final} Let $\widehat{\chi}$ be a character of $\Gamma^-_J$ such that $\widehat{\chi}_\mathfrak{p}=1$
for every $\mathfrak{p}\in J_1$. Then {\small $$ \int \widehat{\chi}
\cdot \sum_a\left(\sum_{v\in J\backslash J_2} { _vl}\right)^{r}\cdot
1_{Z_{\emptyset,J_2,a}} \widetilde{\mu}_J = r!
\prod_{\mathfrak{p}\in J_1
}\left(h_\mathfrak{p}\mathcal{L}_{\mathfrak{p},l}\right ) \cdot
\mathscr{L}_{J\backslash J_1} (\pi_K,
\widehat{\nu}\widehat{\chi}),$$ }where
$\mathcal{L}_{\mathfrak{p},l}=\mathcal{L}_{\mathfrak{p},l}(f)$ is
defined in Definition $\ref{defn:L-inv}$.
\end{prop}

When $\widehat{\chi}_\mathfrak{p}=1$, $\widehat{\chi}$ factors
through $\Gamma^-_{J\backslash J_1}$, and
$\widehat{\nu}\widehat{\chi}$ lies in the $J\backslash J_1$-branch
of $\widehat{\nu}$. So, the notation $\mathscr{L}_{J\backslash J_1}
(\pi_K, \widehat{\nu}\widehat{\chi})$ makes sense.

\begin{proof}  By continuity we may suppose that
$\widehat{\chi}$ factors through $\BA _F^{\infty,\times}K^\times
\backslash \BA _K^{\infty,\times} /
\widehat{\mathcal{O}}_{J_0;\vec{n}, \mathfrak{c}}^\times $ for some
$J\backslash J_0$-tuple $\vec{n}$. Indeed, among the characters
satisfying our condition, those of finite order are dense.

Write $J_1=\{\mathfrak{p}_1,\cdots, \mathfrak{p}_r\}$ and
$J'=\{\mathfrak{p}_1,\cdots, \mathfrak{p}_d\}$. For our convenience
we identify $J'$ with the index set $\{1,\cdots, d\}$. For each
$i=1,\cdots, d$, put $\ell_i=\sum\limits_{\sigma\in
\Sigma_{\mathfrak{p}_i}}s_\sigma \log_\sigma$. Then
$l=\sum\limits_{i=1}^d\ell_i$.

Since $J_2$ is fixed, we will write $\Lambda_{?, \mathfrak{p}, a}$
for $\Lambda_{?, \mathfrak{p}, J_2,a}$ ($\mathfrak{p}\in J_1$), and
write $\Lambda_{\emptyset, a}$ for $\Lambda_{\emptyset, J_2,a}$. For
the meanings of these notations see Section \ref{ss:comp-II}.

Let $\bar{\Omega}_a$ be the matrix $$ \left(\begin{array}{cccc}
\Lambda_{l,\mathfrak{p}_1,a} -
{_{\mathfrak{p}_1}l(\beta_{\mathfrak{p}_1})}\Lambda_{\emptyset,a} &
-{_{\mathfrak{p}_1}l(\beta_{\mathfrak{p}_2})}\Lambda_{\emptyset,a} &
\cdots &
-{_{\mathfrak{p}_1}l(\beta_{\mathfrak{p}_r})}\Lambda_{\emptyset,a}
\\ -{_{\mathfrak{p}_2}l(\beta_{\mathfrak{p}_1})}\Lambda_{\emptyset,a} & \Lambda_{l,\mathfrak{p}_2,a}
- {_{\mathfrak{p}_2}l(\beta_{\mathfrak{p}_2})}\Lambda_{\emptyset,a} & \cdots & -{_{\mathfrak{p}_2}l(\beta_{\mathfrak{p}_r})}\Lambda_{\emptyset,a} \\
\vdots & \vdots & \ddots & \vdots \\
-{_{\mathfrak{p}_r}l(\beta_{\mathfrak{p}_1})}\Lambda_{\emptyset,a} &
-{_{\mathfrak{p}_r} l(\beta_{\mathfrak{p}_2})} \Lambda_{\emptyset,a}
& \cdots & \Lambda_{ l, \mathfrak{p}_r,a} -{_{\mathfrak{p}_r}
l(\beta_{\mathfrak{p}_r})}\Lambda_{\emptyset,a}
 \end{array} \right) . $$
 If $i\neq j\in\{1,\cdots, r\}$, then
 $$ \bar{\Omega}_{a,ij}= -{_{\mathfrak{p}_i}}l(\beta_{\mathfrak{p}_j}) \Lambda_{\emptyset,a}  =
 - \sum_{\sigma\in \Sigma_{\mathfrak{p}_i}} s_\sigma {_{\mathfrak{p}_i}}\log_\sigma(\beta_{\mathfrak{p}_j})\Lambda_{\emptyset,a}. $$
 For $i\in\{1,\cdots, r\}$ we have
\begin{eqnarray*} \bar{\Omega}_{a,ii} &=& \Lambda_{\ell_i,
\mathfrak{p}_i,a}+\sum_{1\leq i'\leq d, i'\neq i}
\Lambda_{\ell_{i'}, \mathfrak{p}_i,a} - {_{\mathfrak{p}_i}}
l(\beta_{\mathfrak{p}_i})\Lambda_{\emptyset,a} .\end{eqnarray*} Let
$\Omega_a$ be the matrix with
$$\Omega_{a, ii}= \Lambda_{\ell_i, \mathfrak{p}_i,a}$$ and
$$\Omega_{a, ij}=\bar{\Omega}_{a, ij}$$ if $i\neq j$. Let $M_a$ be the
diagonal matrix {\small $$\left(\begin{array}{ccc}
\sum\limits_{1\leq i'\leq d, i'\neq 1} \Lambda_{\ell_{i'},
\mathfrak{p}_1,a}
-{_{\mathfrak{p}_1}}l(\beta_{\mathfrak{p}_1})\Lambda_{\emptyset,a}
&& \\ & \ddots & \\ && \sum\limits_{1\leq i'\leq d, i'\neq r}
\Lambda_{\ell_{i'}, \mathfrak{p}_r,a}
-{_{\mathfrak{p}_r}}l(\beta_{\mathfrak{p}_r})\Lambda_{\emptyset,a}
\end{array} \right).$$}

\noindent Then $\bar{\Omega}_a=\Omega_a+M_a$.

For arbitrary two disjoint subsets $T$ and $S$ of $J_1$ we set
$$ Z^T_{S,J_2,a} =  \prod_{\mathfrak{p}\in S} B_{\mathfrak{p}} \times \prod_{\mathfrak{p}\in J_2} \underline{D}_{\mathfrak{p}}
\times a(\prod_{v\notin S\cup T\cup J_2 } D_v)
$$ and $$B_T=\prod_{\mathfrak{p}\in T} B_{\mathfrak{p}}.$$
Put
\begin{eqnarray*}f^T_{a}:&=& \prod_{i\in J_1\backslash T}
\Big(\sum\limits_{i':1\leq i'\leq d, i'\neq
i}{_{\mathfrak{p}_i}}\ell_{i'} \cdot 1_{Z^T_{J_1\backslash (T\cup
\{\mathfrak{p}_i\}),J_2,a}}-
{_{\mathfrak{p}_i}}l(\beta_{\mathfrak{p}_i})\cdot
1_{Z^T_{J_1\backslash T, J_2,a}}\Big) \\
&=& \sum_{S\subseteq J_1\backslash T}\prod_{i\in J_1\backslash
(T\cup S)}\left(\sum\limits_{i':1\leq i'\leq d, i'\neq i}
{_{\mathfrak{p}_i}}\ell_{i'} \right) \cdot \prod_{i\in
S}{_{\mathfrak{p}_i}}l(\beta_{\mathfrak{p}_i} )\cdot
1_{Z^T_{S,J_2,a}} .\end{eqnarray*} Then $f^T_{a}$ is
$\Delta_T$-invariant, and
\begin{eqnarray*} 1_{B_T}\otimes f^T_{a} &=& \prod_{i\in J_1\backslash T}
\left(\sum\limits_{1\leq i'\leq d, i'\neq i} \Lambda_{\ell_{i'},
\mathfrak{p}_i,a}
-{_{\mathfrak{p}_i}}l(\beta_{\mathfrak{p}_i})\Lambda_{\emptyset,a}
\right)
\end{eqnarray*}

By our condition on $\widehat{\chi}$, considered as a function on
$\BA_F^{\infty,\times}\backslash \BA_K^{\infty,\times}$,
$\widehat{\chi}$ is the pull-back of a function on
$\BA_F^{\infty,J_1,\times}\backslash \BA_K^{\infty,J_1,\times}$. We
write $\widehat{\chi}^{J_1}$  for the latter function. For any
$T\subset J_1$ let $\widehat{\chi}^T$ be the pull-back of
$\widehat{\chi}^{J_1}$ to $\BA_F^{\infty,T,\times}\backslash
\BA_K^{\infty,T,\times} $.

We use $\Omega_{a,T}$ to denote the submatrix of $\Omega_a$
consisting of the $(i,j)$-entries ($i,j\in T$). Then
\begin{eqnarray*} \det (\bar{\Omega}_a) &=& \sum_{T} \det
(\Omega_{a,T}) \cdot \prod_{i\in J_1\backslash T}
\Big(\sum\limits_{1\leq i'\leq d, i'\neq i} \Lambda_{\ell_{i'},
\mathfrak{p}_i,a}
-{_{\mathfrak{p}_i}}l(\beta_{\mathfrak{p}_i})\Lambda_{\emptyset,a}
\Big) \\
&=& \sum_{T} \det (\Omega_{a,T}) \cdot ( 1_{B_T} \otimes f^T_{a})
\end{eqnarray*} with the convention that
$\det\Omega_{a,\emptyset}=1$ when $T=\emptyset$.

When $T=\{i_1, \cdots, i_k\}$ with $i_1<i_2<\cdots<i_k$, we define
$$ c_{\log, T}(g_1, \cdots, g_k) = \sum_{P}(-1)^P c_{\mathfrak{p}_{i_1},\ell_{i_1}}(g_{{p_1}}) \cdots c_{\mathfrak{p}_{i_k},\ell_{i_k}}
(g_{{p_k}}), \hskip 10pt g_1,\cdots, g_k\in \Delta_T,
$$ where in the sum $P$ runs over all permutations $(p_1,\cdots, p_k)$ of $\{1,\cdots,
k\}$, and $$(-1)^P=\left\{\begin{array}{cl} 1 & \text{ if } P \text{ is an even permutation}\\
-1 & \text{ if } P \text{ is an odd permutation}.
\end{array} \right.$$ Then $c_{\log, T}$ is a $k$-cocycle on
$\Delta_T$, and $$[c_{\log,
T}]=[c_{\mathfrak{p}_{i_1},\ell_{i_1}}|_T]\cup\cdots\cup
[c_{\mathfrak{p}_{i_k},\ell_{i_1}}|_T].$$ Here, for a $1$-cocycle
$c$ on $\Delta_{J_0}$, $c|_T$ means the restriction of $c$ to
$\Delta_T$. Note that
$$\det(\Omega_{a,T})= c_{\log, T}(\beta_{\mathfrak{p}_{i_1}}, \cdots, \beta_{\mathfrak{p}_{i_k}})
\otimes 1_{Z^T_{J_1\backslash T,J_2,a}} .$$  As $\widehat{\chi}^T
f^T_a1_{Z^T_{J_1\backslash T,J_2,a}}$ is $\Delta_T$-invariant, we
have
$$ \int \widehat{\chi}\cdot\det (\Omega_{a,T}) \cdot (1_{B_T}\otimes f^T_{a}) \widetilde{\mu}_J =
 ([\widehat{\chi}^Tf^T_{a}1_{Z^T_{J_1\backslash T,J_2,a}}]\cup [c_{\log, T}] \cup [\widetilde{\mu}_J]) (\beta_{\mathfrak{p}_{i_1}}, \cdots, \beta_{\mathfrak{p}_{i_k}}).
$$ By Theorem \ref{thm:coh} and its proof we
obtain \begin{eqnarray*} && \int \widehat{\chi}\cdot \det
(\Omega_{a,T}) \cdot (1_{B_T}\otimes f^T_{a}) \widetilde{\mu}_J
\\ &=& \big( \prod_{i\in
T} \mathcal{L}_{ \mathfrak{p}_i,\ell_i} \big)
([\widehat{\chi}^Tf^T_{a}1_{Z^T_{J_1\backslash T,J_2,a}}]\cup
[c_{\mathfrak{p}_{i_1}, \ord}]\cup \cdots \cup
[c_{\mathfrak{p}_{i_k} ,\ord}] \cup
[\widetilde{\mu}_J])(\beta_{i_1},\cdots, \beta_{i_k}) \\
&=& \big( \prod_{i\in T} \mathcal{L}_{\mathfrak{p}_i,\ell_i}
\big)\cdot \int (\sum_{s_1=1}^{h_{i_1}}\cdots\sum_{s_k=1}^{h_{i_k}}
1_{\omega_{\mathfrak{p}_{i_1}}^{s_1}B_{\mathfrak{p}_{i_1}}}\otimes
\cdots\otimes 1_{\omega_{\mathfrak{p}_{i_k}}^{s_k}B_{\mathfrak{p}_{i_k}}})  \\
&&  \otimes   \sum_{S\subseteq J_1\backslash T}\prod_{i\in
J_1\backslash (T\cup S)}\left(\sum\limits_{i':1\leq i'\leq d, i'\neq
i} {_{\mathfrak{p}_i}}\ell_{i'} \right) \cdot \prod_{i\in
S}\left(-{_{\mathfrak{p}_i}}l(\beta_{\mathfrak{p}_i} )\right)\cdot
\widehat{\chi}^T 1_{Z^T_{S,J_2,a}} \widetilde{\mu}_J.
\end{eqnarray*} Here, we write $h_i$ for $h_{\mathfrak{p}_i}$ ($1\leq i\leq
d$).

When $S\subsetneq J_1\backslash T$, writing $U:=J_1\backslash (S\cup
T)=\{j_1,\cdots, j_u\}$ we have
$$ \prod_{i\in
J_1\backslash (T\cup S)}\left(\sum\limits_{i':1\leq i'\leq d, i'\neq
i} {_{\mathfrak{p}_i}}\ell_{i'} \right) =
\sum_{t_1=0}^{h_{j_1}-1}\cdots \sum_{t_u=0}^{h_{j_u}-1}
c_{t_1,\cdots, t_u}
1_{\omega_{\mathfrak{p}_{j_1}}^{t_1}\mathcal{O}^\times_{\mathfrak{p}_{j_1}}}\otimes
\cdots \otimes
1_{\omega_{\mathfrak{p}_{j_u}}^{t_u}\mathcal{O}^\times_{\mathfrak{p}_{j_u}}},
$$ where $c_{t_1,\cdots, t_u}$ are constants in $\BC_p$ that can be
precisely determined but not important. Take $b\in K^\times$ such
that $$(b)=s_1 \mathfrak{P}_{i_1}+ \cdots + s_k\mathfrak{P}_{i_k} +
t_1 \mathfrak{P}_{j_1}+ \cdots + t_u\mathfrak{P}_{j_u} + \text{sum
of primes not above } p .$$ Let $b^{(p)}$ be the out-$p$-part of
$b$. Then {\allowdisplaybreaks \begin{eqnarray*} && \sum_a   \int
\widehat{\chi} \Big( (
1_{\omega_{\mathfrak{p}_{i_1}}^{s_1}B_{\mathfrak{p}_{i_1}}}\otimes
\cdots\otimes
1_{\omega_{\mathfrak{p}_{i_k}}^{s_k}B_{\mathfrak{p}_{i_k}}}) \otimes
(1_{\omega_{\mathfrak{p}_{j_1}}^{t_1}\mathcal{O}^\times_{\mathfrak{p}_{j_1}}}\otimes
\cdots \otimes
1_{\omega_{\mathfrak{p}_{j_u}}^{t_u}\mathcal{O}^\times_{\mathfrak{p}_{j_u}}})
\otimes  1_{Z^{T\cup U}_{S,J_2,a}} \Big) \widetilde{\mu}_J
\\ &=& \sum_a   \int b^*\left( \widehat{\chi} \Big( (
1_{\omega_{\mathfrak{p}_{i_1}}^{s_1}B_{\mathfrak{p}_{i_1}}}\otimes
\cdots\otimes
1_{\omega_{\mathfrak{p}_{i_k}}^{s_k}B_{\mathfrak{p}_{i_k}}}) \otimes
(1_{\omega_{\mathfrak{p}_{j_1}}^{t_1}\mathcal{O}^\times_{\mathfrak{p}_{j_1}}}\otimes
\cdots \otimes
1_{\omega_{\mathfrak{p}_{j_u}}^{t_u}\mathcal{O}^\times_{\mathfrak{p}_{j_u}}})
\otimes  1_{Z^{T\cup U}_{S,J_2,a}} \Big)\right) \widetilde{\mu}_J
\\&=& \sum_a   \int  \widehat{\chi} \Big( (
1_{ B_{\mathfrak{p}_{i_1}}}\otimes \cdots\otimes 1_{
B_{\mathfrak{p}_{i_k}}}) \otimes (1_{
\mathcal{O}^\times_{\mathfrak{p}_{j_1}}}\otimes \cdots \otimes 1_{
\mathcal{O}^\times_{\mathfrak{p}_{j_u}}}) \otimes  1_{Z^{T\cup
U}_{S,J_2,a/b^{(p)}}}\Big) \widetilde{\mu}_J \\
&=& \sum_a \int \widehat{\chi} \cdot 1_{S\cup T, J_2\cup
U,a/b^{(p)}} \widetilde{\mu}_J \hskip 5pt = \hskip 5pt
\mathscr{L}_{J\backslash (S\cup T)}(\widehat{\nu}\widehat{\chi})
.\end{eqnarray*} }

\noindent By our condition on $\widehat{\chi}$  we have
$\widehat{\chi}_\mathfrak{p}=1$ for any $\mathfrak{p}\in
J_1\backslash (S\cup T)$. So, by Corollary \ref{cor:vanish-value} we
have $\mathscr{L}_{J\backslash (S\cup
T)}(\widehat{\nu}\widehat{\chi})=0$. Hence,
\begin{eqnarray*}&& \big( \prod_{i\in T} \mathcal{L}_{
\mathfrak{p}_i,\ell_i} \big)\cdot \int
(\sum_{s_1=1}^{h_{i_1}}\cdots\sum_{s_k=1}^{h_{i_k}}
1_{\omega_{\mathfrak{p}_{i_1}}^{s_1}B_{\mathfrak{p}_{i_1}}}\otimes
\cdots\otimes 1_{\omega_{\mathfrak{p}_{i_k}}^{s_k}B_{\mathfrak{p}_{i_k}}})  \\
&&  \otimes   \sum_{S\subsetneq J_1\backslash T}\prod_{i\in
J_1\backslash (T\cup S)}\left(\sum\limits_{i':1\leq i'\leq d, i'\neq
i} {_{\mathfrak{p}_i}}\ell_{i'} \right) \cdot \prod_{i\in
S}(-{_{\mathfrak{p}_i}}l(\beta_{\mathfrak{p}_i} ))\cdot
\widehat{\chi}^T 1_{Z^T_{S,J_2,a}} \widetilde{\mu}_J
=0.\end{eqnarray*}

When $S= J_1\backslash T$, we have
\begin{eqnarray*}&& \sum_a\int
 (1_{\omega_{\mathfrak{p}_{i_1}}^{s_1}B_{\mathfrak{p}_{i_1}}}\otimes
\cdots 1_{\omega_{\mathfrak{p}_{i_k}}^{s_k}B_{\mathfrak{p}_{i_k}}})
\otimes \widehat{\chi}^T 1_{Z^T_{J_1\backslash T,J_2,a}}
\widetilde{\mu}_J
\\ &=& \sum_a \int \widehat{\chi} 1_{Z_{J_1, J_2,a}}
\widetilde{\mu}_J \hskip 5pt = \hskip 5pt \mathscr{L}_{J\backslash
J_1}(\pi,\widehat{\nu}\widehat{\chi}) .
\end{eqnarray*}

Therefore
\begin{eqnarray*} & & \sum_a \int \widehat{\chi} \cdot \det (\Omega_{a,T}) \cdot (1_{B_T}\otimes f^T_{a} )
\widetilde{\mu}_J \\
&=& (\prod_{i=1}^r h_{\mathfrak{p}_i})\cdot
  \big( \prod_{i\in T} \mathcal{L}_{
\mathfrak{p}_i,\ell_i} \big)\cdot\Big(\prod_{i\in J_1\backslash T}
\sum\limits_{i':1\leq i'\leq d, i'\neq i}
{_{\mathfrak{p}_i}}\ell_{i'}(\omega_{\mathfrak{p}_i})\Big)\cdot
\mathscr{L}_{J\backslash J_1}(\pi,\widehat{\nu}\widehat{\chi}).
\end{eqnarray*}

Now, applying Proposition \ref{prop:ell} we obtain
{\allowdisplaybreaks
\begin{eqnarray*}
&&\sum_a\int \widehat{\chi} \cdot\left(\sum_{v\in J\backslash J_2} {
_vl}\right)^{r}\cdot 1_{Z_{\emptyset,J_2,a}}
\widetilde{\mu}_J   \\
&=&r!\sum_a\int \widehat{\chi} \cdot  \det(\bar{\Omega}_a)
\widetilde{\mu}_J  \\
&=& r! \sum_{T\subset J_1} \sum_a\int \widehat{\chi}\cdot \det
(\Omega_{a,T}) \cdot (1_{B_T}\otimes f^T_{a} ) \widetilde{\mu}_J
\\ &=&
r! \cdot (\prod_{i=1}^r h_{\mathfrak{p}_i})\cdot \sum_{T\subset J_1}
\big( \prod_{i\in T} \mathcal{L}_{ \mathfrak{p}_i, \ell_i }
\big)\cdot\Big(\prod_{i\in J_1\backslash T} \sum\limits_{i':1\leq
i'\leq d, i'\neq i}
{_{\mathfrak{p}_i}}\ell_{i'}(\omega_{\mathfrak{p}_i})\Big)\cdot
\mathscr{L}_{J\backslash J_1}(\pi,\widehat{\nu},\widehat{\chi})
\\
&=& r! \cdot (\prod_{i=1}^r h_{\mathfrak{p}_i}) \cdot \prod_{i=1}^r
\left(\mathcal{L}_{ \mathfrak{p}_i, \ell_i}+ \sum\limits_{i':1\leq
i'\leq d,i'\neq i}
{_{\mathfrak{p}_i}}\ell_{i'}(\omega_{\mathfrak{p}_i})\right) \cdot
\mathscr{L}_{J\backslash J_1}(\pi,\widehat{\nu} \widehat{\chi}),
\end{eqnarray*}}

\noindent as expected.
\end{proof}

With $(s_\sigma)_{\sigma\in \Sigma_J}$ fixed, we put
$\mathscr{L}_J(t;\widehat{\nu}\widehat{\chi}, (s_\sigma)_\sigma)=
\mathscr{L}_J((ts_\sigma)_\sigma,\pi_K, \widehat{\nu}\widehat{\chi}
)$.

\begin{thm} \label{thm:final} We have \begin{eqnarray*}
&& \frac{d^k}{dt^k}\mathscr{L}_J(t; \widehat{\nu}\widehat{\chi}, (s_\sigma)_\sigma)|_{t=0}\\
&=&\left\{\begin{array}{ll}0 & \text{ if } 0\leq k< r , \\
r!\prod\limits_{\mathfrak{p}\in J_1 }\left( \sum\limits_{\sigma\in
\Sigma_J  \backslash \Sigma_\mathfrak{p}} s_\sigma
{_\mathfrak{p}}\ell_\sigma(\omega_{\mathfrak{p}})
+\sum\limits_{\sigma\in \Sigma_\mathfrak{p}}s_\sigma
\mathcal{L}^\Tei_{\sigma, u_\mathfrak{p}^\sigma} \right ) \cdot
\mathscr{L}_{J\backslash J_1} (\pi_K,\widehat{\nu} \widehat{\chi}) &
\text{ if } k=r.
\end{array}\right.\end{eqnarray*}
\end{thm}
\begin{proof} We have
$$ \frac{d^k}{dt^k}\mathscr{L}_J(t;\widehat{\nu}\widehat{\chi},
(s_\sigma)_\sigma)|_{t=0} = \int \widehat{\chi}\cdot l^k \mu_J =
\frac{1}{\prod_{\mathfrak{p}\in J_1}h_\mathfrak{p}}\int \sum_a
\widehat{\chi}\cdot \left(\sum_{v} { _vl}\right)^{k}\cdot
1_{Z_{\emptyset,J_2,a}} \widetilde{\mu}_J,
$$ where $v$ runs through all finite places of $F$. When $k<r$, by Lemma \ref{lem:int-0} this is equal to
$0$. When $k=r$, again using Lemma \ref{lem:int-0} we obtain
$$ \frac{1}{\prod_{\mathfrak{p}\in J_1}h_\mathfrak{p}}\int
\sum_a \widehat{\chi}\cdot \left(\sum_{v} { _vl}\right)^{r}\cdot
1_{Z_{\emptyset,J_2,a}}\ \widetilde{\mu}_J =
\frac{1}{\prod_{\mathfrak{p}\in J_1}h_\mathfrak{p}}\int
\sum_a\widehat{\chi}\cdot\left(\sum_{v\in J} { _vl}\right)^{r}\cdot
1_{Z_{\emptyset,J_2,a}} \ \widetilde{\mu}_J. $$ But
$$\left(\sum_{v\in J} { _vl}\right)^{r}\cdot 1_{Z_{\emptyset,J_2,a}}=
\left(\sum_{v\in J\backslash J_2} { _vl}\right)^{r}\cdot
1_{Z_{\emptyset,J_2,a}}.$$ By Proposition \ref{prop:final} we obtain
the assertion for $k=r$.
\end{proof}

Allowing $s_\sigma$ varying, we obtain the following

\begin{thm}\label{thm:auxi}
Let $\widehat{\chi}$ be a character of $\Gamma^-_J$ such that
$\widehat{\chi}_\mathfrak{p}=1$ for every $\mathfrak{p}\in J_1$. If
$\vec{\sigma}=(\sigma_1, \cdots, \sigma_r)$ is a tuple of elements
in $\Sigma_{J'}$, then
$$ \left.\frac{\partial^r}{\partial s_{\sigma_1}\cdots \partial s_{\sigma_r}} \mathscr{L}_J(\vec{s},\pi_K,\widehat{\nu}\widehat{\chi}
)\right|_{\vec{s}=\vec{0}} =  \Big(\sum_{P=(j_1,\cdots, j_r)}
\prod_{i=1}^r \mathcal{L}_{\mathfrak{p}_{j_i},\sigma_i}\Big)\cdot
\mathscr{L}_{J\backslash J_1}(\pi_K,\widehat{\nu} \widehat{\chi}),$$
where in the sum $P$ runs over all permutations $(j_1, \cdots, j_r)$
of $\{1,\cdots, r\}$.
\end{thm}
\begin{proof}
By Theorem \ref{thm:final} we have
\begin{equation} \label{eq:auxi}
\mathscr{L}_J(\vec{s},\pi_K,\widehat{\nu}\widehat{\chi}) =
\mathscr{L}_{J\backslash J_1} (\pi_K,\widehat{\nu} \widehat{\chi})
\prod\limits_{\mathfrak{p}\in J_1 }\left( \sum\limits_{\sigma\in
 \Sigma_J } s_\sigma
 \mathcal{L}_{\mathfrak{p},\sigma} \right )
  +
 \text{ higher order terms}.  \end{equation}
For every vector $\vec{n}=(n_\tau)_{\tau\in\Sigma_{J'}}$ of nonnegative integers with $\sum_{\tau}n_\tau=r$,
the coefficient of $\prod_{\tau\in \Sigma_{J'}}  s_{\tau}^{n_\tau}$ is
$$\mathscr{L}_{J\backslash J_1} (\pi_K,\widehat{\nu} \widehat{\chi}) \sum_{f}\prod_{i=1}^r\mathcal{L}_{\mathfrak{p}_i,{f(i)}},$$
where $f$ runs over all maps $f:\{1,\cdots, r\}\rightarrow
\Sigma_{J'}$ such that $\sharp f^{-1}(\tau)=n_\tau$ for each
$\tau\in \Sigma_{J'}$. Let $f_0$ be such a map, and write
$\sigma_i=f_0(i)$. Then $$\mathscr{L}_{J\backslash J_1}
(\pi_K,\widehat{\nu} \widehat{\chi})
\sum_{f}\prod_{i=1}^r\mathcal{L}_{\mathfrak{p}_i,{f(i)}}
=\frac{1}{\vec{n}!}\mathscr{L}_{J\backslash J_1}
(\pi_K,\widehat{\nu} \widehat{\chi})\sum_{P=(j_1,\cdots,
j_r)}\prod_{i=1}^r\mathcal{L}_{\mathfrak{p}_i,\sigma_{j_i}},  $$
where $P$ runs over all permutations $(j_1, \cdots, j_r)$ of
$\{1,\cdots, r\}$. Hence,
\begin{eqnarray*}\frac{\partial^r}{\partial s_{\sigma_1}\cdots \partial s_{\sigma_r}}
\mathscr{L}_J(\vec{s},\pi_K,\widehat{\nu}\widehat{\chi})\Big|_{\vec{s}=\vec{0}}
 &=& \frac{\partial^r}{\prod_{\tau\in\Sigma_{J'}}\partial
s_{\tau}^{n_\tau}}\mathscr{L}_J(\pi_K,\widehat{\nu}\widehat{\chi},
\vec{s}) \\ & = & \vec{n}!\cdot\frac{1}{\vec{n}!}\cdot
\mathscr{L}_{J\backslash J_1} (\pi_K,\widehat{\nu}
\widehat{\chi})\sum_{P=(j_1,\cdots,
j_r)}\prod_{i=1}^r\mathcal{L}_{\mathfrak{p}_i,\sigma_{j_i}}
\\ &=& \mathscr{L}_{J\backslash J_1} (\pi_K,\widehat{\nu}
\widehat{\chi})\sum_{P=(j_1,\cdots,
j_r)}\prod_{i=1}^r\mathcal{L}_{\mathfrak{p}_i,\sigma_{j_i}}\\  &=&
\mathscr{L}_{J\backslash J_1} (\pi_K,\widehat{\nu}
\widehat{\chi})\sum_{P=(j_1,\cdots,
j_r)}\prod_{i=1}^r\mathcal{L}_{\mathfrak{p}_{j_i},\sigma_{i}} ,
\end{eqnarray*} as desired.
\end{proof}

As last, we prove Theorems \ref{thm:A}, \ref{thm:B}, \ref{thm:C} and
\ref{thm:D}. Theorem \ref{thm:D} follows immediately from
(\ref{eq:auxi}).

We fix notations. Take $J_1$ to be $J_p$, and let $J_1$ and $J_2$ be
as in the introduction. We take the CM type $\Sigma_K$ to be set of
the places above the infinite place of $\mathcal{K}$ corresponding
to a fixed embedding $\mathcal{K}\hookrightarrow \BC$. Write
$\mathcal{K}\cap \jmath^{-1}\Omega_p=\mathscr{P}$ and
$p\mathcal{O}_\mathcal{K}=\mathscr{P}\bar{\mathscr{P}}$. Then in the
decomposition
$\mathfrak{p}\mathcal{O}_K=\mathfrak{P}\bar{\mathfrak{P}}$ for each
prime $\mathfrak{p}$ above $p$, $\mathfrak{P}$ is above
$\mathscr{P}$, and $\bar{\mathfrak{P}}$ is above
$\bar{\mathscr{P}}$. We take $\mathfrak{c}$ in Section
\ref{ss:anti-ext-char} to be $\mathcal{O}_F$. Take $\mathbf{m}$ to
be the zero vector $\mathbf{0}$, and take $\nu$ to be the trivial
character so that $\widehat{\nu}$ is also trivial.

Let $\pi$ be the automorphic representation of
$\mathrm{GL}_2(\BA_F)$ associated to $\mathbf{f}$. The local
conditions on $\mathbf{f}$ ensure that there exists a definite
quaternion algebra $B$ over $F$ whose discriminant $\mathbf{n}^-_b$
is the product of the primes that divide $\mathfrak{n}^{-}$ and are
inert in $K$, and an automorphic representation on $B$ attached to
$\pi$ via Jacquet-Langlands correspondence. So we can carry out the
constructions in Sections \ref{sec:special-value} and
\ref{sec:p-L-function}. We should note that $\pi$, $\nu=1$ and $B$
satisfy the local assumptions made after Proposition
\ref{prop:cal-P}.

Let $L_{J_p}(\pi_K, \widehat{\chi})$ be the $p$-adic $L$-function
constructed in Section \ref{sec:p-L-function}. Put
$\Omega^-_{\mathbf{f}}=\Omega^-_{J_p,\phi}$ and
$$L_p(\pi_K, \widehat{\chi})=
\widehat{\chi}(\mathfrak{N}^+)^{-1}\cdot
(\prod_{v}\widehat{\chi}_v(\omega_v))^{-1}\cdot L_{J_p}(\pi_K,
\widehat{\chi})$$ where $v$ runs over the set $(\ref{eq:set})$.
Restricting $L_p(\pi_K, \widehat{\chi})$ to (\ref{eq:cont-char}) we
obtain the $p$-adic $L$-function demanded in Theorem \ref{thm:A}.

We apply Theorem \ref{thm:final} to prove Theorem \ref{thm:B} and
Theorem \ref{thm:C}.

Put
$$\mathcal{L}^\Tei_\mathfrak{p}(\mathbf{f})=\sum_{\iota \in \Sigma_\mathfrak{p}}\mathcal{L}^\Tei_{\iota,
u_\mathfrak{p}^\iota}(\mathbf{f})+\sum\limits_{\sigma\in
\Sigma_{J_p} \backslash \Sigma_\mathfrak{p}}
{_\mathfrak{p}}\ell_\sigma(\omega_{\mathfrak{p}}).$$

Let $\beta_\mathcal{K}$ be an element of $\mathcal{K}^\times$ such
that $\beta_\mathcal{K}$ is a unit at all finite places outside of
$\mathscr{P}$. Let $u_\mathcal{K}$ be the element $
\frac{\beta_{\mathcal{K},\mathscr{P}}}{\beta_{\mathcal{K},\bar{\mathscr{P}}}}$
of $\BQ_p$. Consider the infinitesimal of $\langle \ \cdot \
\rangle_{\mathrm{anti}}^s$ as an character of $\mathcal{K}^\times
\backslash \BA_\mathcal{K}^\times$. Its restriction to
$\mathcal{K}^\times_p=\mathcal{K}^\times_\mathscr{P}\times
\mathcal{K}^\times_{\bar{\mathscr{P}}}$ becomes $(x,y)\mapsto
\log_{u_\mathcal{K}}(\frac{x}{y})$.  Note that $
\epsilon^{(s,\cdots,s)}$ is the restriction of $\langle \ \cdot \
\rangle_{\mathrm{anti}}^s$. So $\sum\limits_{\sigma\in
\Sigma_{J_p}}\ell_\sigma$ is the infinitesimal of $\langle \ \cdot \
\rangle_{\mathrm{anti}}^s$. Hence, $$\sum\limits_{\sigma\in
\Sigma_{p} \backslash \Sigma_\mathfrak{p}}
{_\mathfrak{p}}\ell_\sigma(\omega_{\mathfrak{p}}) =
\frac{1}{\ord_\mathfrak{p}(u_\mathfrak{p})} \sum_{\sigma\in
\Sigma_{J_p}} {_\mathfrak{p}}\ell_\sigma(u_{\mathfrak{p}})  =
\frac{1}{\ord_\mathfrak{p}(u_\mathfrak{p})}
\log_{u_\mathcal{K}}(N_{F/\BQ}(u_\mathfrak{p}))=\frac{1}{\ord_\mathfrak{p}(u_\mathfrak{p})}\sum_{\iota\in\Sigma_\mathfrak{p}}\log_{u_\mathcal{K}}
(u_\mathfrak{p}^\iota).$$ Observe that
$$\frac{1}{\ord_\mathfrak{p}(u_\mathfrak{p})}\log_{u_\mathcal{K}}
(u_\mathfrak{p}^\iota)=\log_{u_\mathcal{K}}(\omega_\mathfrak{p})-\log_{u_\mathfrak{p}^\iota}(\omega_\mathfrak{p}).$$
By Proposition \ref{prop:diff-L} we have
$$ \mathcal{L}^\Tei_\mathfrak{p}(\mathbf{f})=\sum_{\iota
\in \Sigma_\mathfrak{p}}\mathcal{L}^\Tei_{\iota,
u_\mathcal{K}}(\mathbf{f}).$$

Theorem \ref{thm:final} tells us
\begin{eqnarray*} \frac{d^n L_p(s,\mathbf{f}/K)}{d s^n}\Big{|}_{s=0}
= \left\{ \begin{array}{ll} 0 & \text{ if } n< 2r , \\ (r!)^2
\left(\prod\limits_{\mathfrak{p}\in J_1 }
\mathcal{L}_{\mathfrak{p}}^\Tei(\mathbf{f})
 \right )^2 \cdot
L_{J_2} ((0,\cdots,0),\pi_K) & \text{ if }n=2r,
\end{array} \right.
\end{eqnarray*}
and
$$ L_{J_p}((s_\sigma)_{\sigma\in \Sigma_{J_p}},\pi_K) = L_{J_2} ((0,\cdots 0),\pi_K)
\cdot \prod\limits_{\mathfrak{p}\in J_1 }\left(
\sum\limits_{\sigma\in
 \Sigma_{J_p} } s_\sigma
 \mathcal{L}_{\mathfrak{p},\sigma} \right )^2
  + \text{ higher order terms}.   $$

By the interpolation formula (i.e. Theorem \ref{thm:interpol}) for
$L_{J_2} (\cdot, \pi_K)$ we have
$$L_{J_2} ((0,\cdots, 0),\pi_K)= \prod_{\mathfrak{p}\in J_2}
e_{\mathfrak{p}}(1) \cdot \frac{L(\frac{1}{2},
\pi_K)}{\Omega^-_{J_2,\phi}} = \prod_{\mathfrak{p}\in J_2}
e_{\mathfrak{p}}(1) \cdot \frac{L(1,
\mathbf{f}/K)}{\Omega^-_{J_2,\phi}}
$$ We show $\Omega^-_{J_2,\phi}=\Omega^-_{J_p,\phi}$ to complete our
proof. By the formula on $\Omega^-_{J,\phi}$ given in Section
\ref{sec:p-L-function} we obtain
$$ \Omega^-_{J_2,\phi}=\Omega^-_{J_p,\phi}\cdot \prod_{\mathfrak{p}\in J_1}\epsilon(\frac{1}{2}, \pi_\mathfrak{p},\psi_\mathfrak{p}) .$$
The root number $\epsilon(\frac{1}{2},
\pi_\mathfrak{p},\psi_\mathfrak{p})$ is computed by
\cite[Proposition 3.6]{JL70} which says that
$$\epsilon(\frac{1}{2}, \pi_\mathfrak{p},\psi_\mathfrak{p})=
\epsilon(\frac{1}{2},
\mu_\mathfrak{p},\psi_\mathfrak{p})\epsilon(\frac{1}{2},
\mu_\mathfrak{p}|\cdot|_\mathfrak{p}^{-1},\psi_\mathfrak{p})$$ if
$\pi_\mathfrak{p}=\sigma(\mu_\mathfrak{p},
\mu_\mathfrak{p}|\cdot|_\mathfrak{p}^{-1})$. Since
$\mu_\mathfrak{p}^2|\cdot|_\mathfrak{p}^{-1}=1$, it follows from
Tate's local functional equation (see \cite[Proposition 3.1.5,
3.1.9]{Bump}) that $\epsilon(\frac{1}{2},
\mu_\mathfrak{p},\psi_\mathfrak{p})\epsilon(\frac{1}{2},
\mu_\mathfrak{p}|\cdot|_\mathfrak{p}^{-1},\psi_\mathfrak{p})=1$.

\end{document}